\documentclass[oneside,english]{amsart}
\usepackage[T1]{fontenc}
\usepackage[latin9]{inputenc}
\usepackage{color}
\usepackage{babel}
\usepackage{float}
\usepackage{wrapfig}
\usepackage{units}
\usepackage{multicol}
\usepackage{amstext}
\usepackage{amsthm}
\usepackage{amssymb}
\usepackage{graphicx}
\usepackage{esint}
\usepackage{xargs}[2008/03/08]
\usepackage[unicode=true,pdfusetitle,
 bookmarks=true,bookmarksnumbered=false,bookmarksopen=false,
 breaklinks=false,pdfborder={0 0 1},backref=false,colorlinks=true]
 {hyperref}

\makeatletter

\newcommand{\noun}[1]{\textsc{#1}}

\numberwithin{equation}{section}
\numberwithin{figure}{section}
\theoremstyle{definition}
\newtheorem*{problem*}{\protect\problemname}
\newcommand\thmsname{\protect\theoremname}
\newcommand\nm@thmtype{theorem}
\theoremstyle{plain}

\newenvironment{namedthm}[1][Undefined Theorem Name]{
  \ifx{#1}{Undefined Theorem Name}\renewcommand\nm@thmtype{theorem*}
  \else\renewcommand\thmsname{#1}\renewcommand\nm@thmtype{namedtheorem}
  \fi
  \begin{\nm@thmtype}}
  {\end{\nm@thmtype}}
\theoremstyle{plain}
\newtheorem{thm}{\protect\theoremname}[section]
\theoremstyle{remark}
\newtheorem{rem}[thm]{\protect\remarkname}
\theoremstyle{definition}
\newtheorem{example}[thm]{\protect\examplename}
\newenvironment{lyxlist}[1]
	{\begin{list}{}
		{\settowidth{\labelwidth}{#1}
		 \setlength{\leftmargin}{\labelwidth}
		 \addtolength{\leftmargin}{\labelsep}
		 }}
	{\end{list}}
\theoremstyle{definition}
\newtheorem{defn}[thm]{\protect\definitionname}
\theoremstyle{plain}
\newtheorem{lem}[thm]{\protect\lemmaname}
\theoremstyle{plain}
\newtheorem{prop}[thm]{\protect\propositionname}
\theoremstyle{plain}
\newtheorem{cor}[thm]{\protect\corollaryname}

\makeatother

\providecommand{\corollaryname}{Corollary}
\providecommand{\definitionname}{Definition}
\providecommand{\examplename}{Example}
\providecommand{\lemmaname}{Lemma}
\providecommand{\problemname}{Problem}
\providecommand{\propositionname}{Proposition}
\providecommand{\remarkname}{Remark}
\providecommand{\theoremname}{Theorem}

\begin{document}


\global\long\def\tx#1{\mathrm{#1}}%
\global\long\def\dd#1{\tx d#1}%
\global\long\def\tt#1{\mathtt{#1}}%
\global\long\def\ww#1{\mathbb{#1}}%
\global\long\def\DD#1{\tx D#1}%
\global\long\def\nf#1#2{\nicefrac{#1}{#2}}%
\global\long\def\group#1{{#1}}%

\newcommand{\bigslant}[2]{{\raisebox{.3em}{$#1$}/\raisebox{-.3em}{$#2$}}}


\global\long\def\quot#1#2{\bigslant{#1}{#2}}%

\global\long\def\rr{\mathbb{R}}%
\global\long\def\rbar{\overline{\mathbb{R}}}%
\global\long\def\cc{\mathbb{C}}%
\global\long\def\cbar{\overline{\cc}}%
\global\long\def\disc{\mathbb{D}}%
\global\long\def\dbar{\overline{\disc}}%
\global\long\def\zz{\mathbb{Z}}%
\global\long\def\zp{\mathbb{Z}_{\geq0}}%
\newcommandx\zsk[1][usedefault, addprefix=\global, 1=k]{\nicefrac{\zz}{#1\zz}}%
\global\long\def\nn{\mathbb{N}}%
\global\long\def\nbar{\overline{\nn}}%
\global\long\def\qq{\mathbb{Q}}%
\global\long\def\qbar{\overline{\qq}}%

\global\long\def\rat#1{\cc\left(#1\right) }%
\global\long\def\pol#1{\cc\left[#1\right] }%
\global\long\def\id{\tx{Id}}%

\global\long\def\GL#1#2{\tx{GL}_{#1}\left(#2\right)}%

\global\long\def\spec#1{\tx{Spec}\left(#1\right)}%


\global\long\def\fol#1{\mathcal{F}_{#1}}%
\global\long\def\pp#1{\frac{\partial}{\partial#1}}%
\global\long\def\ppp#1#2{\frac{\partial#1}{\partial#2}}%
\newcommandx\sat[2][usedefault, addprefix=\global, 1=\fol{}]{\tx{Sat}_{#1}\left(#2\right)}%
\global\long\def\lif#1{\mathcal{L}_{#1}}%
\newcommandx\holo[1][usedefault, addprefix=\global, 1=\gamma]{\mathfrak{h}_{#1}}%
\newcommandx\mono[1][usedefault, addprefix=\global, 1=\chi]{\mathfrak{m}_{#1}}%
\global\long\def\sing#1{\tx{Sing}\left(#1\right)}%
\global\long\def\flow#1#2#3{\Phi_{#1}^{#2}#3}%
\global\long\def\ddd#1#2{\frac{\dd{#1}}{\dd{#2}}}%
\newcommandx\vfs[1][usedefault, addprefix=\global, 1={\cc^{2},0}]{\frak{X}\left(#1\right)}%
\newcommandx\per[2][usedefault, addprefix=\global, 1=, 2=\bullet]{\mathfrak{T}_{#2}^{#1}}%
\newcommandx\cs[2][usedefault, addprefix=\global, 1={\fol{},p_{*}}, 2=\bullet]{\mathfrak{cs}\left(#1\right)}%


\newcommandx\spt[1][usedefault, addprefix=\global, 1=\gamma]{#1_{\star}}%
\newcommandx\ept[1][usedefault, addprefix=\global, 1=\gamma]{#1^{\star}}%


\global\long\def\ii{\tx i}%
\global\long\def\ee{\tx e}%

\global\long\def\re#1{\Re\left(#1\right)}%
\global\long\def\im#1{\Im\left(#1\right)}%

\global\long\def\sgn#1{\mbox{sign}\left(#1\right)}%
\global\long\def\floor#1{\left\lfloor #1\right\rfloor }%
\global\long\def\ceiling#1{\left\lceil #1\right\rceil }%


\global\long\def\germ#1{\cc\left\{  #1\right\}  }%
\global\long\def\frml#1{\cc\left[\left[#1\right]\right] }%


\newcommandx\norm[2][usedefault, addprefix=\global, 1=\bullet]{\left|\left|#1\right|\right|_{#2}}%
\global\long\def\adh#1{\overline{#1}}%
\global\long\def\intr#1{\tx{int}\left(#1\right)}%
\global\long\def\ccnx#1{\tx{cc}\left(#1\right)}%
\newcommandx\neigh[2][usedefault, addprefix=\global, 1=, 2=0]{\left(\cc^{#1},#2\right)}%


\newcommandx\proj[2][usedefault, addprefix=\global, 1=1, 2=\cc]{\mathbb{P}_{#1}\left(#2\right)}%
\global\long\def\sone{\mathbb{S}^{1}}%


\global\long\def\inj{\hookrightarrow}%
\global\long\def\surj{\twoheadrightarrow}%
\global\long\def\longinj#1#2{\xymatrix{#1\ar@{^{(}->}[r]  &  #2}
 }%
\global\long\def\longsurj#1#2{\xymatrix{#1\ar@{->>}[r]  &  #2}
 }%
\global\long\def\longto{\longrightarrow}%
\global\long\def\longmaps{\longmapsto}%
\global\long\def\cst{\tx{cst}}%
\global\long\def\restr#1{\big|_{#1}}%
\global\long\def\Beta#1#2{\tx B\left(#1,#2\right)}%
\global\long\def\Gama#1{\tx{\Gamma}\left(#1\right)}%
\global\long\def\oo#1{\tx o\left(#1\right)}%
\global\long\def\OO#1{\tx O\left(#1\right)}%
\newcommandx\diff[1][usedefault, addprefix=\global, 1={\cc^{2},0}]{\tx{Diff}\left(#1\right)}%
\newcommandx\holf[1][usedefault, addprefix=\global, 1={\cc^{2},0}]{\tx{Holo}\left(#1\right)}%
\newcommandx\holb[1][usedefault, addprefix=\global, 1={\cc^{2},0}]{\tx{Holo_{c}}\left(#1\right)}%
\newcommandx\fdiff[2][usedefault, addprefix=\global, 1=2, 2=0]{\widehat{\tx{Diff}}\left(\cc^{#1},#2\right)}%
\newcommandx\homeo[1][usedefault, addprefix=\global, 1={\cc^{2},0}]{\tx{Homeo}\left(#1\right)}%
\newcommandx\mero[1][usedefault, addprefix=\global, 1={\cc,0}]{\tx{Mero}\left(#1\right) }%
\global\long\def\hom#1{\tx{Hom}\left(#1\right) }%
\global\long\def\ev{\tx{B\acute{E}V}}%
\newcommandx\parab[1][usedefault, addprefix=\global, 1=1]{\tx{Parab}_{#1}}%
\newcommandx\isect[2][usedefault, addprefix=\global, 1=V, 2=\sharp]{#1^{#2}}%
\newcommandx\zero[1][usedefault, addprefix=\global, 1=V]{\isect[#1][0]}%
\newcommandx\infi[1][usedefault, addprefix=\global, 1=V]{\isect[#1][\infty]}%
\newcommandx\adapt[2][usedefault, addprefix=\global, 1=\lambda, 2=\varphi]{\tx{Adapt}_{#1}\left(#2\right)}%
\newcommandx\cauchein[1][usedefault, addprefix=\global, 1=\varphi]{\text{CH}^{#1}}%
\global\long\def\act{\sigma{}^{\circledast}}%
\global\long\def\zset#1{\text{Zero}\left(#1\right)}%
\global\long\def\pset#1{\text{Pole}\left(#1\right)}%
\global\long\def\spine#1{\text{Spine}\left(#1\right)}%
\global\long\def\sep#1{\text{Sep}\left(#1\right)}%

\title{Spherical normal forms for germs of parabolic line biholomorphisms}
\author{Loïc\noun{ Teyssier}}
\date{September 2020}
\address{Laboratoire IRMA (UMR 7501)\\
Université de Strasbourg --- CNRS\\
\noun{France}}
\email{\href{mailto:teyssier@math.unistra.fr}{teyssier@math.unistra.fr}}
\begin{abstract}
We address the inverse problem for holomorphic germs of a tangent-to-identity
mapping of the complex line near a fixed point. We provide a preferred
(family of) parabolic map $\Delta$ realizing a given Birkhoff--Écalle-Voronin
modulus $\psi$ and prove its uniqueness in the functional class we
introduce. The germ is the time-$1$ map of a Gevrey formal vector
field admitting meromorphic sums on a pair of infinite sectors covering
the Riemann sphere. For that reason, the analytic continuation of
$\Delta$ is a multivalued map admitting finitely many branch points
with finite monodromy. In particular $\Delta$ is holomorphic and
injective on an open slit sphere containing $0$ (the initial fixed
point) and $\infty$, where sits the companion parabolic point under
the involution $\frac{-1}{\id}$. It turns out that the Birkhoff--Écalle-Voronin
modulus of the parabolic germ at $\infty$ is the inverse $\psi^{\circ-1}$
of that at $0$.
\end{abstract}

\maketitle
\emph{Acknowledgment}. I would like to express my greatest gratitude
to my skillful colleague R.~\noun{Schäfke,} who introduced me to
the joy of holomorphic fixed-point methods in dynamics.

\section{Introduction}

The classification of holomorphic, univariate parabolic germs up to
local change of analytic coordinate was first described by G.~\noun{Birkhoff}~\cite{Birk}
in 1939, but the manuscript has somehow been forgotten almost immediately
to be dug out only by the end of the century. In the meantime, J.~\noun{Écalle~\cite{Ecal}}
and S.~\noun{Voronin}~\cite{VoroParaboEng} carried out this task
again nearly forty years ago. In addition to building a modulus --characterizing
the conjugacy class of a parabolic germ-- they were able to solve
the inverse problem by recognizing which alike objects came as the
moduli of parabolic germs, a property G.~\noun{Birkhoff} was only
able to conjecture. The latter problem is what interests us in the
present paper.

\subsection{Modulus of classification and inverse problem}

Heuristically the conformal structure of the orbits space $\Omega$
of a non-trivial tangent-to-identity germ fixing $0\in\cc$
\begin{align*}
\Delta\left(z\right) & =z+\oo z\in\diff[\cc,0]_{\id}\backslash\left\{ \id\right\} 
\end{align*}
completely encodes the analytic class of $\Delta$ up to local changes
of analytic coordinate near the fixed point. For the sake of simplicity
assume that $\Delta$ is generic, \emph{i.e.} that it belongs to the
space
\begin{align*}
\parab & :=\left\{ z+*z^{2}+\oo{z^{2}}\right\} \subset\diff[\cc,0]_{\id}
\end{align*}
(here and in all the sequel $*$ stands for a nonzero complex number).

\hfill{}\includegraphics[width=0.9\columnwidth]{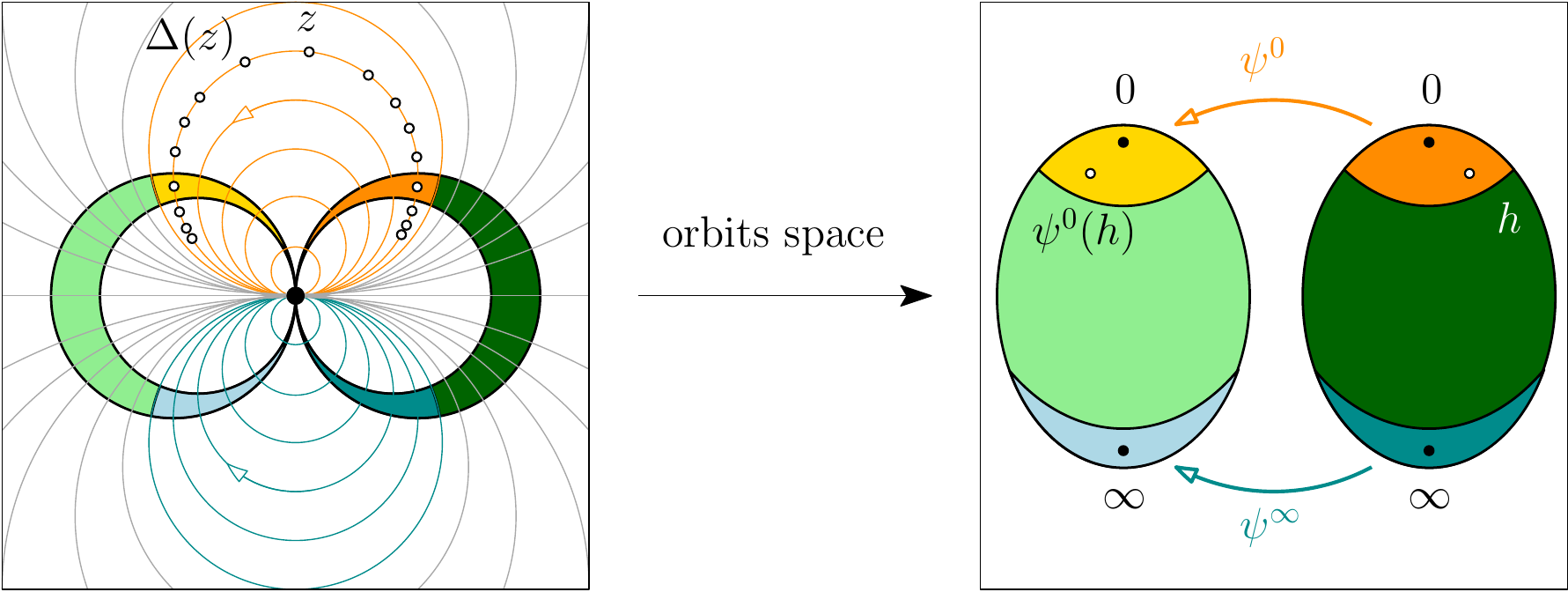}\hfill{}

It is well-known that a fundamental region of the iterative action
of $\Delta$ is given by a pair of crescent-shaped regions which are
identified by the dynamics near their horns, attached at the fixed
point $0$. Said more precisely, the orbits space
\begin{align*}
\Omega & =\quot{\neigh}{\Delta}
\end{align*}
is obtained as the gluing of two spheres $\cbar$ near $0$ and $\infty$
by germs of a diffeomorphism $\zero[\psi]\in\diff[\cbar,0]$ and $\infi[\psi]\in\diff[\cbar,\infty]$,
sometimes called <<horn maps>>:
\begin{align*}
\Omega & =\quot{\cbar\sqcup\cbar}{\left({\zero[\psi]},{\infi[\psi]}\right)}.
\end{align*}
When we endow the space $\quot{\neigh}{\Delta}$ with the quotient
topology, the two points $0$ and $\infty$ (corresponding to the
fixed point of $\Delta$) are not separated. 

One can choose conformal coordinates on the spheres so that $\infi[\psi]$
itself becomes tangent-to-identity. The only degree of freedom that
remains is the linear change of coordinates $h\mapsto ch$ for $c\in\cc^{\times}$
applied simultaneously to both spheres. The Birkhoff--Écalle--Voronin
theorem states that the map
\begin{align*}
\ev~:~\quot{\parab}{\diff[\cc,0]} & \longto\quot{{\diff[\cc,0]}\times{\diff[\cc,0]}_{\id}}{\cc^{\times}}\\
\left[\Delta\right] & \longmaps\left[\left(\zero[\psi],\infi[\psi]\right)\right]
\end{align*}
is well-defined and one-to-one. The Écalle--Voronin theorem states
that it is also onto, answering the inverse problem.
\begin{problem*}
Being given a pair $\left(\zero[\psi],\infi[\psi]\right)\in\diff[\cc,0]\times\diff[\cc,0]_{\id}$,
to recover (<<synthesize>>) a parabolic germ $\Delta$ such that
its modulus is $\ev\left(\Delta\right)=\left(\zero[\psi],\infi[\psi]\right)$.
\end{problem*}
In the present article we revisit the question and provide an essentially
unique representative $\Delta$ for a given pair $\left(\zero[\psi],\infi[\psi]\right)$,
what can be called a normal form. The germ $\Delta$ will be <<global>>
on the sphere $\cbar$ in some sense, which is why we borrow Écalle
terminology and speak about \emph{spherical normal forms}.

\subsection{Statement of the main results}

We answer the inverse problem in the following fashion.
\begin{namedthm}[Synthesis Theorem]
For fixed $\mu\in\cc$ and positive $\lambda>0$ define the rational
vector field (holomorphic near $0$ and $\infty$), called here the
\textbf{formal model},
\begin{align*}
X_{0}\left(z\right) & :=\frac{1-z^{2}}{1+z^{2}}\times\frac{\lambda z^{2}}{1+\mu\lambda z-z^{2}}\pp z
\end{align*}
and denote by $f\mapsto X_{0}\cdot f$ the associated Lie derivative
on power series. Let $\psi=\left(\zero[\psi],\infi[\psi]\right)\in{\diff[\cc,0]}\times{\diff[\cc,0]}_{\id}$
be given and pick $\mu$ so that $\ddd{\zero[\psi]}h\left(0\right)=\exp\left(4\pi^{2}\mu\right)$.

One can find an effective $\lambda\left(\psi\right)>0$ such that
for each $0<\lambda\leq\lambda\left(\psi\right)$ there exists a unique
$F\in z\frml z$ satisfying all of the following properties.
\begin{enumerate}
\item The time-$1$ map $\Delta$ of the formal vector field, reputed to
be in \textbf{spherical normal form}, 
\begin{align*}
X_{F}:= & \frac{1}{1+X_{0}\cdot F}X_{0}
\end{align*}
is a convergent power series at $0$.
\item $\Delta$ belongs to $\parab$ with
\begin{align*}
\ev\left(\Delta\right)= & \left(\zero[\psi],\infi[\psi]\right).
\end{align*}
\item The power series $F$ is $1$-summable and its 1-sum can be realized
as a pair $\left(f^{+},f^{-}\right)$ of functions holomorphic and
bounded by $1$ on the respective infinite sectors 
\begin{align*}
V^{\pm} & :=\left\{ z\neq0~:~\left|\arg\left(\pm z\right)\right|<\frac{5\pi}{8}\right\} 
\end{align*}
as in Figure~\ref{fig:Sectors}.
\end{enumerate}
\end{namedthm}
\begin{figure}[H]
\includegraphics[height=4cm]{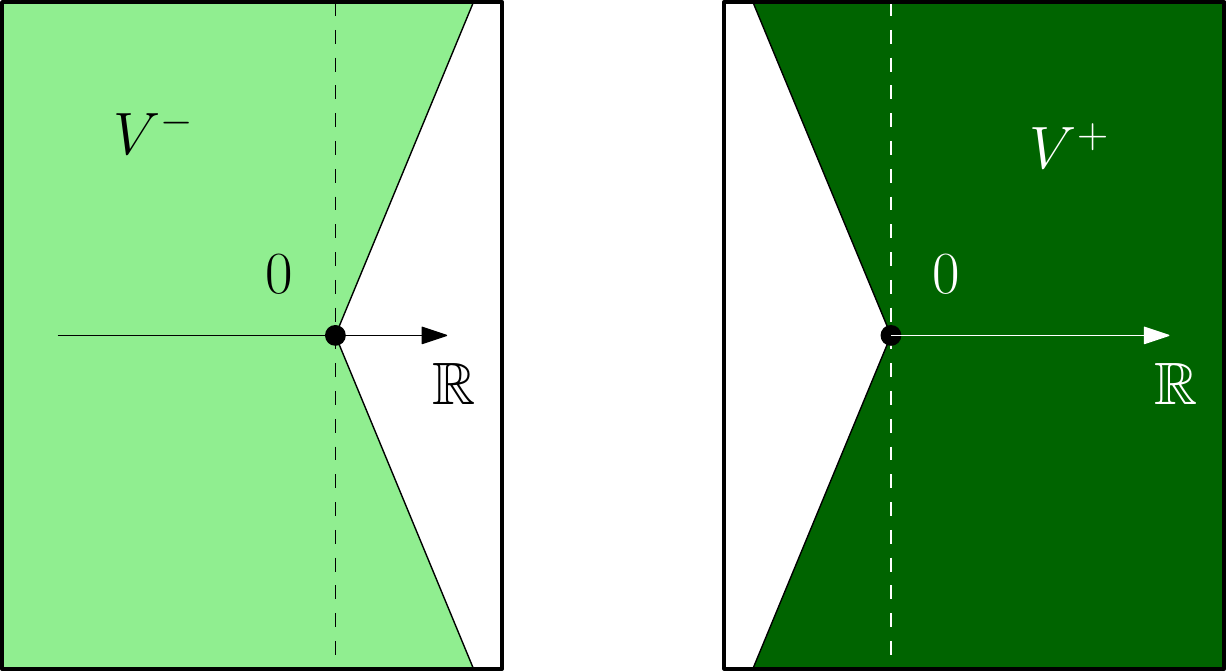}\hfill{}\includegraphics[height=4cm]{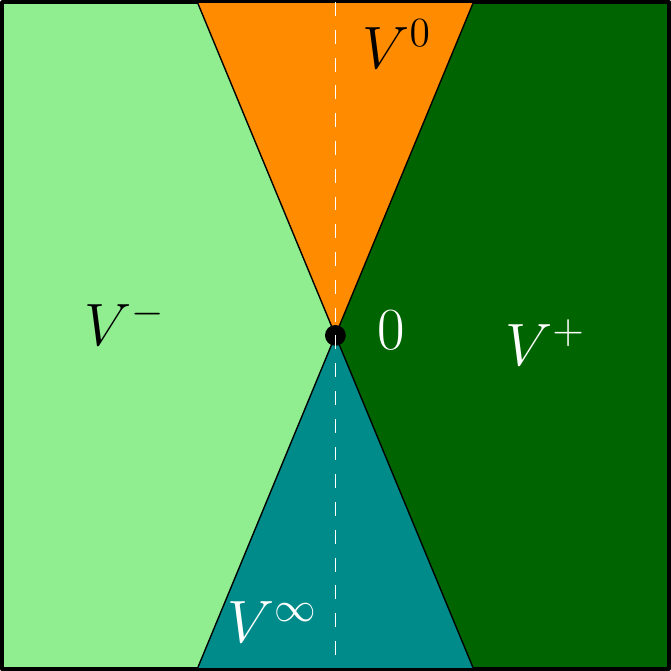}

\caption{\label{fig:Sectors}The infinite sectors $V^{\pm}$ and the components
$\protect\zero$, $\protect\infi$ of their intersection $V^{\cap}$.}
\end{figure}

A pair $f=\left(f^{+},f^{-}\right)$ of a function holomorphic on
the corresponding sector $V^{\pm}$ has order-1\textbf{ flat discrepancy}
at $0$ whenever
\begin{align*}
\limsup_{\begin{array}{l}
z\to0\\
z\in V^{\cap}
\end{array}}\left|z\right|\ln\left|f^{-}\left(z\right)-f^{+}\left(z\right)\right|\in\left]-\infty,0\right[ & ~~~,~V^{\cap}:=V^{+}\cap V^{-}.
\end{align*}
It follows from Ramis-Sibuya theorem~\cite[Theorem \#\#]{LodRich}
that there exists some $F=\sum_{n\geq0}F_{n}z^{n}\in\frml z$ which
is their common Gevrey-1 asymptotic expansion at $0$ (in the sense
of Poincaré):
\begin{align*}
\left(\exists C>0\right)\left(\forall N\in\nn\right) & \left(\forall z\in V^{\pm}\right)~~~~\left|f^{\pm}\left(z\right)-\sum_{n>0}^{N}F_{n}z^{n}\right|\leq C\left(N+1\right)!\left|z\right|^{N+1}.
\end{align*}
We then say that $F$ is \textbf{$1$-summable} with $1$-sum $f$.
The arrow $f\mapsto F$ is injective (because $V^{\pm}$ is wider
than a half-plane). This notion is compatible with the usual arithmetic
and differential operations (by taking maybe a narrower subsector).
\begin{rem}
The notion of $*$-summability has a standard, more general definition
(see \emph{e.g.~}\cite{LodRich}) allowing for more flexibility.
We will not need the whole refinement of the theory as we only rely
on this specific implementation of Ramis-Sibuya theorem.
\end{rem}

We can exhibit a value of $\lambda\left(\psi\right)$ as follows.
There exists $\mathfrak{m}_{\mu},\mathfrak{t}_{\mu}>0$ depending
only on $\mu$ (given for instance in~(\ref{eq:bounds_on_h})) for
which, if we denote by $\zero[\rho]$ and $\infi[\rho]$ the respective
radius of convergence of the Taylor series at $0$ of $\zero[\psi]$
and $\infi[\psi]$, we may define
\begin{align*}
\frac{1}{\ell\left(\psi\right)} & :=\max\left\{ 1,\mathfrak{t}_{\mu},2\pi+\ln\frac{\mathfrak{m}_{\mu}}{\min\left\{ \zero[\rho],\infi[\rho]\right\} }\right\} 
\end{align*}
 and take 
\begin{align*}
\lambda\left(\psi\right) & :=\min\left\{ \ell\left(\psi\right),\frac{1}{4\sqrt{\mathfrak{m}_{\mu}\norm[\psi]{\ell\left(\psi\right)}}}\right\} 
\end{align*}
where 
\begin{align*}
\norm[\psi]{\ell} & :=\max_{\sharp\in\left\{ 0,\infty\right\} }\sup_{\left|h\right|\leq\mathfrak{m}_{\mu}\exp\left(2\pi-\nf 1{\ell}\right)}\left|\ddd{\log\frac{{\isect[\psi]}}{\id}}h\left(h\right)\right|.
\end{align*}

\begin{rem}
~
\begin{enumerate}
\item Any holomorphic vector field $X$ with a double zero at $0$ is the
infinitesimal generator of a generic parabolic germ with modulus $\left(\ee^{4\pi^{2}\mu}\id,\id\right)$
for some $\mu\in\cc$, see~\cite{Ecal}. We say that the modulus
is \textbf{trivial}. This is particularly the case for $X_{0}$. As
a consequence of the uniqueness clause above, the only normal form
with convergent $F\in z\germ z$ is the formal model itself ($F=0$).
\item Let $X\in z^{2}\left(\cc^{\times}+z\frml z\right)\pp z$ be a formal
vector field with a double zero at $0$. It may happen that the formal
power series 
\begin{align*}
\Delta & :=\left(\exp X\right)\cdot\id
\end{align*}
is convergent at $0$, in which case it belongs to $\parab$ and we
say that $\Delta$ is the \textbf{time-1 map} of the infinitesimal
generator $X$. Beware that it does not imply the convergence of $\left(\exp tX\right)\cdot\id$
for $t\notin\zz$ even for small $\left|t\right|$: this is for instance
the case of the entire map $\Delta:z\mapsto\exp\left(z\right)-z$
(see~\cite{Bak}). If the lattice of those $t\in\cc$ for which the
power series converges is not discrete then the modulus of $\Delta$
is trivial~\cite{Ecal}.
\item There might exist other representatives of $\left(\zero[\psi],\infi[\psi]\right)$
of the form $\frac{1}{1+X_{0}\cdot F}X_{0}$ for some $1$-summable
$F$ with $1$-sum $f$ on $V^{\pm}$. Then $f$ must have large norm.
This is likely related to what F.~\noun{Loray} discusses after Theorem~1
in~\cite{Loray}.
\end{enumerate}
\end{rem}

The functions $z\in V^{\pm}\mapsto f^{\pm}\left(z\right)-\frac{1-z^{2}}{\lambda z}+\mu\log\frac{\lambda z}{1-z^{2}}$
are \textbf{Fatou coordinates} for $\Delta$ (\emph{i.e.} local conformal
coordinates $\neigh\to\left(\cbar,\infty\right)$ in which $\Delta$
acts as the translation by $1$, see Section~\ref{subsec:sectorial_dynamics}).
Hence we parameterize the analytic classes of parabolic germs by their
formal Fatou coordinates. By picking them in a class of <<global>>
formal objects we benefit from a rigidity making the parameterization
injective.

The power series $F$ is global in the sense that its $1$-sum $f=\left(f^{+},f^{-}\right)$
exists on a pair of sectors whose union covers $\cc^{\times}$. Moreover
$\Delta|_{V^{\pm}}$ is the time-1 map of the $1$-sum of its formal
infinitesimal generator
\begin{align*}
X_{}^{\pm} & :=\frac{1}{1+X_{0}\cdot f^{\pm}}X_{0}
\end{align*}
which is a meromorphic vector field on $V^{\pm}$. A quick computation
tells us that the formal model $X_{0}$ is invariant by the involution
\begin{align*}
\sigma~:~\cbar & \longto\cbar\\
z & \mapsto\frac{-1}{z}
\end{align*}
fixing $\pm\ii$, while the sectors are swapped
\begin{align*}
\sigma\left(V^{\pm}\right) & =V^{\mp}.
\end{align*}
A nice feature of the synthesis is that $f$ not only defines a parabolic
germ near $0$ but it also creates one at $\infty$, corresponding
to the companion dynamics induced there by $X^{\pm}$, since (Section~\ref{subsec:Cauchy-Heine}):
\begin{align}
f^{\pm}\circ\sigma & =-f^{\mp}.\label{eq:f_sigma_invar}
\end{align}
This is why we call these dynamical systems \emph{spherical} normal
forms, a term borrowed from J.~\noun{Écalle~\cite{EcalReal}} as
discussed in Section~\ref{subsec:=0000C9calle}. We can describe
completely the dynamics of $\Delta$ and its monodromy as a multivalued
map over the whole Riemann sphere $\cbar$.

\begin{figure}
\hfill{}\includegraphics[width=0.75\columnwidth]{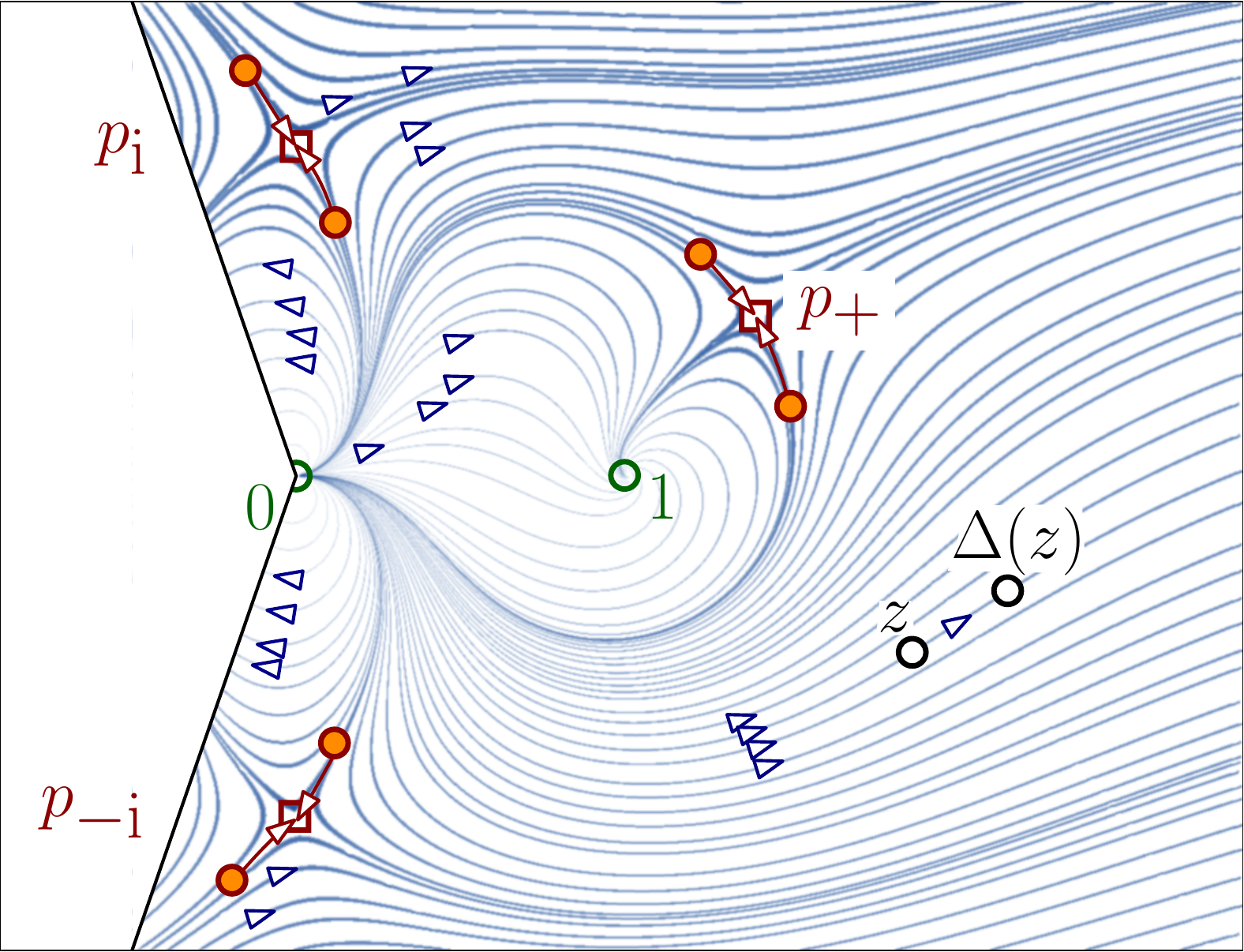}\hfill{}

\caption{\label{fig:generic_sector}Foliation induced by the real-time flow
of a typical sectorial vector field $X^{+}$ with the highlighted
6 ramification points $z_{p}$, $w_{p}$ (orange spots) of its time-1
map $\Delta$. The poles $p_{-\protect\ii},~p_{\protect\ii},~p_{+}$
of $X^{+}$ are figured by red squares.}
\end{figure}

\begin{namedthm}[Globalization Theorem]
Assume $\mu\neq0$. 
\begin{enumerate}
\item If $\lambda<\min\left\{ \frac{1}{2\left|\mu\right|}~,~\frac{1}{12\thinspace800}\right\} $
then $X_{}^{\pm}$ has exactly three (simple) poles $\left\{ p_{-\ii},p_{\ii},p_{\pm}\right\} $
which are $\OO{\sqrt{\lambda}}$-close to $-\ii$, $\ii$ and $\pm1$
respectively. This particularly means that the pole $p_{\pm\ii}$
is shared by $X^{-}$ and $X^{+}$.
\item (We refer to Figure~\ref{fig:generic_sector}.) To each one of the
four poles $p\in\left\{ p_{-\ii},p_{\ii},p_{-},p_{+}\right\} $ is
attached a unique couple of $\OO{\sqrt{\lambda}}$-close points $\left\{ z_{p},w_{p}\right\} $
lying at the intersection of the stable manifolds through $p$ of
$X_{f}^{+}$ and $X_{f}^{-}$, mapped onto $p$ by the time-1 map
of $X_{f}^{\pm}$. In other words, the analytic continuation of $\Delta$
to $\cbar$ provides a multivalued map with ramification points at
each $z_{p}$ and $w_{p}$,\textup{\emph{ to which it extends}}\emph{
}continuously by $\Delta\left(z_{p}\right)=\Delta\left(w_{p}\right)=p$.
A domain of holomorphy for $\Delta$ is given for instance by the
slit sphere
\begin{align*}
\mathcal{D}:= & \cbar\backslash\bigcup_{p\text{ pole}}\gamma_{p},
\end{align*}
where $\gamma_{p}$ is the arc of stable manifold passing through
$p$ of $X_{f}^{+}$ or $X_{f}^{-}$ (choose one) linking $z_{p}$
and $w_{p}$. In particular the radius of convergence of $\Delta$
is $1+\OO{\sqrt{\lambda}}$.
\item The analytic continuation of $\Delta$ to $\cbar$ has $8$ branch
points. The monodromy around any one of the $8$ points $\left\{ z_{p},~w_{p}~:~p\in\left\{ p_{-\ii},p_{\ii},p_{-},p_{+}\right\} \right\} $
is involutive.
\item $\Delta|_{\mathcal{D}}$ is holomorphic and injective. It sports four
fixed-points $0$, $\infty$ and $\pm1$. The first two are parabolic
while $\Delta$ admits a linearizable dynamics at $\pm1$ of multiplier
$\exp\left(\mp\frac{1}{\mu}\right)$. The dynamics at infinity $\sigma^{*}\Delta$
belongs to $\parab$ and 
\begin{align*}
\ev\left(\sigma^{*}\Delta\right) & =\ev\left(\Delta\right)^{\circ-1}.
\end{align*}
\end{enumerate}
In case $\mu=0$ the result still holds save for the fact that the
pole $z_{\pm}$ near $\pm1$ cancels the stationary point at $\pm1$
out. There only remain the parabolic fixed-points of $\Delta$ and
its ramification locus coming from $\left\{ -\ii,\ii\right\} $.
\end{namedthm}
It is remarkable that a twin fixed-point is automatically produced
at $\infty$, especially because the modulus there is given by the
reciprocal modulus. This has a simple explanation, though: the involution
$\sigma$ permutes the sectors $V^{+}$ and $V^{-}$ while it does
not change much the infinitesimal generators, therefore the identification
at the horns is the same while taking place in the reverse direction.

Another remarkable feature is that $\Delta$ has a very simple dynamics:
it is an injective, holomorphic map on a slit sphere $\mathcal{D}$,
that can be forward iterated on the open set $\cbar\backslash\gamma$
where $\gamma$ is the union of the closure of the stable manifolds
of $X^{\pm}$ passing through the poles. On each connected component
of $\cbar\backslash\gamma$ the dynamics of $\Delta$ converges uniformly
to a fixed-point (a parabolic one or the attractive point sitting
at $\pm1$ when $\mu\notin\ii\rr$). The reason behind this steady
behavior is the parameter $\lambda$, which is chosen so that no chaotic
remains of a global dynamics, usually materialized by a frontier in
the analytic continuation of the modulus~\cite{Epstein}, can be
seen from the normal form.

The obvious downside of the previous feature is precisely that the
dynamics of $\Delta$ is too simple to readily offer an insight regarding
the dynamical richness that is displayed by holomorphic iteration
on compact curves. What we prove is that any such dynamics, say of
a rational mapping with a parabolic basin $\mathcal{B}$, is locally
conjugate by some $\Psi$ to a spherical normal form near the parabolic
fixed-point $0$, but the domain of the conjugacy cannot contain the
whole component $\partial\mathcal{B}$ of the Julia set. Indeed, $\Psi\left(\partial\mathcal{B}\cap\neigh\right)$
can be pushed forward by the flow of $X^{+}$ to give a compact invariant
set $J$ which is locally conformally equivalent to the fractal set
$\partial\mathcal{B}$. But the one set $J$ has no relevance to the
dynamics of $\Delta$ whatsoever, for the latter has only Fatou components
with real-analytic boundary, and that regularity increases as $\lambda\to0$.
In that respect, an interesting question arises: can one recover part
of the original dynamics by letting $\lambda$ grow (which may cause
$X^{\pm}$ to sport more and more poles, for instance)? 

\subsubsection{About the formal model}

\begin{figure}
\hfill{}\includegraphics[width=0.45\columnwidth]{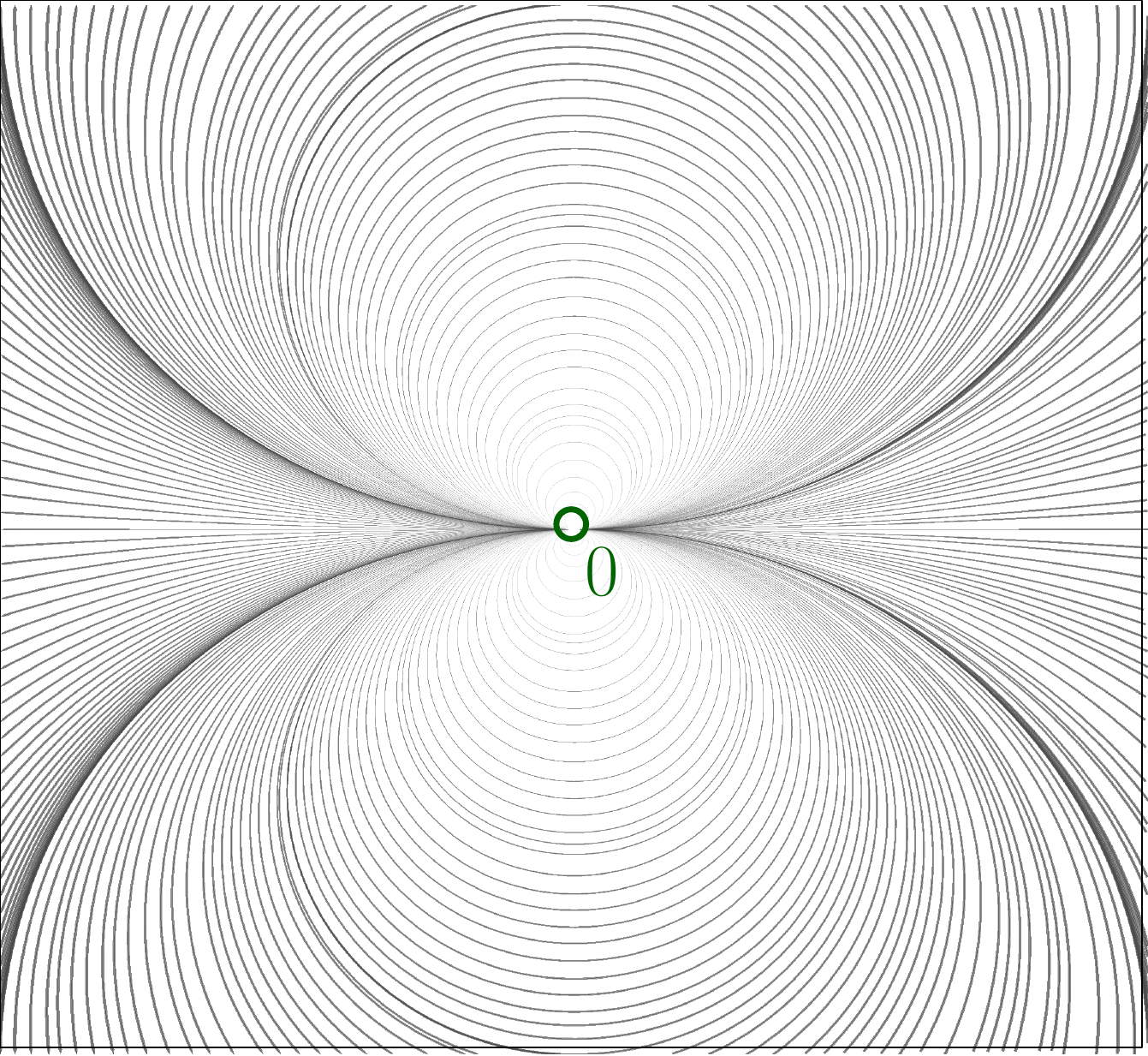}\hfill{}\includegraphics[width=0.45\columnwidth]{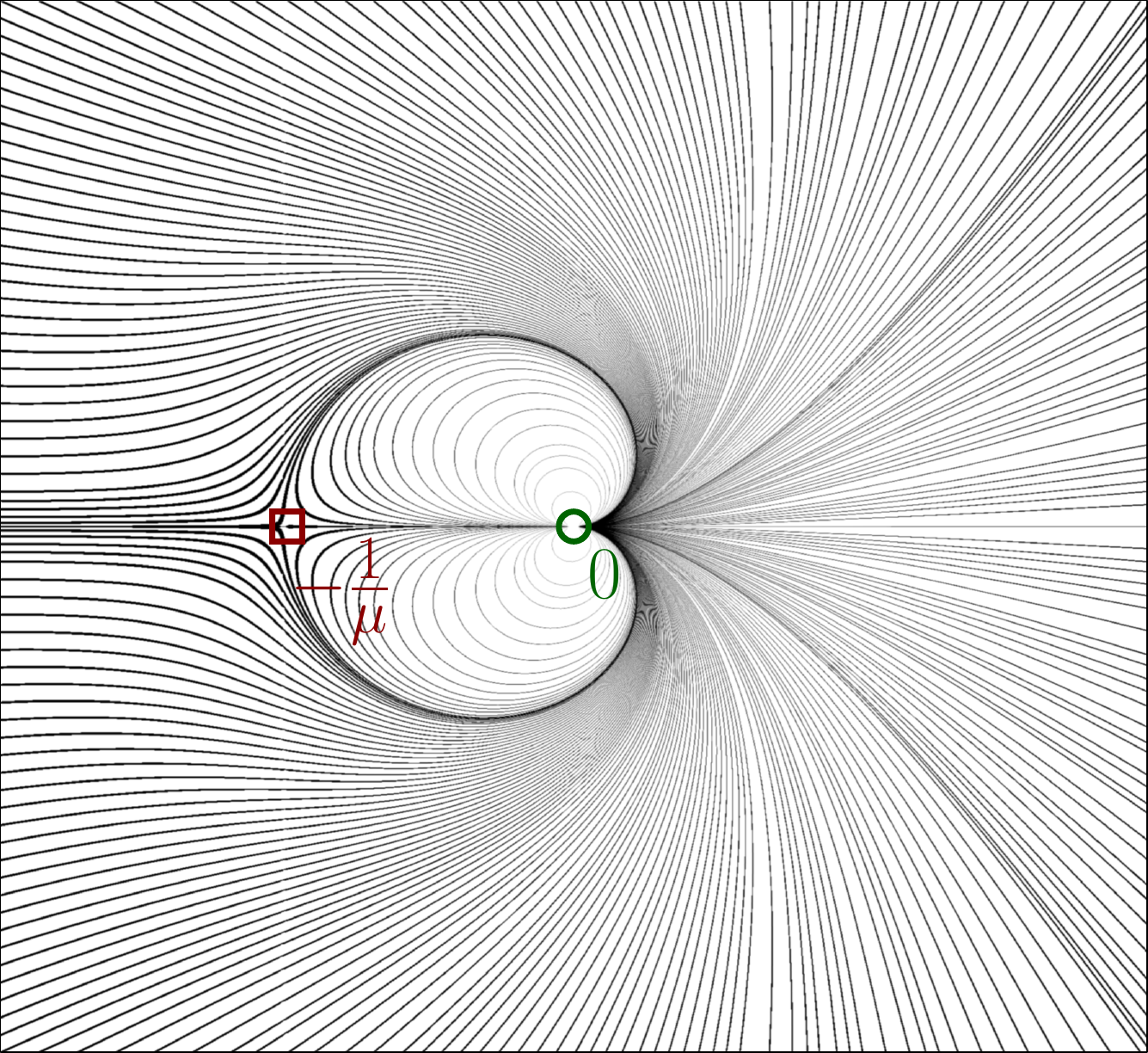}\hfill{}

\caption{\label{fig:model_standard}Foliation induced by the real-time flow
of $\frac{x^{2}}{1+\mu x}\protect\pp x$ (left $\mu:=0$, right $\mu:=2$).
There are one double stationary point at $0$ (green circle), yielding
a parabolic germ for the time-1 map, and one saddle point (red square)
corresponding to the pole $-\frac{1}{\mu}$.}
\end{figure}

The family of vector fields $X_{0}$ depending on $\lambda>0$ is
obtained from the <<usual>> formal model $\frac{x^{2}}{1+\mu x}\pp x$
(see Figure~\ref{fig:model_standard}) by performing the degree-2
pullback
\begin{align*}
\Pi~:~z & \mapsto x=x\left(z\right):=\frac{\lambda z}{1-z^{2}}
\end{align*}
(which is invertible near $0$). Observe that $\Pi$ is invariant
under the involution $\sigma$, so that $\sigma$ is a symmetry of
$X_{0}$, a fact which has been already remarked. The obvious consequence
is that $X_{0}$ has a stationary point at $\infty$ that mirrors
the one it admits at $0$.
\begin{rem}
The only normal form $X_{f}$ that can be factored through $\Pi$
(\emph{i.e. }realized as a global sectorial perturbation of the usual
formal model) is the formal model itself ($f=0$) since we must both
have $f^{\pm}\circ\sigma=f^{\mp}$ (factorization through $\Pi$)
and $f^{\pm}\circ\sigma=-f^{\mp}$ as in~(\ref{eq:f_sigma_invar}).
\end{rem}

The time-1 map of $X_{0}$ is a generic parabolic germ (Lemma~\ref{lem:.model_parab})
\begin{align*}
\Delta_{0}\left(z\right):=\flow{X_{0}}1{\left(z\right)} & =z+\lambda z^{2}+\lambda^{2}\left(1-\mu\right)z^{3}+\oo{z^{3}}
\end{align*}
with trivial modulus. Any germ $\Delta\in\parab$ with same $\mu$
is formally conjugate to $\Delta_{0}$ (whatever $\lambda$), which
explains the terminology \emph{formal model}.

The sectors  and the model $X_{0}$ are tailored in such a way that
the orbital region of $\Delta_{0}$, as seen from the intersection
$V^{\cap}$, is a neighborhood of $0$ or $\infty$ of size $\OO{\mathfrak{m}_{\mu}\ee^{-\nf 1{\lambda}}}$
as $\lambda\to0$ (\emph{cf} Section~\ref{subsec:choice_X0}). Gaining
such a control is essential: eventually the neighborhood fits within
the discs of convergence of the data $\left(\zero[\psi],\infi[\psi]\right)$
and keeps sitting there even after small perturbations.

Picking $X_{0}$ among all vector fields realizing the trivial pair
$\left(\ee^{4\pi^{2}\mu}\id,\id\right)$ is arguably canonical since
the pullback $z\mapsto\frac{\lambda z}{1-z^{2}}$ is very simple and
the pair of sectors $\left(V^{+},V^{-}\right)$ is fairly standard.
\begin{rem}
For technical reasons we also need that $\infty$ and $0$ be mapped
to $0$ by the pullback. Thereby the <<obvious>> choice of a degree-1
pullback, like $z\mapsto\frac{\lambda z}{1-z}$, cannot be used to
perform the construction presented here. Apart from this restriction
one can use pretty much any other pullback, provided the <<sectors>>
are tailored conveniently.
\end{rem}

\begin{example}
\label{exa:mu_zero}We refer to Figure~\ref{fig:model_mu_0}. Assume
here that $\mu:=0$ so that
\begin{align*}
X_{0}\left(z\right) & =\frac{\lambda z^{2}}{1+z^{2}}\pp z
\end{align*}
has only stationary (double) points at $0$ and $\infty$, as well
as two simple poles at $\pm\ii$ with residue $\pm\frac{\lambda\ii}{2}$.
These poles spawn the ramification locus of $\Delta_{0}$, as can
be seen from a direct integration of $\dot{z}=X_{0}\left(z\right)$
at time $1$ with initial value $z_{*}$:
\begin{align*}
1 & =\int_{z_{*}}^{\Delta_{0}\left(z_{*}\right)}\frac{1+z^{2}}{\lambda z^{2}}\dd z
\end{align*}
which yields
\begin{align*}
z_{*}\Delta_{0}\left(z_{*}\right)^{2}-\left(z_{*}^{2}+\lambda z_{*}-1\right)\Delta_{0}\left(z_{*}\right)-z_{*} & =0.
\end{align*}
Hence $\Delta_{0}$ is algebraic and its ramification points are the
four points $z$ for which:
\begin{align*}
\left(z^{2}+\lambda z-1\right)^{2}+4z^{2} & =0.
\end{align*}
Notice that $\Delta_{0}$ admits a limit at each one of these points
$z$, and $\Delta_{0}\left(z\right)$ lies within the polar locus
$\left\{ \pm\ii\right\} $ of $X_{0}$. Therefore $\Delta_{0}$ extends
as a multivalued function over the domain $\cbar\backslash\left\{ z^{2}+\lambda z=1\right\} $
(which contains $0$) with involutive monodromy around any one of
the ramification points given by the action of $\sigma$, which is
moreover sent to a pole of $X_{0}$ by the time-1 map. More details
can be found in Lemma~\ref{lem:polar_foliation}.
\end{example}

\begin{rem}
When $\mu\neq0$ the mapping $\Delta_{0}$ cannot be algebraic (see~\cite{Ecal}).
It seems safe to conjecture that the only algebraic normal form $\Delta$
is $\Delta_{0}$ when $\mu=0$.
\end{rem}

\begin{figure}
\hfill{}\includegraphics[height=7cm]{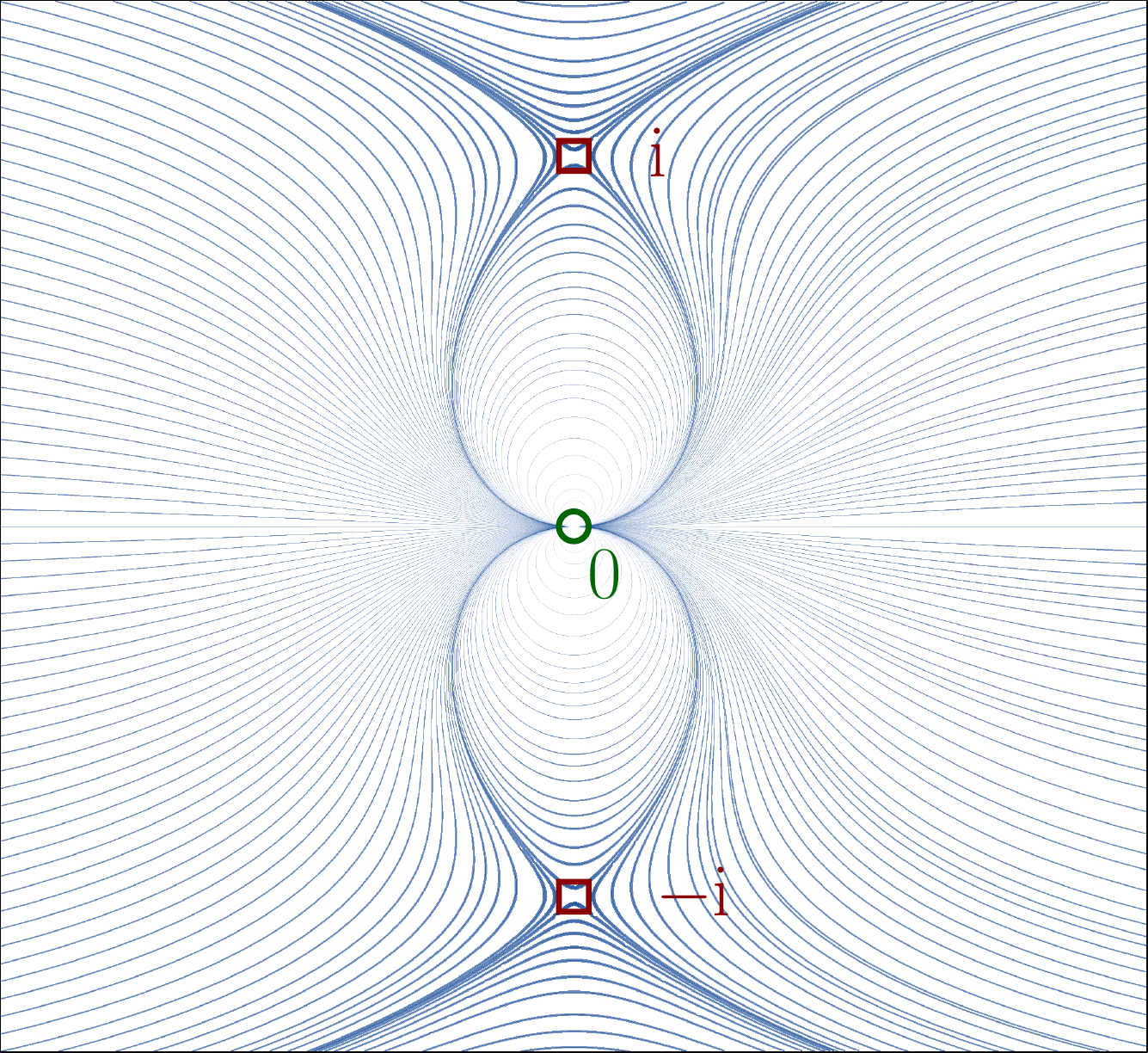}\hfill{}

\caption{\label{fig:model_mu_0}Foliation induced by the real-time flow of
$X_{0}$ for $\mu:=0$, revealing the double stationary point (green
circle) and two simple poles (red squares).}
\end{figure}

\subsubsection{\label{subsec:=0000C9calle}Link with Écalle canonical synthesis}

In~\cite{EcalReal} J.~\noun{Écalle} revisits the inverse problem
by performing the <<canonical-spherical synthesis>> using the framework
of resurgent functions and alien/mould calculus~(for the specific
context at hands see~\cite{Ecal_III}) that he developed in the wake
of his original paper~\cite{Ecal}. In that setting the modulus is
not built from dynamical considerations, as was first done by G.~\noun{Birkhoff},
himself and S.~\noun{Voronin} (the viewpoint taken here), but as
coefficients $\left(A_{n}\right)_{n\in\zz}$ coming from the associated
<<resurgence equation>> and carried by the discrete set of singularities
of the analytic continuation of the Borel transform of the formal
normalization $\Delta\widehat{\sim}\Delta_{0}$. The link between
the pair $\left(\zero[\psi],\infi[\psi]\right)$ and the collection
$\left(A_{n}\right)_{n}$ is described in~\cite[Section 5.1]{EcalReal}.

To perform the canonical-spherical synthesis J.~\noun{Écalle} introduces
a (usually large) parameter he calls the \textbf{twist} $c>0$, which
corresponds here to $\frac{1}{\lambda}$, and a family of resurgent
monomials which serves as building blocks for his synthesis. It turns
out that the normal forms produced here are very likely the objects
obtained by J.~\noun{Écalle}, as witnessed for instance by the definition
of the Cauchy-like kernel~\cite[Sections 5.1 and 10.3]{EcalReal}
which is pretty much the one involved in the present construction
(compare Definition~\ref{def:Cauchy-Heine}).
\begin{namedthm}[Canonical-Spherical Synthesis Conjecture]
Being given a data $\psi:=\left(\zero[\psi],\infi[\psi]\right)\in\diff[\cc,0]\times\diff[\cc,0]_{\id}$
and $0<\lambda\leq\lambda\left(\psi\right)$, define the twist $c:=\nf 1{\lambda}$.
The normal form $\Delta$ given by the Synthesis Theorem is the parabolic
germ coming from Écalle's canonical-spherical synthesis.
\end{namedthm}
If it were true this would allow us to invoke the effective methods
provided by~\cite{EcalReal} in the present dynamical context. Conversely,
let us discuss in which sense the approach given here completes and
sheds a dynamical light on Écalle's synthesis for parabolic germs.
To quote J.~\noun{Écalle} (beware that he works in the variable $\frac{1}{z}$):

\smallskip{}

<<\emph{As already pointed out, our twisted monomials have much the
same behavior at both poles of the Riemann sphere. The exact correspondence
has just been described {[}...{]} using the so-called antipodal involution:
in terms of the objects being produced, this means that canonical
object synthesis automatically generates two objects: the \textquotedblleft true\textquotedblright{}
object, local at $\infty$ and with exactly the prescribed invariants,
and a \textquotedblleft mirror reflection\textquotedblright , local
at 0 and with closely related invariants. Depending on the nature
of the {[}...{]} invariants (verification or non-verification of an
\textquotedblleft overlapping condition\textquotedblright ), these
two objects may or may not link up under analytic continuation on
the Riemann sphere.}>>

\smallskip{}

It could happen that both synthesis processes produce different objects,
but what is sure is that the objects built in this paper match very
precisely the above expectations and comments. In our setting the
<<mirror reflection>> comes from the degree-2 pullback $\Pi~:~z\mapsto\frac{\lambda z}{1-z^{2}}$
invariant by the symmetry $\sigma$. The Synthesis Theorem gives quantitative
bounds and a uniqueness statement that are not readily available in~\cite{EcalReal},
while the Globalization Theorem simply tells that the normal form
is defined and injective on almost all $\cc$. Although it is true
that the Taylor series of $\Delta$ at $0$, and that of its companion
near $\infty$, has a radius of convergence of order $1+\OO{\sqrt{\lambda}}$,
both germs are \emph{unconditionally} obtained by analytic continuation
one from the other, provided of course that $0<\lambda\leq\lambda\left(\psi\right)$.
In our synthesis the <<closely related>> moduli are in fact mutually
reciprocal.

\subsubsection{General parabolic and rationally indifferent germs}

Even though in this article we address in details only the case of
germs in $\parab$, our construction can be adapted straightforwardly
to the space $\parab[k]$ consisting in those germs in the form $\Delta\left(z\right)=z+*z^{k+1}+\oo{z^{k+1}}$
for some $k\in\zz_{>0}$. The inverse problem then regards the realization
of collections $\left(\zero[\psi]_{j},\infi[\psi]_{j}\right)_{j\in\zsk}$
and in that case:
\begin{itemize}
\item the formal model $X_{0,k}$ is given by the pull-back of $\frac{x^{k+1}}{1+\mu x^{k}}\pp x$
by the degree-$2k$ mapping $z\mapsto\frac{\lambda z}{1-z^{2k}}$;
\item the normal form is $\frac{1}{1+X_{0,k}\cdot F}X_{0,k}$ with a $k$-summable
$F\in z\frml z$;
\item $F$ admits bounded $k$-sums on infinite sectors $\left\{ z\neq0~:~\left|\arg\left(\pm z^{k}\right)\right|<\frac{5\pi}{8}\right\} $.
\end{itemize}
The general case of a rationally indifferent germ $\Delta\left(z\right)=\ee^{2\ii\pi\nf pq}z+\oo z$
with $p\wedge q=1$, not conjugate to a rational rotation, can also
be covered as usual by taking a $kq$ ramified covering $x\mapsto x^{kq}$
of the usual formal model and sectors.

\subsection{Parabolic renormalization}

The parabolic renormalization is the map 
\begin{align*}
\mathcal{R}~:~\parab[1] & \longto\parab[~]\\
\Delta & \longmapsto\ev\left(\Delta\right)^{\infty}
\end{align*}
which to a generic parabolic germ associates the $\infty$-component
of its invariant. It may happen that $\mathcal{R}\left(\Delta\right)$
be itself generic, in which case $\mathcal{R}$ can be iterated. The
fixed-points of $\mathcal{R}$ play an important role in iterative
complex dynamics; here we describe those in spherical normal form.
As far as I know a statement of the kind was first communicated privately
by R.~\noun{Schäfke}.
\begin{namedthm}[Parabolic Renormalization Corollary]
 Take an arbitrary $\zero[\psi]\in\diff[\cc,0]$. There exists a
bound $\widehat{\lambda}\left(\zero[\psi]\right)>0$ such that, for
all $0<\lambda\leq\widehat{\lambda}\left(\zero[\psi]\right)$ there
exists a unique normal form $\Delta$ satisfying
\begin{align*}
\ev\left(\Delta\right) & =\left(\zero[\psi],\Delta\right).
\end{align*}
\end{namedthm}
Again the bound $\widehat{\lambda}$ can be made explicit, we refer
to Corollary~\ref{cor:fixed-point_parab}.

\subsection{Real case}

If $\Delta\in\parab$ is real (that is, commutes with the complex
conjugation $\overline{\bullet}$) then $\mu\in\rr$ and its modulus
$\ev\left(\Delta\right)=\left(\zero[\psi],\infi[\psi]\right)$ satisfies
the identity:
\begin{align*}
\left(\forall h\in\neigh\right)~~~~\overline{\zero[\psi]\left(\overline{h}\right)} & =\frac{\ee^{4\pi^{2}\mu}}{\infi[\psi]\left(\frac{1}{h}\right)}.\tag{\ensuremath{\square}}
\end{align*}

\begin{namedthm}[Real Synthesis Corollary]
 Take $\mu\in\rr$, a pair $\left(\zero[\psi],\infi[\psi]\right)$
and $\lambda>0$ as in the Synthesis Theorem. If moreover the identity~$\left(\square\right)$
holds then the synthesized normal form $\Delta$ is real.
\end{namedthm}

\subsection{\label{subsec:Structure}Structure of the article and summary of
the proofs}

One of the main concerns of this paper is to provide bounds as explicit
as possible, which explains why so many efforts went into technical
lemmas. The main tools are basic calculus~/~complex analysis, and
the work is necessary anyway to obtain uniform bounds with respect
to $\lambda$.

All the constructions will take place within the space of pairs $f=\left(f^{+},f^{-}\right)$
with 1-flat difference in $V^{\cap}$, both at $0$ and $\infty$:
\begin{align*}
\mathcal{S} & :=\left\{ f\in\mathcal{S}\left(V^{+}\right)\times\mathcal{S}\left(V^{-}\right)~:~\limsup_{z\to0,\infty}\left|z\right|\ln\left|f^{-}\left(z\right)-f^{+}\left(z\right)\right|<0\text{ for }z\in V^{\cap}\right\} ,
\end{align*}
where
\begin{align*}
\mathcal{S}\left(V^{\pm}\right) & :=\left\{ f^{\pm}\begin{array}{c}
\text{ holomorphic on }V^{\pm}\\
\text{continuous on }\adh{V^{\pm}}
\end{array}~:~\begin{array}{l}
f^{\pm}\left(0\right)=0\\
f^{\pm}\left(\infty\right)=0
\end{array}~,~\norm[f^{\pm}]{}:=\sup_{z\in V^{\pm}}\left|f^{\pm}\left(z\right)\right|<\infty\right\} ,
\end{align*}
equipped with the canonical product Banach norm. We denote by $\mathcal{B}$
its unit ball.
\begin{enumerate}
\item In Section~\ref{sec:Secto_normalization} we give some basic material
about vector fields and their time-1 maps, focusing more particularly
on the local dynamics near their poles and zeroes and how they stitch
together to organize the global dynamics of the model $X_{0}$. This
will allow us to conjugate 
\begin{align*}
X_{f}^{\pm} & :=\frac{1}{1+X_{0}\cdot f^{\pm}}X_{0}
\end{align*}
to the model $X_{0}$ over the better part of $V^{\pm}$ (effective
bounds are described in Section~\ref{subsec:Quantitative-bounds}).
This sectorial mapping takes the form $\flow{X_{0}}{f^{\pm}}{}$ given
by the flow of $X_{0}$ at time $f^{\pm}$. The dynamics of the formal
model $\Delta_{0}$, studied in Section~\ref{subsec:formal-model},
is then pulled-back by $\Psi^{\pm}$ to give sectorial counterparts
for the time-1 map of $X_{f}^{\pm}$.
\item Coupled to Ramis-Sibuya theorem, the above preliminary part allows
one to reduce the problem of finding a parabolic germ $\Delta$ with
prescribed modulus to that of finding $\left(f^{+},f^{-}\right)\in\mathcal{S}$
satisfying some (non-linear) Cousin problem~$\left(\star\right)$
below; see Proposition~\ref{prop:reduction} in Section~\ref{sec:Reduction}.
\item We are equipped to tackle the main result (Section~\ref{sec:Main_th}).
We begin with fixing a formal class $\mu\in\cc$ in $\parab$. Then
$\zero[\psi]$ must be tangent to the linear map $\ee^{4\pi^{2}\mu}\id$.
Being given $\left(\zero[\psi],\infi[\psi]\right)=\left(\id\exp\left(4\pi^{2}\mu+\zero[\varphi]\right),\id\exp\infi[\varphi]\right)$
with $\isect[\varphi]$ holomorphic near $\sharp\in\left\{ 0,\infty\right\} $
we must find $f^{\pm}$ on $V^{\pm}$ such that
\begin{align*}
\begin{cases}
f^{-}-f^{+} & =\zero[\varphi]\circ H_{f}^{+}\text{~~~~ on }\zero\\
f^{-}-f^{+} & =\infi[\varphi]\circ H_{f}^{+}\text{~~~~ on }\infi
\end{cases} & \tag{\ensuremath{\star}}
\end{align*}
where $\left(\isect\right)_{\sharp\in\left\{ 0,\infty\right\} }$
are the two components of $V^{\cap}$ and $H_{f}^{+}$ is a primitive
first-integral of the sectorial time-1 map $\Delta^{+}$ of $X_{f}^{+}$
(in other words, $H_{f}^{+}$ provide preferred orbital coordinates
over the sector $V^{+}$). Condition~$\left(\star\right)$ implies
that $\Delta^{+}$ coincides with $\Delta^{-}$ in $V^{\cap}$ and
thus is the restriction to $V^{+}$ of a single  $\Delta\in\parab$.

The non-linear Cousin problem~$\left(\star\right)$ is solved using
a fixed-point method involving a classical Cauchy-Heine integral transform.
This technique was already employed in~\cite{SchaTey} (and~\cite{RT2})
to build normal forms for convergent saddle-node foliations in the
complex plane (and their unfoldings). More precisely, we build a map
(Proposition~\ref{prop:cauchy-heine}):
\begin{align*}
\cauchein~:~ & f\longmapsto\Lambda\left(f\right)
\end{align*}
such that $\Lambda\left(f\right)^{-}-\Lambda\left(f\right)^{+}=\varphi\circ H_{f}^{+}$.
Therefore $\cauchein\left(f\right)=f$ if, and only if,~$\left(\star\right)$
holds. We prove (Proposition~\ref{prop:CH_contractant}) that $\cauchein$
eventually stabilizes the unit ball $\mathcal{B}\subset\mathcal{S}$
(for $\lambda$ small enough). The bigger $\lambda$ the more contracting
$\cauchein$, so that eventually a unique fixed-point of $\cauchein$
exists in that ball (Corollary~\ref{cor:fixed-point_existence_local_uniqueness}).
\item In Section~\ref{subsec:inversion_module}, as is discussed above,
we establish the Globalization Theorem making use of the involution
$\sigma$ and the property that for the fixed-point $f$ of $CH^{\varphi}$
we have
\begin{align*}
f^{\pm}\circ\sigma & =-f^{\mp}.
\end{align*}
Prior to that we describe the dynamical properties of $\Delta$ in
Section~\ref{subsec:Sectorial_dynamics} using \emph{a priori} bounds
on the flow of $X^{\pm}$.
\item The fixed-points of the parabolic renormalization are also obtained
as a holomorphic fixed-point in Section~\ref{subsec:Parabolic}.
Let $\text{Synth}_{\lambda}~:~\psi\mapsto\Delta$ be the synthesis
map built previously and consider for given $\zero[\psi]\in\diff[\cc,0]$
the iteration $\Delta_{0}:=\id$ and $\Delta_{n+1}:=\text{Synth}_{\lambda}\left(\zero[\psi],\Delta_{n}\right)$.
It so happens that $\text{Synth}_{\lambda}$ cannot shrink too much
the radius of convergence of $\Delta_{n}$, which always remains above
$\frac{3}{16}$. By taking $\lambda$ smaller we prove next that $\text{Synth}_{\lambda}$
is a contracting self-map of some closed ball of the Banach space
$\holb[\frac{3}{16}\ww D]$.
\item The real setting is explored in Section~\ref{subsec:Real_case} (Real
Synthesis Corollary). It merely consists in checking that every step
leading to the fixed-point of $\cauchein$ preserves realness, which
is straightforward.
\end{enumerate}

\subsection{Notations}
\begin{itemize}
\item $\dot{z}=\ddd zt$ stands for differentiation with respect to the
variable $t\in\cc$.
\item For a topological space $\mathcal{M}$ and a point $p\in\mathcal{M}$
we use $\left(\mathcal{M},p\right)$ to stand for a (sufficiently
small) neighborhood of $p$ in $\mathcal{M}$.
\item The set $\overline{U}$ is the topological closure of $U\subset\mathcal{M}$.
\item We use the notation $\ww D$ for the standard open unit disk in $\cc$.
\item $\cbar$ stands for the compactification of the complex line $\cc$
as the Riemann sphere.
\item The $n^{\text{th}}$ iterate of a diffeomorphism $\Delta$ is written
$\Delta^{\circ n}$ for $n\in\zz$ 
\item $\flow Xt{}$ designates the flow at time $t$ of the vector field
$X$.
\item If $\Psi~:~U\to V$ is locally biholomorphic, we recall its action
by conjugacy (change of coordinate)
\begin{lyxlist}{00.00.0000}
\item [{on~diffeomorphisms:}] ~
\begin{align*}
\Psi^{*}\Delta & :=\Psi^{\circ-1}\circ\Delta\circ\Psi
\end{align*}
\end{lyxlist}
on~vector~fields: ~
\begin{align*}
\Psi^{*}X & :=\DD{\Psi}^{\circ-1}\left(X\circ\Psi\right).
\end{align*}
If $\Delta$ is the time-1 map of $X$ then $\Psi^{*}\Delta$ is the
time-1 map of $\Psi^{*}X$.
\end{itemize}

\section{\label{sec:Secto_normalization}Sectorial normalization}

We recall that the functional space $\mathcal{S}$ is defined in Section~\ref{subsec:Structure}
above. The main purpose of this section is to investigate in some
details the dynamics of 
\begin{align*}
X^{\pm} & :=\frac{1}{1+X_{0}\cdot f^{\pm}}X_{0}~~~,~f\in\mathcal{S},~\norm[f]{}\leq1\\
\Delta^{\pm} & :=\flow{X^{\pm}}1{}\in\parab
\end{align*}
which we deduce from two ingredients:
\begin{itemize}
\item the dynamics of the rational vector field 
\begin{align}
X_{0}\left(z\right):= & \frac{1-z^{2}}{1+z^{2}}\times\frac{\lambda z^{2}}{1+\lambda\mu z-z^{2}}\pp z\label{eq:formal_model}
\end{align}
studied in Section~\ref{subsec:formal-model};
\item the sectorial normalization between $X^{\pm}$ and $X_{0}$ obtained
in Section~\ref{subsec:secto_normalization}, that is the existence
of biholomorphisms on $V^{\pm}\backslash\left\{ \text{small neighborhhod of the poles}\right\} $
conjugating the dynamics of both vector fields.
\end{itemize}
The dynamics of a meromorphic vector field $X$ on some domain $U$
is organized by its singular locus, which splits between the \textbf{stationary}
\textbf{points} and the \textbf{poles}
\begin{align*}
\zset X & :=\left\{ z\in U~:~X\left(z\right)=0\right\} \\
\pset X & :=\left\{ z\in U~:~X\left(z\right)=\infty\right\} .
\end{align*}
Since $f^{\pm}$ is holomorphic on $V^{\pm}$ we can never have $\frac{1}{1+X_{0}\cdot f^{\pm}}=0$
except at poles of $X_{0}$, hence 
\begin{align*}
\text{Zero \ensuremath{\left(X^{+}\right)}}\cup\text{Zero \ensuremath{\left(X^{-}\right)}} & =\text{Zero \ensuremath{\left(X_{0}\right)}}.
\end{align*}
It is also true that $z\in\pset{X^{\pm}}$ if, and only if:
\begin{itemize}
\item $z\in\pset{X_{0}}$ and $\left(f^{\pm}\right)'\left(z\right)=0$;
\item or $z\notin\pset{X_{0}}$ and $1+\left(X_{0}\cdot f^{\pm}\right)\left(z\right)=0$.
\end{itemize}
We begin with fixing some notations and recall the important Lie formula
regarding holomorphic vector fields $X$, then proceed to a short
description of the local behavior of the trajectories near singularities
and finally explain how the global dynamics of rational $X$ is weaved
from these local parts.

\subsection{\label{subsec:Basic}Basic vector fields material }

\subsubsection{Flow and time-1 map}

Let $U\subset\cbar$ be a domain and $X=R\pp z$ be a holomorphic
vector field on $U$ (that is, $R$ is a holomorphic function). For
any $z_{*}\in U$ the differential system $\dot{z}=X\left(z\right)$
has a unique local solution $t\in\left(\cc,0\right)\mapsto z\left(t\right)$
with $z\left(0\right)=z_{*}$, giving rise to a germ of a holomorphic
mapping 
\begin{align*}
\flow X{}{}~:~\neigh\times U & \longto U\\
\left(t,z_{*}\right) & \longmapsto\flow Xt{\left(z_{*}\right)}:=z\left(t\right)
\end{align*}
(to be more correct, the neighborhood $\neigh$ depends on $z_{*}$
but its size is locally uniform). The \textbf{flow} of $X$ is the
maximal multivalued analytic continuation of $\flow X{}{}$. It is
probably best understood as the reciprocal of the time function
\begin{align}
\left(z_{*},\phi\right)\in U^{2} & \longmapsto\left(z_{*},\int_{z_{*}}^{\phi}\frac{\dd z}{R\left(z\right)}\right)\in U\times\cbar,\label{eq:time_function}
\end{align}
obtained by rewriting $\dot{z}=X\left(z\right)$ as $\dd t=\frac{\dd z}{R\left(z\right)}$
and integrating the \textbf{time form} 
\begin{align}
\tau_{X} & :=\frac{\dd z}{R}.\label{eq:time_form}
\end{align}
The analytic continuation of the time function is explicit, at least
when $R$ is rational.
\begin{defn}
\label{def:time-t_map}We call $\Delta~:~z\mapsto\flow X1{\left(z\right)}$
the \textbf{time-1 map} of $X$. As a multivalued mapping over $U$,
for any path $\gamma~:~\left[0,1\right]\to U$ satisfying 
\begin{align*}
1 & =\int_{\gamma}\tau_{X}
\end{align*}
a determination of $\Delta$ satisfies the identity
\begin{align*}
\Delta\left(\gamma\left(0\right)\right) & =\gamma\left(1\right),
\end{align*}
with the special case $\Delta\left(z\right)=z$ if $X\left(z\right)=0$. 
\end{defn}

\begin{rem}
For $t\in\cc^{\times}$ the time-$t$ map $\flow Xt{}$ is obtained
as the time-1 map of $\frac{1}{t}X$.
\end{rem}

\subsubsection{Lie's formula}

One also encounters the notation $\exp\left(tX\right)$ to stand for
$\flow Xt{}$. This can be understood in the light of Lie's formula
relating the action of the operator $\exp\left(tX\right)$ and of
the flow mapping. Pick a function $g\in\holf[U]$ and some $z\in U$.
For any $t\in\neigh$ one has
\begin{align}
g\circ\flow Xt{\left(z\right)} & =\sum_{n=0}^{+\infty}\frac{t^{n}}{n!}\left(X\cdot^{n}g\right)\left(z\right)\label{eq:Lie}\\
 & =:\left(\exp\left(tX\right)\cdot g\right)\left(z\right)\nonumber 
\end{align}
where the iterated Lie's derivative is defined by
\begin{align*}
X\cdot^{n}g & :=\underset{n\text{ times}}{\underbrace{X\cdot X\cdot\left(\cdots\right)\cdot X\cdot}}g.
\end{align*}

\subsubsection{Long-time behavior}

In this paper we only use the real-time flow, \emph{i.e.} for $t\in\left(\rr,0\right)$,
and designate by $t_{\text{max}}\left(z_{*}\right)\in]0,+\infty]$
the maximal time of existence of the forward trajectory $t\in\left[0,t_{\text{max}}\left(z_{*}\right)\right[\mapsto\flow Xt{\left(z_{*}\right)}$.
Without going into too much details, a consequence of Bendixson-Poincaré
theorem and of Cauchy formula applied to~(\ref{eq:time_function})
is the following classification of asymptotic behavior for the real
flow of a meromorphic vector field $X$ on some domain $U$. Starting
from $z_{*}\notin\pset X$ the forward trajectory $t\geq0\mapsto z\left(t\right)$
matches one of the mutually exclusive outcomes: 
\begin{enumerate}
\item either $t_{\text{max}}\left(z_{*}\right)=\infty$, exactly in the
following situations:
\begin{lyxlist}{00.00.0000}
\item [{\textbf{center~case}}] $z$ is periodic (non-constant), in which
case neighboring trajectories also are (with same period) and at least
one stationary point of $X$ of center type is enclosed by the trajectory;
\item [{\textbf{equilibrium~case}}] $\lim_{t\to+\infty}z\left(t\right)\in\zset X$;
\end{lyxlist}
\item or $t_{\text{max}}\left(z_{*}\right)<\infty$, exactly in the following
situations:
\begin{lyxlist}{00.00.0000}
\item [{\textbf{escape~case}}] $\lim_{t\to t_{\text{max}}\left(z_{*}\right)}z\left(t\right)\in\partial U$;
\item [{\textbf{separation~case}}] $\lim_{t\to t_{\text{max}}\left(z_{*}\right)}z\left(t\right)\in\pset X$.
\end{lyxlist}
\end{enumerate}

\subsection{Local behavior near singularities}

\subsubsection{Dynamics near a simple zero}

Let $p\in\cbar$ be a simple stationary point of a meromorphic vector
field $X=R\pp z$, and denote by $\alpha\neq0$ the derivative of
$R$ at this point. One can find a conformal coordinate $z=\Psi\left(w\right)$
near $p$ such that $X$ takes the form 
\begin{align*}
\psi^{*}X & =W:=\alpha w\pp w\,\,,\,w\in\neigh.
\end{align*}
A direct integration describes the local behavior of the dynamics
of the linear vector field around a simple zero:
\begin{align*}
\flow Wt{\left(w\right)} & =w\ee^{\alpha t}.
\end{align*}
In particular the time-1 map of $X$ is linearizable near its fixed-point
$p$, which can be of three types:
\begin{itemize}
\item attractive ($\re{\alpha}<0$) or repulsive ($\re{\alpha}>0$) in the
equilibrium case; then there exists a domain $U:=\left(\cbar,p\right)$
such that any trajectory crossing $\partial U$ converges asymptotically
to $p$ in forward or backward time (a basin of attraction/repulsion);
\item but it can also be a center case ($\re{\alpha}=0$).
\end{itemize}
\begin{example}
When $\mu\neq0$ the formal model $X_{0}$ admits a simple stationary
point at $\pm1$ with linear part $\alpha=-\nf 1{\mu}$.
\end{example}

\subsubsection{Dynamics near a double zero}

Take now a double zero of $X$ which, for the sake of convenience,
we locate at $0$ and define the constants $a\neq0$ and $b$ by
\begin{align*}
X\left(z\right) & =\left(az^{2}+bz^{3}+\OO{z^{4}}\right)\pp z.
\end{align*}

\begin{lem}
\label{lem:.model_parab}For all $a\in\cc^{\times}$ and $b\in\cc$
we have $\flow X1{}\in\parab$. More precisely the $3$-jet of the
mapping at $0$ is given by:
\begin{align*}
\flow X1{\left(z\right)} & =z+az^{2}+\left(a^{2}+b\right)z^{3}+\OO{z^{4}}.
\end{align*}
This germ is linearly conjugate to $w\mapsto w+w^{2}+\left(1+\frac{b}{a^{2}}\right)w^{3}+\oo w$
through the transform $w:=az$, hence the formal invariant of $\flow X1{}$
is $-\nf b{a^{2}}$.\\
\end{lem}

\begin{proof}
Using Lie formula we obtain
\begin{align*}
\flow X1{\left(z\right)} & =z+\left(az^{2}+bz^{3}\right)+a^{2}z^{3}+\OO{z^{4}}.
\end{align*}
The rest is straightforward algebra, but for the fact that the formal
invariant of $w\mapsto w+w^{2}+cw^{3}+\OO{w^{4}}$ is $1-c$.
\end{proof}
\begin{example}
In the case of $X_{0}$ we have
\begin{align*}
X_{0}\left(z\right) & =\lambda z^{2}-\lambda^{2}\mu z^{3}+\OO{z^{4}}\\
\Delta_{0}\left(z\right) & =z+\lambda z^{2}+\lambda^{2}\left(1-\mu\right)z^{3}+\OO{z^{4}}\in\parab
\end{align*}
with formal invariant $\mu$ independently on $\lambda$.
\end{example}

\begin{rem}
There exists a domain $U:=\left(\cbar,p\right)$ such that any trajectory
crossing $\partial U$ converges asymptotically to $p$ in forward-
or backward-time (a parabolic basin).
\end{rem}

\subsubsection{Dynamics near a simple pole (separation case)}

Consider here the case where $X\left(p\right)=\infty$ is a simple
pole. One can find a conformal coordinate $z=\Psi\left(w\right)$
near $p$ such that $X$ takes the form 
\begin{align*}
\psi^{*}X & =\frac{1}{w}\pp w\,\,,\,w\in\neigh.
\end{align*}
A direct integration describes the local behavior of the dynamics
of the linear vector field around a simple pole.

\begin{figure}
\hfill{}\includegraphics[height=4cm]{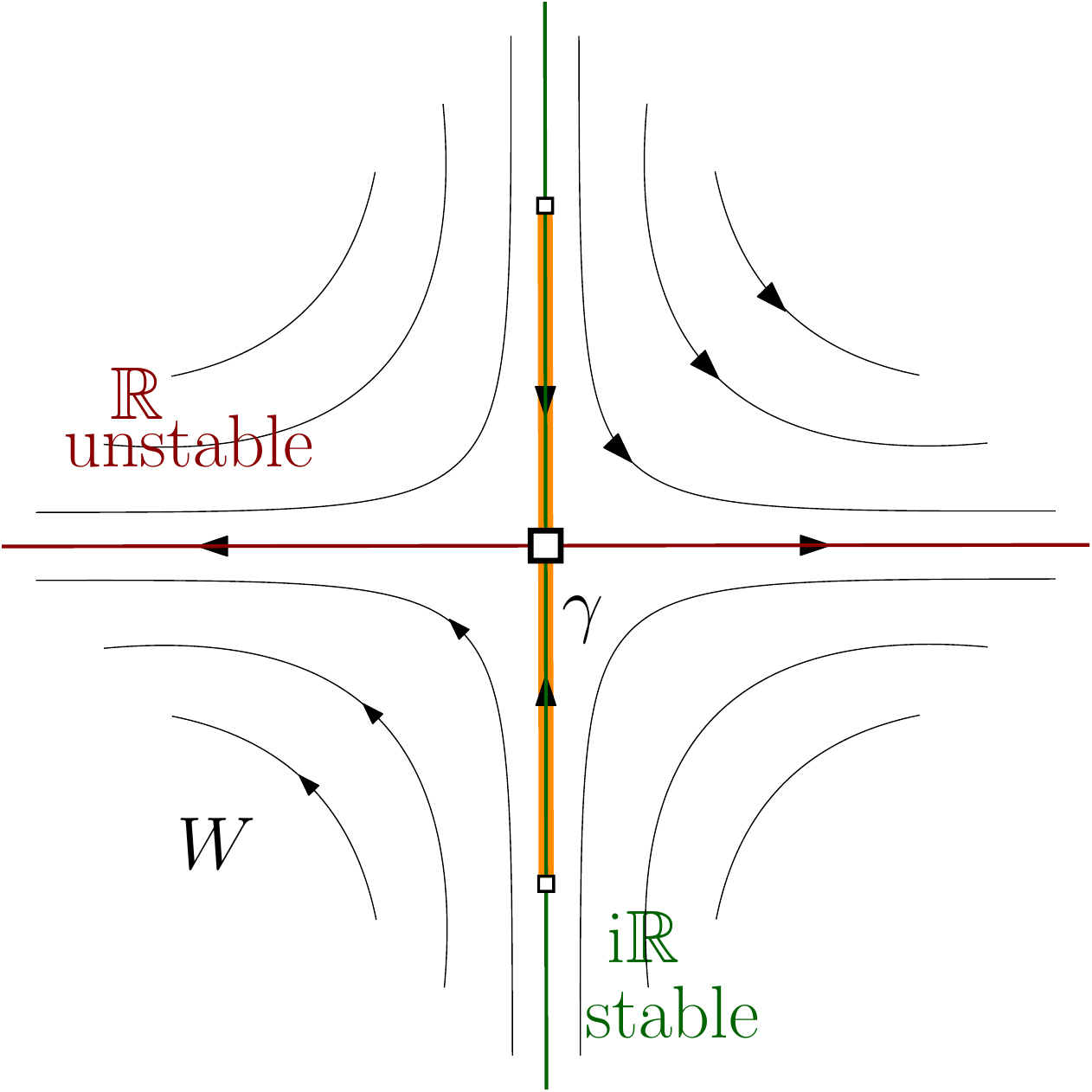}\hfill{}

\hfill{}\includegraphics[height=4cm]{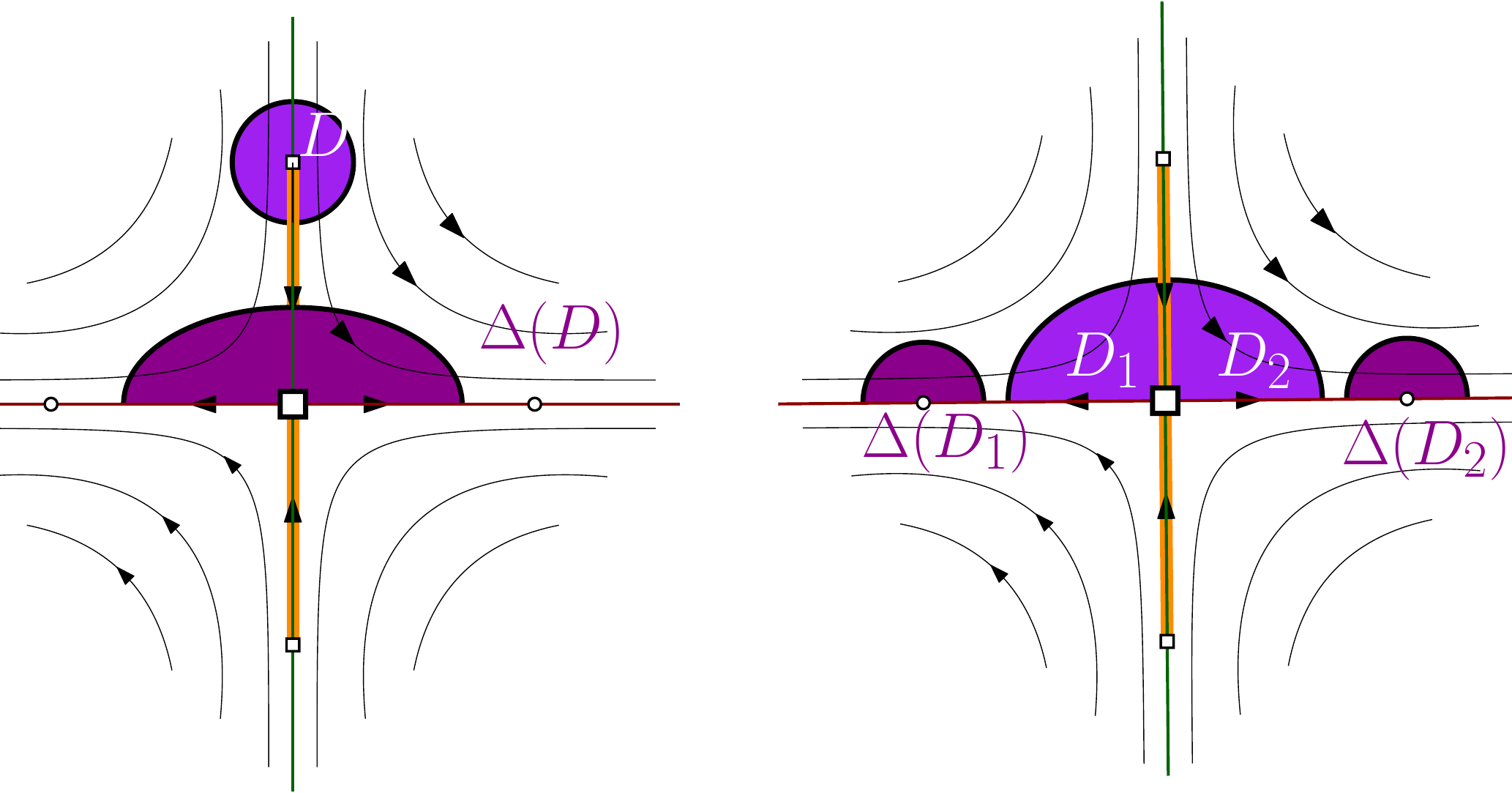}\hfill{}

\caption{\label{fig:polar_dynamics}Saddle dynamics near a simple pole of a
vector field $W$ and induced slicing dynamics of its time-1 map $\Delta$.
The latter is holomorphic on the complement of the arc $\gamma$,
included in the stable manifold and whose endpoints are sent to the
pole in time $1$.}
\end{figure}

\begin{lem}
\label{lem:polar_foliation}Let $W:=\frac{1}{w}\pp w$. We refer to
Figure~\ref{fig:polar_dynamics} for an illustration of the statement
to come.
\begin{enumerate}
\item The foliation induced by $W$ has a saddle-point at $0$ with stable
manifold $\ii\rr$ and unstable manifold $\rr$.
\item The time-1 map of $W$ is defined and holomorphic on $\cc\backslash\gamma$
with $\gamma:=\ii\sqrt{2}\left[-1,1\right]$, where it evaluates to
\begin{align*}
\flow W1{\left(w\right)} & =\pm\sqrt{2+w^{2}}.
\end{align*}
\item The analytic continuation of the above map as a multivalued function
can be realized over $\cc\backslash\left\{ \pm\ii\sqrt{2}\right\} $
with monodromy $\zsk[2]$.
\end{enumerate}
\end{lem}

\begin{example}
The vector field $X_{0}$ admits two symmetric simple poles $\pm\ii$,
fixed by $\sigma$, as well as two additional simple poles 
\begin{align*}
z_{\pm} & :=\frac{\lambda\mu\pm\sqrt{\lambda^{2}\mu^{2}+4}}{2},
\end{align*}
swapped by $\sigma$, provided $\mu\neq\pm\nf{2\ii}{\lambda}$ (otherwise
it is a double pole $z_{+}=z_{-}=\mu\in\ii\rr$) and $\mu\neq0$ (else
the <<pole>> $\pm1$ cancels out the <<zero>> of $X_{0}$ located
at $\pm1$).

Near $\pm\ii$ the vector field $X_{0}$ is conjugate to $\flow W1{}$
and this conjugacy sends $\sigma$ to the involution $w\mapsto-w$,
so that the monodromy of $\Delta_{0}$ around the ramification points
attached to $\pm\ii$ is $\sigma$ (the ramification is induced by
the change of variable $\Pi$). The corresponding local involution
near $z_{\pm}$ has a different nature and comes from the involution
$\nu\neq\id$ of the usual formal model $\frac{x^{2}}{1+\mu x}\pp x$
near the pole $-\frac{1}{\mu}$, which solves
\begin{align*}
\mu\log x-\frac{1}{x} & =\mu\log\nu-\frac{1}{\nu}
\end{align*}
(observe indeed that $t\mapsto\mu\log t-\frac{1}{t}$ has a critical
point at $-\frac{1}{\mu}$).
\end{example}

\begin{rem}
The case of a $k^{\text{th}}$-order pole $\frac{1}{w^{k}}\pp w$
is similar, with time-1 map $\flow W1{\left(w\right)}=\left(k+1+w^{k+1}\right)^{\nf 1{k+1}}$,
giving rise to $k+1$ ramification points and $k+1$ determinations
of the time-1 map. The integer $k$ is a topological invariant.
\end{rem}

\subsection{\label{subsec:formal-model}Global dynamics of the formal model}

We explain why/how we choose the formal model $X_{0}$ and the sector
$V^{\pm}$ in Section~\ref{subsec:choice_X0}. For now, let us present
the global features of the dynamics of $X_{0}$ on $\cbar$ for fixed
$\lambda>0$ and $\mu\in\cc$. We explain how the local dynamics near
the singular set of $X$ we described above stitch together by studying
the real-analytic foliation of the sphere $\cbar$ induced by the
real-time flow (Figure~\ref{fig:model_generic}).
\begin{rem}
~
\begin{enumerate}
\item Whether $\mu=0$ or not leads to different dynamics, since the simpler
$X_{0}\left(z\right)=\frac{\lambda z^{2}}{1+z^{2}}\pp z$ has less
poles and zeroes. The dynamics of this particular vector field has
been studied in Example~\ref{exa:mu_zero}, so we allow ourselves
to assume that $\mu\neq0$ whenever it leads to a more straightforward
exposition.
\item In order to avoid other zero/pole cancellations we must in addition
suppose that $\mu\neq\pm\nf{2\ii}{\lambda}$ (the latter is ensured
whenever $\lambda<\frac{1}{2\left|\mu\right|}$). Under these assumptions,
the rational vector field has three stationary points located at $0$
(double) and $\pm1$ (simple), as well as four simple poles located
at $z_{\pm}:=\frac{\lambda\mu\pm\sqrt{\lambda^{2}\mu^{2}+4}}{2}$
and $\pm\ii$.
\end{enumerate}
\end{rem}

\subsubsection{Spinal graph }

\begin{figure}
\hfill{}\includegraphics[height=5cm]{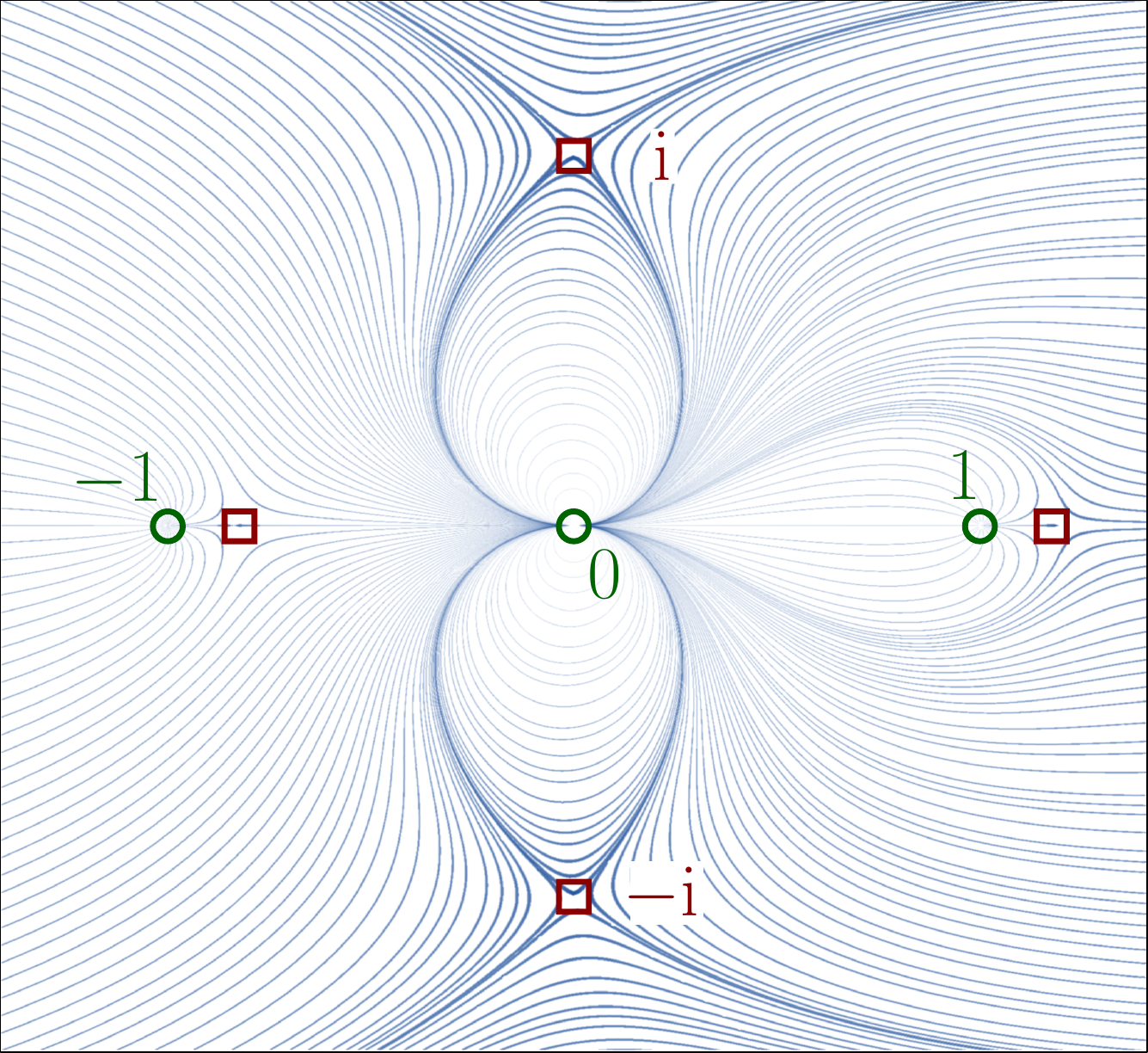}\hfill{}\includegraphics[height=5cm]{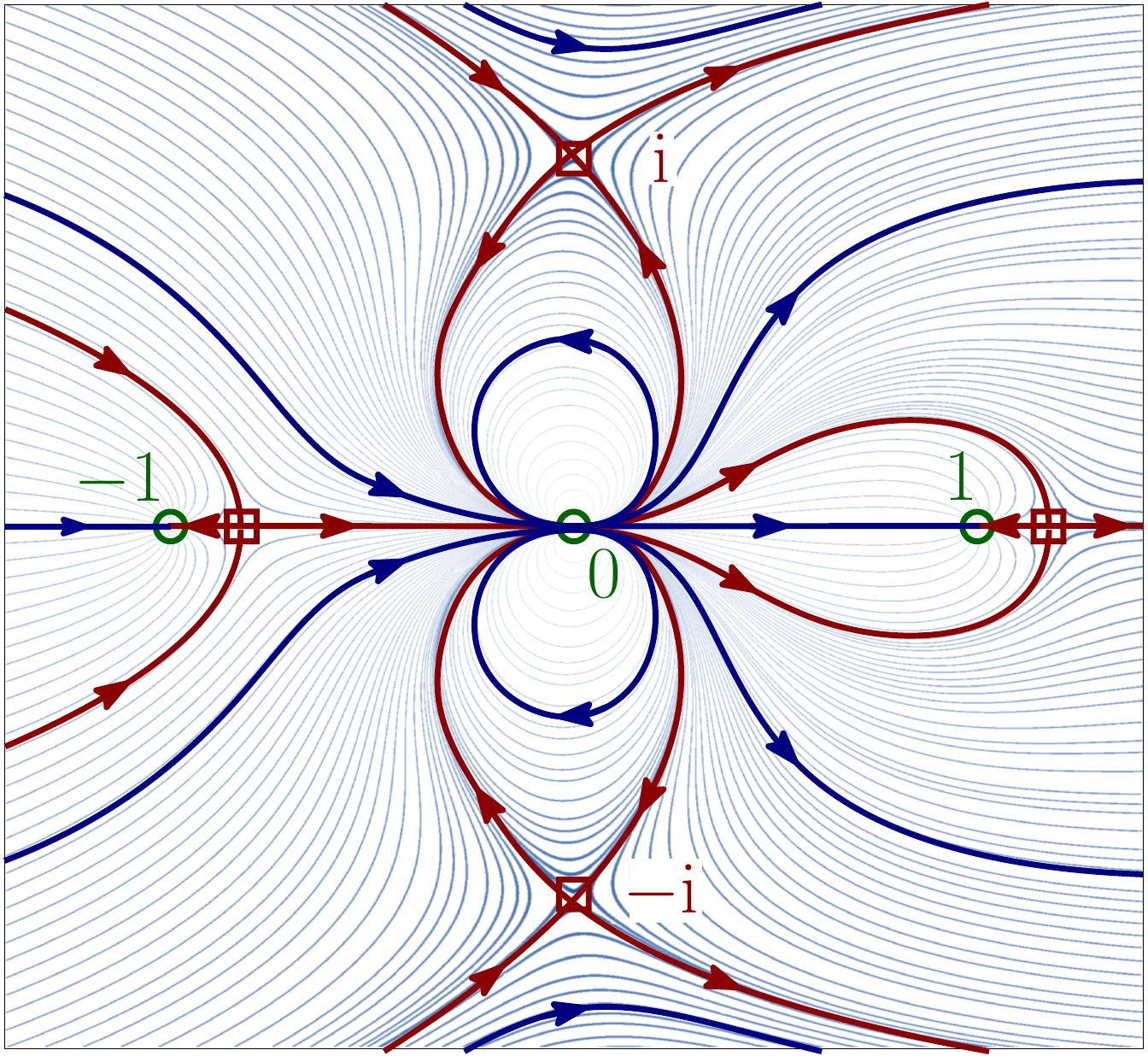}\hfill{}

\caption{\label{fig:model_generic}Foliation induced by the real-time flow
of $X_{0}$ (left) for $\lambda:=\frac{1}{2}$ and $\mu:=\frac{1}{2}$,
revealing the three stationary points (green circles) and four poles
(red squares). On the right, the spinal (blue) and separatrix (red)
graphs are depicted.}
\end{figure}

After the pioneering work of S.~\noun{Smale} and \emph{al} starting
at the beginning of the 1980's to study general numerical Newton-like
schemes for solving polynomial equations~\cite{SmallAlg,ShuTiWil},
the foliations induced by holomorphic polynomial vector fields $X$
has been thoroughly studied in the generic case by A.~\noun{Douady},
F.~\noun{Estrada} and P.~\noun{Sentenac} in an unpublished manuscript~\cite{DES}.
Then B.~\noun{Branner} and K.~\noun{Dias}~\cite{BraDia} completed
the task for every polynomial vector field, and in his thesis J.~\noun{Tomasini}
studied some rational vector fields. All these considerations result
in the following statement: the topological class of $X$ (up to orientation-preserving
homeomorphism) is completely classified by its combinatorial dynamics.
Roughly speaking, it is encoded as the spinal graph of $X$, the way
poles and stationary points of the vector field are connected by the
closure of maximal trajectories.
\begin{lem}
Let $X$ be a rational vector field. The \textbf{separatrix graph}
$\sep X$ of $X$ is the closure of the union of all stable and unstable
manifolds passing through the poles of $X$. Then $\cbar\backslash\sep{X_{}}$
consists in finitely many $X$-invariant connected components. Any
two trajectories in the same component are:
\begin{itemize}
\item either both periodic with same period;
\item or they link the same pair of stationary points $q\to p$ of $X$.
\end{itemize}
\end{lem}

\begin{proof}
Assume that $t\mapsto z\left(t\right)$ is a periodic trajectory.
Cauchy formula indicates that nearby trajectories are also periodic
with same period. A straightforward connectedness argument allows
us to conclude that every other trajectory in the component is periodic
with same period.

Assume now that $t\mapsto z\left(t\right)$ is neither a center nor
a separatrix. Then $z$ goes towards some $p\in\zset X$ in forward
time, say, and then there exists a domain $U_{p}\ni p$ (be it the
basin of attraction of an attractive fixed-point or a parabolic basin
for a multiple zero) such that any trajectory landing at $p$ eventually
hits $\partial U_{p}$ and \emph{vice versa}. Hence $z$ links $\partial U_{q}$
to $\partial U_{p}$ in finite time, and the flow-box theorem asserts
this remains the case for neighboring trajectories. Again a direct
connectedness argument brings the conclusion.
\end{proof}
\begin{defn}
We define the \textbf{spinal graph} $\spine X$ as the oriented graph
with vertices $\zset X$ and ordered edges corresponding to a trajectory
of $X$ linking those vertices $q\to p$, one per component of $\cbar\backslash\sep{X_{}}$.
Isolated vertices correspond to center stationary points.
\end{defn}

\begin{rem}
Both graphs $\spine X$ and $\sep X$ come with a canonical, non-ambiguous
geometric realization with edges as integral curves of $X$. For that
reason we identify the combinatorial data and its canonical geometric
realization as a subset of $\cbar$.\\
We can equip both graphs with an oriented length, by integrating the
time form~(\ref{eq:time_form}) along injective subpaths $\gamma$
\begin{align}
\ell_{X}\left(\gamma\right) & :=\int_{\gamma}\tau_{X}=\int_{\gamma}\frac{\dd z}{X\cdot\id}\in\rbar.\label{eq:arc_length}
\end{align}
\end{rem}

We omit the proof of the next lemma.
\begin{lem}
If $\mu\neq0$ and $\mu\neq\pm\frac{2\ii}{\lambda}$ the graph $\spine{X_{0}}$
belongs to the following list, according to the sign of $\re{\mu}$.
A center bifurcation occurs when $\re{\mu}=0$. The separatrix graph
is figured in red while the spinal graph is colored blue.

\hfill{}\includegraphics[width=0.3\columnwidth]{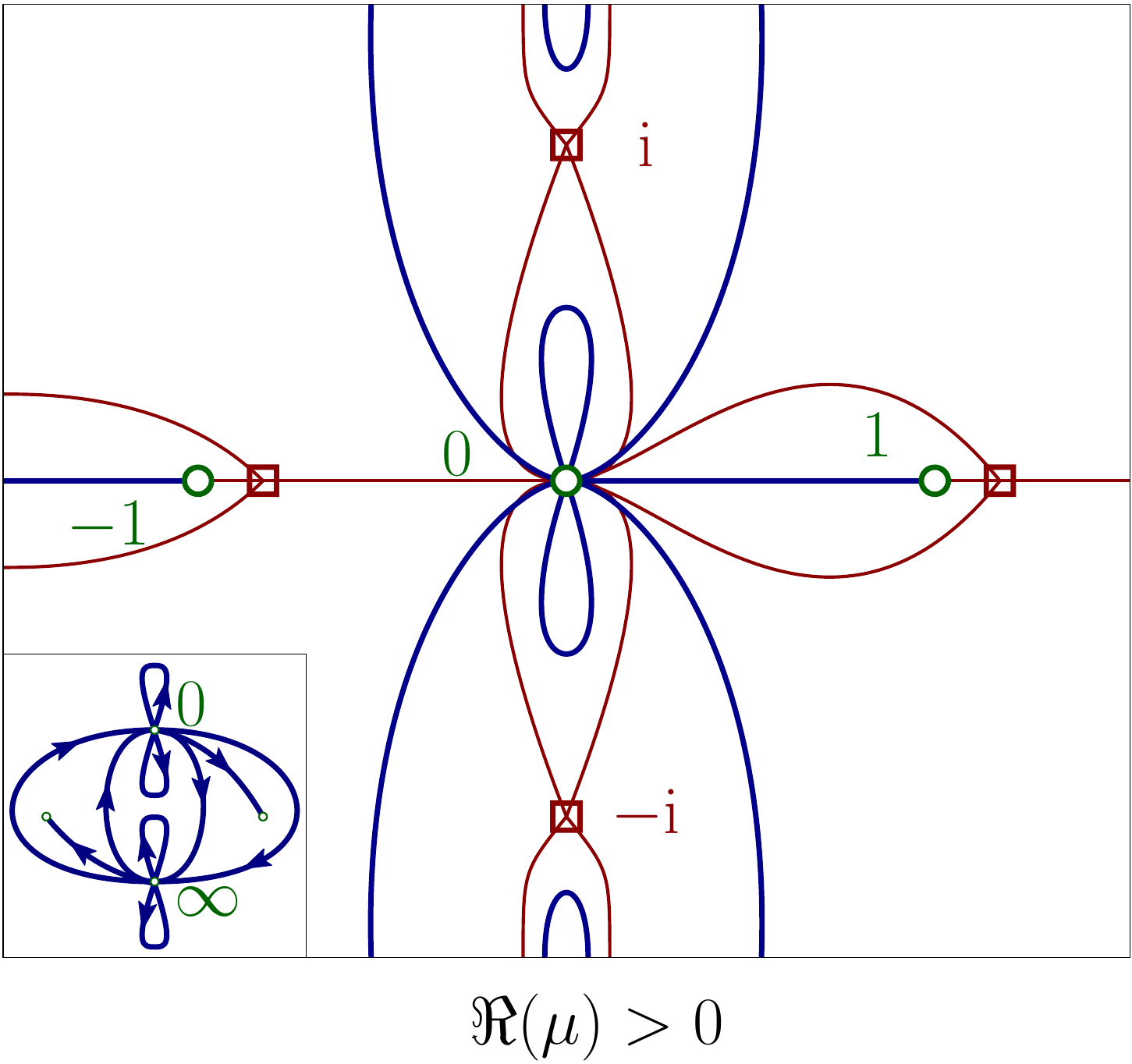}\hfill{}\includegraphics[width=0.3\columnwidth]{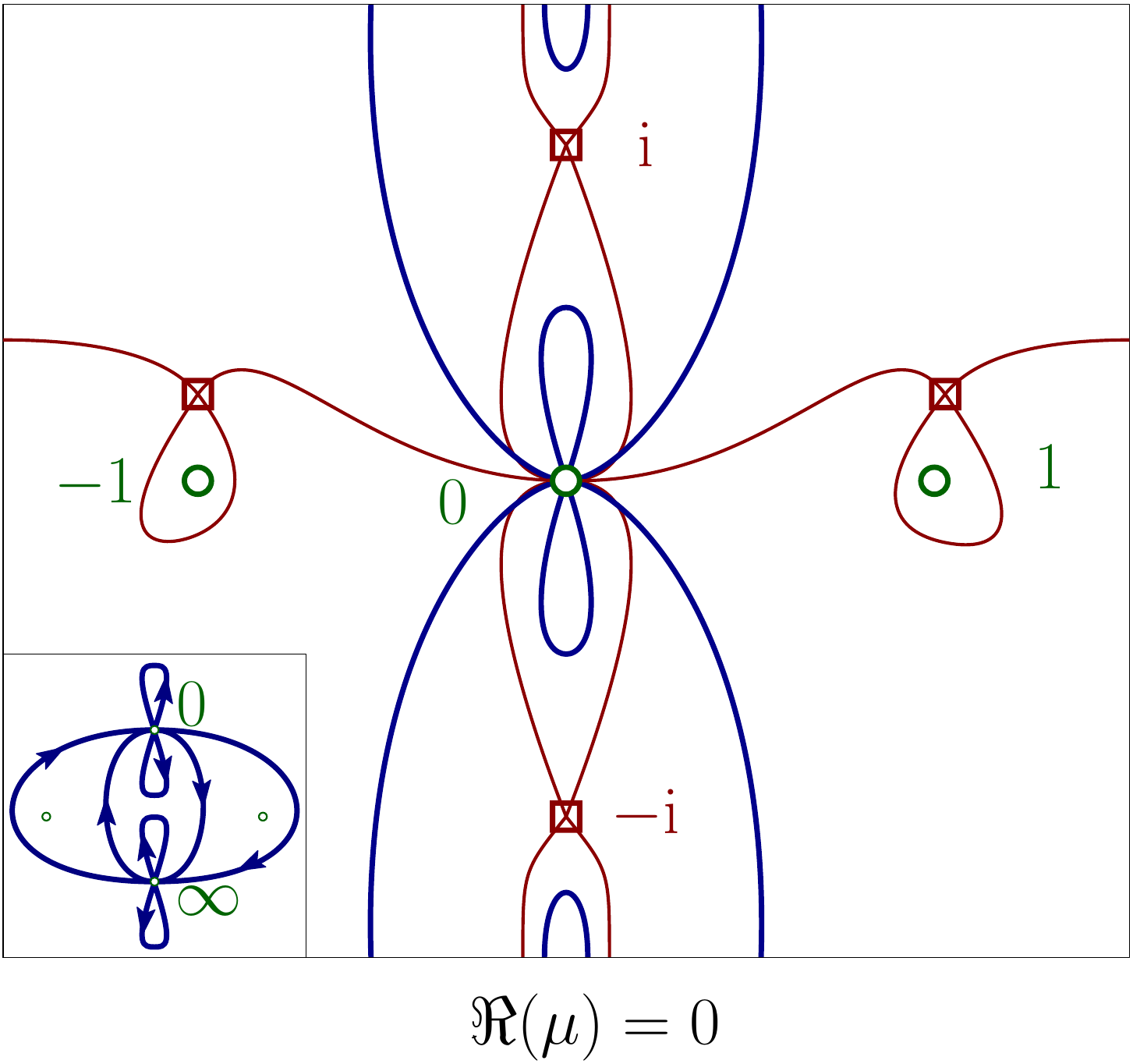}\hfill{}\includegraphics[width=0.3\columnwidth]{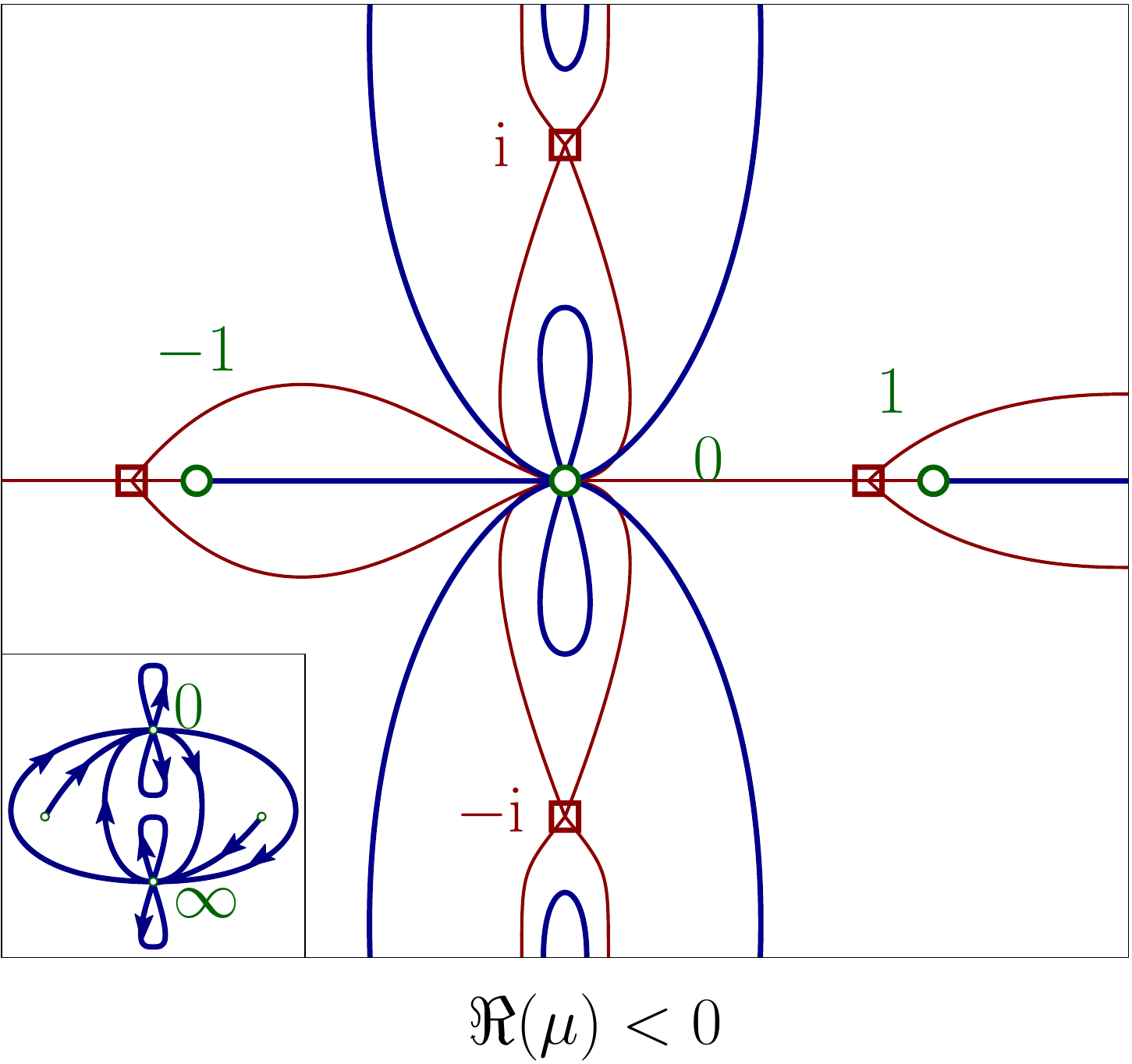}\hfill{}~
\end{lem}

\subsubsection{Dynamics of the time-1 map $\Delta_{0}$}

From the description of the qualitative dynamics of time-1 maps of
rational vector fields, we work out more quantitative specifics for
the model $X_{0}$.
\begin{lem}
\label{lem:time-1_rational}Let $X$ be a rational vector field and
$\Delta$ be its time-1 map. Then $\Delta$ is a multivalued map over
$\cbar$ with branch points $z_{*}\in\sep X$ at time $1$ from an
element of $\pset X$, by which we mean that there exists a subpath
$\gamma\subset\sep X$ with $\gamma\left(0\right)=z_{*}$, $\gamma\left(1\right)\in\pset X$
and $\ell_{X}\left(\gamma\right)=1$ as in~(\ref{eq:arc_length}).
A pole of order $k$ gives rise to $k+1$ local determinations.
\end{lem}

\begin{proof}
For $\Psi$ to be locally holomorphic at $z\in U$ there must exist
a path $\gamma~:~\left[0,1\right]\longto\cc$ with $\gamma\left(0\right)=0$
and $\gamma\left(1\right)=f\left(z\right)$ on a neighborhood of which
the germ $t\in\neigh\mapsto\flow Xt{\left(z\right)}$ admits an analytic
continuation. This is clearly the case as long as $\flow Xt{\left(z\right)}\notin\pset X$.
If $t\neq\gamma\left(1\right)$ but $t\in\pset X$ then one can slightly
deform $\gamma$ to avoid the pole (the different nonequivalent choices
of the deformation providing differing determinations of $\Psi$).
Hence $\Psi$ can be analytically continued around $z$ as long as
$\text{\ensuremath{\flow X{\gamma\left(1\right)}{\left(z\right)=}}}\Psi\left(z\right)\notin\pset X$.
If the trajectory issued from $z_{*}\in\cc$ reaches some pole $p$
in finite time $t_{\text{max}}\left(z_{*}\right)\in\rr>0$ then $z_{*}$
belongs to a stable manifold passing through the pole $p$. It follows
from Lemma~\ref{lem:polar_foliation} that there is only one such
smooth manifold in the neighborhood of $p$, thus $\Gamma$ as described
is a well-defined closed set of the sphere. Hence, whenever $z_{*}\notin\Gamma$
we can define a locally analytic $\Delta_{0}$ near $z_{*}$, and
that mapping can be extended to $\partial\Gamma$ continuously by
setting $\Delta_{0}\left(\partial\gamma\left(p\right)\right):=\left\{ p\right\} $.
\end{proof}
\begin{prop}
\label{prop:X0_dynamics}Assume $\mu\neq0$ and $0<\lambda<\frac{1}{2\left|\mu\right|}$
small enough.
\begin{enumerate}
\item The time-1 map $\Delta_{0}$ of $X_{0}$ is holomorphic and injective
on the dense domain $\overline{\cc}\backslash\Gamma=\neigh\cup\left(\cbar,\infty\right)$,
where $\Gamma=\bigcup_{p\in\pset{X_{0}}}\gamma\left(p\right)$ is
the union of $4$ real-analytic, forward $X_{0}$-invariant and smooth
curves passing through the poles of $X_{0}$. Each one of these curves
is the arc of the stable manifold of $X_{0}$ through $p$ joining
the two points mapped to $p$ in time 1 along $X_{0}$.
\item The holed out sphere
\begin{align*}
\mathcal{D}_{\lambda}:=\cbar\backslash\left(D_{-\ii}\cup D_{\ii}\cup D_{z_{-}}\cup D_{z_{+}}\right)
\end{align*}
is included in $\cbar\backslash\Gamma$, where $D_{p}$ is an open
disc centered at the pole $p$ whose diameter decreases to $0$ as
$\lambda$ does. In fact $\flow{X_{0}}{\tau}{}$ is holomorphic on
$\mathcal{D}_{\lambda}$ for any $\tau\in\dbar$.
\item $\Delta_{0}$ extends as a multivalued function over $\cc\backslash\bigcup_{p\in\pset{X_{0}}}\partial\gamma\left(p\right)$
with involutive monodromy around each one of the poles $p$. The endpoints
$\partial\gamma\left(p\right)$ of the arc $\gamma\left(p\right)$
are mapped to $p$ by $\Delta_{0}$.
\end{enumerate}
In case $\mu=0$ the result still holds save for the fact that the
pole at $z_{\pm}=\pm1$ cancels out the stationary point at $\pm1$.
There only remain the parabolic fixed-points of $\Delta_{0}$ and
their ramification locus coming from the two poles $\left\{ \pm\ii\right\} $
of $X_{0}$.
\end{prop}

\begin{rem}
~
\begin{enumerate}
\item We give quantitative bounds on the smallness of $\lambda$ and $D_{p}$
in Section~\ref{subsec:Quantitative-bounds}.
\item Whatever the value of $\lambda>0$ and $\left|\tau\right|\leq1$,
the flow $\flow{X_{0}}{\tau}{}$ is holomorphic on $\neigh$ since
$0$ can be reached only in infinite time.
\item The complement in $\cbar$ of the closure of all stable manifolds
of $X_{0}$ through its poles is an open and dense forward $X_{0}$-invariant
set $U$, \emph{i.e.} on which $\Delta_{0}$ can be (forward-)iterated
\emph{ad lib}. It is not a neighborhood of $0$ nor of $\infty$.
\end{enumerate}
\end{rem}

\begin{proof}
~
\begin{enumerate}
\item The holomorphy of $\Delta_{0}$ outside $\Gamma$ is simply the content
of Lemma~\ref{lem:time-1_rational}. Observe next that only at a
pole $p$ can two trajectories meet in finite time. As a matter of
consequence, if $\Delta_{0}\left(z_{0}\right)=\Delta_{0}\left(z_{1}\right)=:p$
then the real trajectories of $X_{0}$ issued from $z_{0}$ and $z_{1}$
must cross each other at time 1, hence $p$ is a pole of $X_{0}$
and $z_{0},\,z_{1}$ belong to $\partial\Gamma$. In other words $\Delta_{0}$
is injective outside $\Gamma$. 
\item We wish to bound the magnitude of $\flow{X_{0}}{\tau}{\left(z_{*}\right)}$
for any $z_{*}$ close to a pole $p$ of $X_{0}$ and $\tau\in\sone$.
It is clear that the connected component of $\ell_{X_{0}}^{-1}\left(\ww D\right)$
containing $p$ shrinks to $0$ as $\lambda$ does, one can then take
$D_{p}$ containing that component. More details are given in Section~\ref{subsec:Quantitative-bounds}.
\item $X_{0}$ is conjugate to $\frac{1}{w}\pp w$ as in Lemma~\ref{lem:polar_foliation}
on a full neighborhood of each $\gamma\left(p\right)$. For a detailed
proof we refer to the proof of Proposition~\ref{prop:sectorial_dynamics}.
\end{enumerate}
\end{proof}

\subsection{\label{subsec:secto_normalization}Sectorial normalization of $X^{\pm}$}

Let us reformulate the previous study for variable-time flow $\flow Xf{}$.
\begin{prop}
\label{prop:flow_change_variable}Let $f$ be a function holomorphic
on a domain $U\subset\cc$ and $X$ be a meromorphic vector field
on $\cc$. Define the meromorphic vector field on $U$ by
\begin{align*}
X_{f} & :=\frac{1}{1+X\cdot f}X.
\end{align*}
\begin{enumerate}
\item $X$ and $X_{f}$ share the same stationary points while
\begin{align*}
\pset{X_{f}} & =\pset X\cap\left(f'\right)^{-1}\left(0\right)~~\cup~~\left(1+X\cdot f\right)^{-1}\left(0\right)\backslash\pset X.
\end{align*}
\item If the time-$f$ flow $\Psi~:~z\mapsto\flow X{f\left(z\right)}{\left(z\right)}$
along $X$ is locally holomorphic around some $z\in U$ then:
\begin{enumerate}
\item ~
\begin{align*}
\frac{X\cdot\Psi}{1+X\cdot f} & =X\circ\Psi
\end{align*}
(this particularly means that if $\Psi$ is locally biholomorphic
at $z\in U$ then $\Psi^{*}X=X_{f}$ around $z$);
\item ~
\begin{align*}
\Psi\circ\Delta_{f} & =\Delta\circ\Psi
\end{align*}
where $\Delta_{f}$ is the time-1 map of $X_{f}$ and $\Delta$ is
that of $X$ (this particularly means that if $\Psi$ is locally biholomorphic
at $z\in U$ then $\Psi^{*}\Delta=\Delta_{f}$ around $z$).
\end{enumerate}
\item $\Psi$ is locally holomorphic at all $z\in U$ except maybe for $z\in\pset{X_{f}}$.
In particular $\Psi$ is holomorphic on a neighborhood of $\zset X$
and $\Psi|_{\zset X}=\id$.
\end{enumerate}
\end{prop}

\begin{proof}
~
\begin{enumerate}
\item If $p\in\pset X$ then $X_{f}\left(p\right)=\frac{1}{f'\left(p\right)}\pp z$,
else a pole in $X_{f}$ can only come from a zero of $1+X\cdot f$.
\item ~
\begin{enumerate}
\item Using Lie's formula $\Psi=\sum_{n=0}^{\infty}\frac{f^{n}}{n!}X\cdot^{n}\id$
we derive formally
\begin{align*}
X_{f}\cdot\Psi & =\frac{1}{1+X\cdot f}X\cdot\sum_{n=0}^{\infty}\frac{f^{n}}{n!}X\cdot^{n}\id\\
 & =\frac{1}{1+X\cdot f}\left(\sum_{n=0}^{\infty}\frac{f^{n}}{n!}X\cdot^{n+1}\id+\left(X\cdot f\right)\frac{f^{n-1}}{\left(n-1\right)!}X\cdot^{n}\id\right)\\
 & =\sum_{n=0}^{\infty}\frac{f^{n}}{n!}X\cdot^{n+1}\id\\
 & =X\circ\Psi.
\end{align*}
\item The Lie formula again implies that
\begin{align*}
\Psi\circ\Delta_{f}=\Psi\circ\flow{X_{f}}1{} & =\sum_{n=0}^{\infty}\frac{1}{n!}X_{f}\cdot^{n}\Psi
\end{align*}
but a direct recursion yields $X_{f}\cdot^{n}\Psi=\left(X\cdot^{n}\id\right)\circ\Psi$.
\end{enumerate}
\item See Proposition~\ref{prop:X0_dynamics} and subsequent remark.
\end{enumerate}
\end{proof}
\begin{rem}
Item~2. remains true for formal $f\in z\frml z$ and formal $X\in z\frml z\pp z$
(in that case $\Psi$ is invertible as a formal transform at $0$).
\end{rem}

Following the previous discussion we define for $f\in\mathcal{S}$
with $\norm[f]{}\leq1$
\begin{align*}
\Psi^{\pm} & ~:~z\longmapsto\flow{X_{0}}{f^{\pm}\left(z\right)}{\left(z\right)}.
\end{align*}
Whenever $\Psi^{\pm}$ is locally invertible we have 
\begin{align*}
\left(\Psi^{\pm}\right)^{*}X_{0} & =X^{\pm}
\end{align*}
by Proposition~\ref{prop:flow_change_variable}~1 with $X:=X_{0}$.
\begin{lem}
\label{lem:secto_normalization}Invoking the notations of Lemma~\ref{prop:X0_dynamics},
the mapping $\Psi^{\pm}$ is locally biholomorphic on the holed out
domain $V^{\pm}\cap\mathcal{D}_{\lambda}$ and therefore conjugate
$X_{0}$ with $X^{\pm}=X_{f^{\pm}}$.
\end{lem}

\begin{proof}
Because $\left|f^{\pm}\left(z\right)\right|\leq1$ for all $z\in V^{\pm}$
the construction of $\mathcal{D}_{\lambda}$ guarantees that $\flow{X_{0}}{f^{\pm}}{}$
is well-defined and holomorphic on $V^{\pm}\cap\mathcal{D}_{\lambda}$.
But the poles of $X_{0}$ lie outside $\mathcal{D}_{\lambda}$ therefore
the conclusion follows from Proposition~\ref{prop:flow_change_variable}~3.
\end{proof}
From this Lemma we deduce that the dynamics of $X^{\pm}$ (governed
by its <<sectorial spinal graph>>) is very close to that of $X_{0}|_{V^{\pm}}$.
A more quantitative analysis is conducted in Sections~\ref{sec:Globalization-Theorem}
and~\ref{subsec:Parabolic}.

\section{\label{sec:Reduction}Reduction}

The proof of the next proposition, which reduces the problem of finding
$\Delta$ to that of finding the pair $\left(f^{+},f^{-}\right)$
in orbits space, is performed in Section~\ref{subsec:Reduction}.
The definition of the functional space $\mathcal{S}$ is given in
Definition~\ref{def:function_space} and the orbits space of the
model vector field $X_{0}$ is studied in details in Section~\ref{subsec:model_primitive_function}
below.
\begin{defn}
~
\begin{enumerate}
\item We say that a holomorphic $L~:~U\to\cc$ is a \textbf{first-integral}
of $\Delta$ when 
\begin{align*}
L\circ\Delta & =L.
\end{align*}
\item A function $H\neq0$ holomorphic on $V$ is a \textbf{primitive function}
of $X$ whenever 
\begin{align*}
X\cdot H & =2\ii\pi H.
\end{align*}
\end{enumerate}
\end{defn}

Primitive functions $H$ (actually unique up to a multiplicative constant)
play a central role in the present study, because the Cousin problem~$\left(\star\right)$
below comes from the structure of the ring of first-integrals of the
map $\Delta$. When $U$ is a component of $V^{\cap}$ the primitive
function $H$ turns out to be a functional generator of the ring (Corollary~\ref{cor:primitive_first-integral}):
every first-integral $L$ of $\Delta$ factors holomorphically as
$\phi\circ H$. The latter property can be reworded equivalently as:
$H$ gives a preferred coordinate on the space of orbits of $\Delta$.
\begin{rem}
A primitive function $H$ of $X$ is indeed a first-integral of $\Delta$
(Lie formula):
\begin{align*}
H\circ\Delta & =\sum_{n=0}^{+\infty}\frac{1}{n!}X\cdot^{n}H=\left(\sum_{n=0}^{+\infty}\frac{\left(2\ii\pi\right)^{n}}{n!}\right)H=H,
\end{align*}
hence any function of the form $\phi\circ H$ is a first-integral
of $\Delta$.
\end{rem}

\begin{prop}
\label{prop:reduction}Take $f\in\mathcal{S}$ and fix $\lambda>0$.
Let 
\begin{align*}
X_{}^{\pm} & :=\frac{1}{1+X_{0}\cdot f^{\pm}}X_{0}
\end{align*}
 and $\Delta^{\pm}$ be the time-1 map of $X_{}^{\pm}$. We write
$\Delta_{0}$ for the time-1 map of $X_{0}$.
\begin{enumerate}
\item ~
\begin{enumerate}
\item $\Delta^{\pm}$ is holomorphic and bounded on $V^{\pm}\cap\neigh$
and $\Delta^{\pm}\left(z\right)=\Delta_{0}\left(z\right)+\oo{z^{3}}$
near $0\in\adh{V^{\pm}}$.
\item Its space of orbits over $V^{\pm}\cap\neigh$ is canonically given
by the range of the primitive function $H_{}^{\pm}:=H_{0}\exp\left(2\ii\pi f^{\pm}\right)$,
where
\begin{align*}
H_{0}\left(z\right) & :=\left(\frac{\lambda z}{1-z^{2}}\right)^{2\ii\pi\mu}\exp\left(-2\ii\pi\frac{1-z^{2}}{\lambda z}\right)
\end{align*}
is a primitive function of $X_{0}$. Moreover 
\begin{align*}
\quot{{W^{\pm}}\cap{\neigh}}{\Delta^{\pm}} & \simeq H_{}^{\pm}\left(V^{\pm}\cap\neigh\right)=\cc^{\times}.
\end{align*}
\item For $\sharp\in\left\{ 0,\infty\right\} $ one has
\begin{align*}
\lim_{\begin{array}{c}
z\to0\\
z\in\isect
\end{array}}H_{}^{\pm}\left(z\right) & =\sharp,
\end{align*}
so that $H_{}^{\pm}\left(\adh{\isect}\cap\neigh\right)=\left(\cbar,\sharp\right)$.
\end{enumerate}
\item The following properties are equivalent.
\begin{enumerate}
\item $\Delta^{\pm}$ is the restriction to $V^{\pm}\cap\neigh$ of a parabolic
germ in $\parab$;
\item $\Delta^{+}=\Delta^{-}$ on $V^{\cap}\cap\neigh$;
\item the difference $f^{-}-f^{+}$ is a first-integral of $\Delta^{+}$;
\item there exists $\isect[\varphi]\in\holf[\cbar,\sharp]$, $\sharp\in\left\{ 0,\infty\right\} $,
such that 
\begin{align*}
\begin{cases}
f^{-}-f^{+} & =\zero[\varphi]\circ H_{f}^{+}\text{~~~~ on }\zero\cap\neigh\\
f^{-}-f^{+} & =\infi[\varphi]\circ H_{f}^{-}\text{~~~~ on }\infi\cap\neigh
\end{cases} & .\tag{\ensuremath{\star}}
\end{align*}
\end{enumerate}
\item Assume here that~$\left(\star\right)$ holds, so that $\Delta^{\pm}=\Delta\in\parab$.
\begin{enumerate}
\item The modulus of $\Delta$ is $\ev\left(\Delta\right)=\left(\id\exp\left(4\pi^{2}\mu+\zero[\varphi]\right),\id\exp\infi[\varphi]\right)$.
\item $\left(f^{+},f^{-}\right)$ is the 1-sum of a formal power series
$F\in z\frml z$ and $\Delta$ is the time-1 map of $\frac{1}{1+X_{0}\cdot F}X_{0}$.
\end{enumerate}
\end{enumerate}
\end{prop}

\subsection{\label{subsec:model_primitive_function}Primitive function and space
of orbits near $0$}

\begin{figure}
\hfill{}\includegraphics[width=3.5cm]{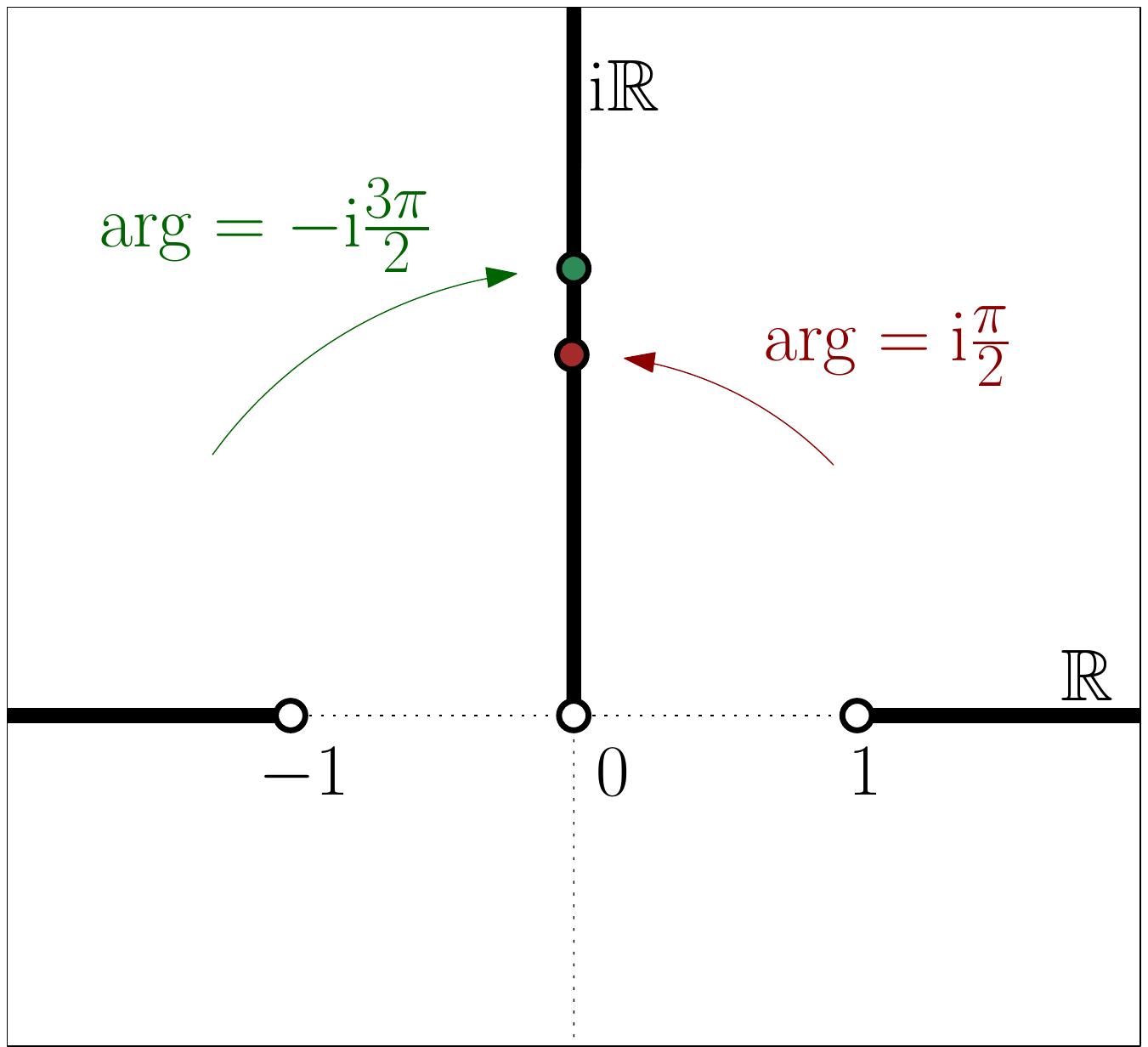}\hfill{}

\hfill{}

\caption{\label{fig:H0_branchcut}Branch-cut scheme of $H_{0}$ (cuts along
fat lines).}
\end{figure}

Define, by analytic continuation, the multivalued function over $\cc\backslash\left\{ 0,\pm1\right\} $
\begin{align*}
H_{0}\left(z\right): & =\exp\left(2\ii\pi t\left(z\right)\right)=\left(\frac{\lambda z}{1-z^{2}}\right)^{2\ii\pi\mu}\exp\left(-2\ii\pi\frac{1-z^{2}}{\lambda z}\right),
\end{align*}
with branch-cuts and determination indicated in Figure~\ref{fig:H0_branchcut}.
By construction it is a primitive function of $X_{0}$:
\begin{align*}
X_{0}\cdot H_{0} & =2\ii\pi H_{0}
\end{align*}
and the real-flow of $X_{0}$ corresponds to level curves of $\left|H_{0}\right|$.
\begin{defn}
\label{def:cut_sectors}The \textbf{cut sector} $\widehat{V}$ is
the complement
\begin{align*}
\widehat{V}^{\pm} & :=V^{\pm}\backslash\pm\rr_{\geq1}
\end{align*}
on which $H_{0}$ is given the holomorphic determination induced by
Figure~\ref{fig:H0_branchcut} and written $H_{0}|_{\widehat{V}}$.
Notice that 
\begin{align*}
\widehat{V}^{-}\cap\widehat{V}^{+} & =V^{-}\cap V^{+}=V^{\cap}=\zero\cup\infi.
\end{align*}
\end{defn}

For every $z\in\widehat{V}$ such that $\Delta_{0}\left(z\right)\in\widehat{V}$
one has 
\begin{align*}
H_{0}\left(\Delta_{0}\left(z\right)\right) & =H_{0}\left(z\right).
\end{align*}
Moreover, the orbits of $\Delta_{0}|_{\widehat{V}}$ are in 2-to-$1$
correspondence with level sets of $H_{0}|_{\widehat{V}}$, but the
correspondence becomes 1-to-1 on $\neigh$. As $z\to0$ in $V^{\pm}$
we have
\begin{align*}
H_{0}\left(z\right) & \sim_{0}\left(\lambda z\right)^{2\ii\pi\mu}\ee^{\nf{-2\ii\pi}{\left(\lambda z\right)}}
\end{align*}
hence
\begin{align}
\begin{cases}
\left|H_{0}\left(z\right)\right| & =\OO{\ee^{-\frac{1}{\left|*z\right|}}}\,\text{as }z\underset{\zero[V]}{\longto}0\\
\left|\frac{1}{H_{0}\left(z\right)}\right| & =\OO{\ee^{-\frac{1}{\left|*z\right|}}}\,\text{as }z\underset{\infi[V]}{\longto}0
\end{cases} & .\label{eq:estim_H0_0}
\end{align}
This particularly shows that $H_{0}\left(\widehat{V}^{\pm}\right)=\cc^{\times}$:
the sectorial space of orbits near $0$ of $\Delta_{0}$ is a doubly-punctured
sphere,
\begin{align}
\quot{{V^{\pm}}\cap{\neigh}}{\Delta_{0}} & \simeq\cc^{\times},\label{eq:model_secto_orbit_space}
\end{align}
the whole sphere $\cbar$ being obtained by taking the closure $\adh{\isect}\cap\neigh$
for $\sharp\in\left\{ 0,\infty\right\} $.
\begin{rem}
As $z\to\infty$ in $V^{\pm}$ we have
\begin{align*}
H_{0}\left(z\right) & \sim_{\infty}\left(-\nf z{\lambda}\right)^{-2\ii\pi\mu}\ee^{2\ii\pi\nf z{\lambda}}
\end{align*}
so that
\begin{align}
\begin{cases}
\left|H_{0}\left(z\right)\right| & =\OO{\ee^{-\left|*z\right|}}\,\text{as }z\underset{\zero[V]}{\longto}\infty\\
\left|\frac{1}{H_{0}\left(z\right)}\right| & =\OO{\ee^{-\left|*z\right|}}\,\text{as }z\underset{\infi[V]}{\longto}\infty
\end{cases} & .\label{eq:estim_H0_inf}
\end{align}
\end{rem}

The (direct) monodromy of $H_{0}$ around $0$ is generated by the
linear map
\begin{align*}
h & \mapsto\ee^{-4\pi^{2}\mu}h.
\end{align*}
Therefore, for the choice of the determination and branch-cuts we
made, we have
\begin{align}
\begin{cases}
H_{0}|_{V^{-}}=\ee^{4\pi^{2}\mu}H_{0}|_{V^{+}}=:\zero[\psi]\circ H_{0}|_{V^{+}} & \text{on }V^{0}\\
H_{0}|_{V^{-}}=H_{0}|_{V^{+}}=:\infi[\psi]\circ H_{0}|{}_{V^{+}} & \text{on }V^{\infty}
\end{cases} & ,\label{eq:model_orbital_gluing}
\end{align}
which gives us a representation of the space of orbits of $\Delta_{0}|_{\neigh}$
as
\begin{align*}
\quot{\neigh}{\Delta_{0}} & \simeq\quot{\cbar\sqcup\cbar}{\left(\psi^{0},\psi^{\infty}\right)}.
\end{align*}

\subsection{\label{subsec:sectorial_dynamics}Sectorial orbits space and first-integrals}

Here we are particularly interested in the structure of the space
of orbits of the sectorial time-1 maps $\Delta^{\pm}$ of $X_{}^{\pm}$.
We wish to prove that these perturbations of $X_{0}$ still possess
the properties underlined in Section~\ref{subsec:model_primitive_function}
for the model, which is obtained \emph{via }the sectorial normalization
(Section~\ref{subsec:secto_normalization}).
\begin{example}
A \textbf{Fatou coordinate} of $X$ is a locally biholomorphic mapping
$\Psi$ such that
\begin{align*}
\Psi^{*}\pp z & =X.
\end{align*}
This condition is equivalent to $X\cdot\Psi=1$. Obviously the Abel
equation $\Psi^{*}\Delta=\id+1$ holds whenever the time-1 map is
defined, and we also speak of a Fatou coordinate for $\Delta$.

Let $H_{}^{\pm}:=H_{0}\exp\left(2\ii\pi f^{\pm}\right)$ for $f\in\mathcal{S}$.
By construction $H_{}^{\pm}$ is a primitive function of $X_{}^{\pm}$:
\begin{align*}
X_{}^{\pm}\cdot H_{}^{\pm} & =\frac{1}{1+X_{0}\cdot f^{\pm}}X_{0}\cdot H_{}^{\pm}=\frac{1}{1+X_{0}\cdot f^{\pm}}\left(X_{0}\cdot H_{0}+2\ii\pi H_{0}X_{0}\cdot f^{\pm}\right)\exp\left(2\ii\pi f^{\pm}\right)\\
 & =2\ii\pi H_{}^{\pm}.
\end{align*}
Hence $\frac{1}{2\ii\pi}\log H_{}^{\pm}$ is a (sectorial) Fatou coordinate
of $X_{}^{\pm}$ on $\widehat{V}^{\pm}$.
\end{example}

\subsubsection{Sectorial orbits space}

According to Section~\ref{subsec:secto_normalization} we have near
$\neigh$
\begin{align*}
X_{}^{\pm} & =\left(\flow{X_{0}}{f^{\pm}}{}\right)^{*}X_{0}
\end{align*}
and the sectorial dynamics of $X_{}^{\pm}$ (\emph{resp}. of $\Delta^{\pm}$)
are conjugate to that of $X_{0}|_{V^{\pm}}$ (\emph{resp}. of $\Delta_{0}|_{V^{\pm}}$).
Because $X_{0}\left(0\right)=0$ the mapping $\flow{X_{0}}{f^{\pm}}{}$
is tangent-to-identity near $0$. We deduce at once the following
result from the study performed above Section~\ref{subsec:formal-model}
by pull-back. In particular, observe that the primitive function $H_{}^{\pm}$
of $X_{}^{\pm}$ (which provides the preferred coordinate on the orbit
space of $\Delta$) is the pull-back of $H_{0}|_{V^{\pm}}$ by $\flow{X_{0}}{f^{\pm}}{}$
according to Lie's formula~(\ref{eq:Lie}):
\begin{align*}
H_{0}\circ\flow{X_{0}}{f^{\pm}}{} & =\sum_{n=0}^{\infty}\frac{\left(f^{\pm}\right)^{n}}{n!}X_{0}\cdot^{n}H_{0}=H_{0}\sum_{n=0}^{\infty}\frac{\left(2\ii\pi f^{\pm}\right)^{n}}{n!}=H_{}^{\pm}.
\end{align*}
In that sense $\Delta^{\pm}$ and $\Delta_{0}|_{V^{\pm}}$ share the
same canonical orbital coordinate, a fact we summarize below.
\begin{lem}
\label{lem:secto_orbit_space}Let $f^{\pm}$ be holomorphic on $V^{\pm}$
with continuous extension to $\cbar$. Define $X_{}^{\pm}:=\frac{1}{1+X_{0}\cdot f^{\pm}}X_{0}$
and $H_{}^{\pm}:=H_{0}\times\exp\left(2\ii\pi f^{\pm}\right)$. There
exists $\mathcal{V}:=\neigh$ such that the following assertions hold. 
\begin{enumerate}
\item The time-1 map $\Delta^{\pm}$ of $X_{}^{\pm}$ is holomorphic and
injective on $V^{\pm}\cap\mathcal{V}$. Moreover $\Delta^{\pm}\left(z\right)=\Delta_{0}\left(z\right)+\oo{z^{3}}$
near $0$ in $V^{\pm}$.
\item $H^{\pm}\left(V^{\pm}\cap\mathcal{V}\right)=\cc^{\times}$ and $H_{}^{\pm}\left(\adh{\isect}\cap\mathcal{V}\right)=\left(\cbar,\sharp\right)$
for $\sharp\in\left\{ 0,\infty\right\} $.
\item There exists a bijection between orbits of $\Delta^{\pm}$ on $V^{\pm}\cap\mathcal{V}$
and level sets of $H^{\pm}|_{V^{\pm}\cap\mathcal{V}}$.
\end{enumerate}
\end{lem}

\subsubsection{Sectorial primitive first-integrals}
\begin{cor}
\label{cor:primitive_first-integral} Let $U$ be an open subsector
of $V^{\pm}$ attached to $0$. The primitive function $H_{}^{\pm}$
of $X$ is a primitive first-integral of $\Delta^{\pm}$, in the sense
that any first-integral $L$ of $\Delta^{\pm}$ on $U$ factors as
\begin{align*}
L & =g\circ H_{}^{\pm}
\end{align*}
for some $g$ holomorphic on $H_{}^{\pm}\left(U\right)$. (The converse
is trivial.)
\end{cor}

\begin{proof}
Let $L$ be such that $L\circ\Delta=L$. Then it induces a holomorphic
function $g~:~h\mapsto g\left(h\right)$ on the orbit space $\quot{U\cap{\neigh}}{\Delta}$.
According to the previous lemma this space embeds as $H_{}^{\pm}\left(U\right)\subset\cc^{\times}$.
Therefore we can assume that the coordinate $h$ is given by $H_{}^{\pm}$,
\emph{i.e}. $L=g\circ H_{}^{\pm}$.
\end{proof}

\subsection{\label{subsec:Reduction}Reduction: proof of Proposition~\ref{prop:reduction}}

\subsubsection{Item 1: sectorial dynamics and orbits space}

This is the content of Lemma~\ref{lem:secto_orbit_space}, hence
the item is proved.

\subsubsection{Item 2: gluing condition}
\begin{lyxlist}{00.00.0000}
\item [{(a)$\Rightarrow$(b)}] This is trivial.
\item [{(b)$\Rightarrow$(a)}] The condition $\Delta^{-}=\Delta^{+}$ guarantees
that both sectorial functions glue to form a holomorphic germ, still
called $\Delta$, on a punctured neighborhood of $0$. As $\Delta$
is bounded, Riemann's theorem on removable singularity applies: $\Delta$
extends holomorphically to $\neigh$. Finally, since $\Delta_{0}\left(z\right)=z+\lambda z^{2}+\oo{z^{2}}$
(Lemma~\ref{lem:.model_parab}) and $\Psi$ fixes $0$ we conclude
$\Delta\in\parab$.
\item [{(c)$\Rightarrow$(b)}] The time-1 map of $X_{}^{\pm}$ is $\Delta^{-}$
with primitive first-integral $H_{f}^{-}=H_{0}\exp\left(2\ii\pi f^{-}\right)=H_{f}^{+}\times\exp\left(2\ii\pi\left(f^{-}-f^{+}\right)\right)$.
Hence $H_{f}^{-}$ is a first-integral of $\Delta^{+}$ as well as
of $\Delta^{-}$. Take $z\in V^{\cap}$ a point where both $\Delta^{-}$
and $\Delta^{+}$ are defined (this property holds on some bounded
subsector $V^{\sharp}\cap\neigh$). Then 
\begin{align*}
H_{f}^{-}\left(\Delta^{+}\left(z\right)\right) & =H_{f}^{-}\left(z\right)\\
 & =H_{f}^{-}\left(\Delta^{-}\left(z\right)\right),
\end{align*}
which means that $\Delta^{+}=\left(\Delta^{-}\right)^{\circ\ell}$
for some $\ell\in\zz$, since orbits of $\Delta^{-}$ are in $1$-to-$1$
correspondence with level sets of $H_{f}^{-}$. The integer $\ell$
may depend on $z$, but Baire's category theorem and the principle
of analytic continuation assert that $\ell$ does actually not.\\
On the one hand $\Delta^{+}$ and $\Delta^{-}$ coincide with $\Delta_{0}$
up to $\oo{z^{3}}$, while on the other hand $\Delta_{0}^{\circ\ell}\left(z\right)=z+\ell\lambda z^{2}+\oo{z^{2}}$
for all $z\in\neigh$. Thus $\ell=1$ and $\Delta^{+}=\Delta^{-}$
on $V^{\cap}$.
\item [{(b)$\Rightarrow$(c)}] The previous argument can actually be read
backwards too.
\item [{(c)$\Leftrightarrow$(d)}] This is the content of Corollary~\ref{cor:primitive_first-integral}.
\end{lyxlist}

\subsubsection{Item 3: realization}
\begin{enumerate}
\item[(a)] By construction the Écalle--Voronin modulus of $\Delta$ satisfies
\begin{align*}
\begin{cases}
H_{f}^{-}=\zero[\psi]\circ H_{f}^{+} & \text{on }\zero\cap\neigh\\
H_{f}^{-}=\infi[\psi]\circ H_{f}^{+} & \text{on }\infi\cap\neigh
\end{cases} & .
\end{align*}
Replacing $H_{}^{\pm}$ by $H_{0}\exp\left(2\ii\pi f^{\pm}\right)$,
then $\zero[\psi]$ by $\id\exp\left(4\pi^{2}\mu+\zero[\varphi]\right)$
and $\infi[\psi]$ by $\id\exp\left(\infi[\varphi]\right)$ brings
the previous system into~$\left(\star\right)$.
\item[(b)] Near $0$ in $\zero$ the function $H_{}^{\pm}$ is 1-flat according
to~(\ref{eq:estim_H0_0}):
\begin{align*}
\left|H_{}^{\pm}\left(z\right)\right| & \leq\left|H_{0}\left(z\right)\right|\exp\left(2\pi\left|f^{\pm}\left(z\right)\right|\right)\\
 & \leq C\exp\frac{-C'}{\left|z\right|}
\end{align*}
with $0<C'<2\pi\cos\frac{\pi}{8}$. Because $\zero[\varphi]\left(0\right)=0$
the composition $\zero[\varphi]\circ H_{f}^{+}$ is also flat at $0$.
Therefore $f^{-}-f^{+}$ is 1-flat at $0$ in $\zero$. The same argument
also applies in $\infi$ (beware that $\infi[\varphi]$ is embodied
as a convergent power series in $\frac{1}{H_{f}^{+}}$). Then Ramis--Sibuya
theorem~\cite[Theorem \#\#]{LodRich} asserts exactly that $\left(f^{+},f^{-}\right)$
is the 1-sum of some $F\in z\frml z$. Being a $1$-sum is stable
by differentiation so that $\left(X_{0}\cdot f^{+},X_{0}\cdot f^{-}\right)$
is a $1$-sum of $X_{0}\cdot F$. Hence $\left(\flow{X_{0}}{f^{+}}{},\flow{X_{0}}{f^{-}}{}\right)$
is locally\footnote{Here the notion of $1$-sum differs from that stated in the introduction
in that we only have the holomorphy of $\Psi^{\pm}$ on a small sectorial
region near $0$. This is sufficient, though.} a $1$-sum of $\Psi:=\flow{X_{0}}F{}$ (according to Lemma~\ref{prop:flow_change_variable}),
from which follows that $\left(\Delta^{+},\Delta^{-}\right)$ is a
$1$-sum of 
\begin{align*}
\flow X1{} & =\Psi^{*}\Delta_{0}.
\end{align*}
Since $\Delta^{+}=\Delta^{-}$ is a convergent power series at $0$,
it can only mean that $\Delta=\flow X1{}$.
\end{enumerate}

\section{\label{sec:Main_th}Synthesis}

We begin with fixing a formal class $\mu\in\cc$ in $\parab$ (then
$\zero[\psi]$ must be tangent to the linear map $\ee^{4\pi^{2}\mu}\id$)
and pick an analytic data
\begin{align*}
\zero[\psi]~:~\left(\cbar,0\right) & \longto\left(\cbar,0\right)\\
h & \longmapsto h\exp\left(4\pi^{2}\mu+\zero[\varphi]\left(h\right)\right)~~~~~,~\zero[\varphi]\left(0\right)=0~,\\
\infi[\psi]~:~\left(\cbar,\infty\right) & \longto\left(\cbar,\infty\right)\\
h & \longmapsto h\exp\left(\infi[\varphi]\left(h\right)\right)~~~~~,~\infi[\varphi]\left(\infty\right)=0~.
\end{align*}
We wish to incarnate the abstract variable $h$ as a concrete coordinate
on a sectorial space of orbits of a germ $\Delta\in\parab$. For this
we seek a pair of functions $\left(H^{+},H^{-}\right)$
\begin{align*}
H^{\pm}~:~V^{\pm} & \longto\cc\subset\cbar
\end{align*}
whose range is biholomorphic to $\quot{V^{\pm}}{\Delta}$ and such
that $\left(\zero[\psi],\infi[\psi]\right)$ gives the transition
between $H^{+}$ and $H^{-}$:
\begin{align*}
H^{-} & =\isect[\psi]\circ H^{+}~~~~\text{on }\isect\,\,\text{for\,}\sharp\in\left\{ 0,\infty\right\} .
\end{align*}
We recall that the two connected components of the intersection $V^{\cap}$
are sectors defined by

\begin{multicols}{2}
\hfill{}\includegraphics[width=0.5\columnwidth]{images/inter_sectors.pdf}\hfill{}

\columnbreak

\begin{align*}
\zero & :=\left\{ z~:~\im z>0\right\} \cap V^{\cap},\\
\infi & :=\left\{ z~:~\im z<0\right\} \cap V^{\cap}.
\end{align*}
\end{multicols}

Starting from a pair 
\begin{align*}
\varphi & :=\left(\zero[\varphi],\infi[\varphi]\right),
\end{align*}
the idea is to obtain $H^{\pm}$ as a perturbation 
\begin{align*}
H_{}^{\pm} & :=H_{0}\times\exp\left(2\ii\pi f^{\pm}\right)~~,~f^{\pm}\in\mathcal{S}
\end{align*}
of a primitive first-integral of the time-1 map $\Delta_{0}$ of the
formal model $X_{0}$ (see Section~\ref{subsec:formal-model}). According
to Proposition~\ref{prop:reduction} we need to solve the non-linear
Cousin problem 
\begin{align*}
\begin{cases}
f^{-}-f^{+} & =\zero[\varphi]\circ H_{f}^{+}\text{~~~~ on }\zero\cap\neigh\\
f^{-}-f^{+} & =\infi[\varphi]\circ H_{f}^{+}\text{~~~~ on }\infi\cap\neigh
\end{cases} & \tag{\ensuremath{\star}}.
\end{align*}

We introduce the Cauchy-Heine operator $\cauchein[\varphi]$ in Section~\ref{subsec:Cauchy-Heine},
then prove it admits a unique fixed point in the unit ball of $\mathcal{S}$
(Section~\ref{subsec:Convergence}). The latter is the sought solution
of~$\left(\star\right)$.
\begin{rem}
The proofs of all the technical lemmas involved below are to be found
in Section~\ref{subsec:proofs}.
\end{rem}

\subsection{\label{subsec:choice_X0}Choice of the formal model and of the sectors}

Before even considering solving the Cousin problem $\left(\star\right)$,
it should be well-posed. In particular the compositions $\zero[\varphi]\circ H_{f}^{+}$
and $\infi[\varphi]\circ H_{f}^{-}$ must make sense in the respective
intersections $\zero$ and $\infi$. This condition is not technical:
for a genuine parabolic germ $\Delta$ a horn map is defined at least
on the maximal domain of orbits which are sent by the local dynamics
from one end of the sector $V^{+}$ to the corresponding end of $V^{-}$.

Roughly speaking, this orbital domain is the range of the corresponding
sectorial first-integral $H_{f}^{+}$ of the respective intersection's
component $\zero$ or $\infi$. In order to be able to gain control
on its size, we impose the technical restriction
\begin{align*}
\norm[H_{f}^{+}]{\zero}:=\sup_{z\in\zero}\left|H_{f}^{+}\left(z\right)\right| & <\zero[\rho]\\
\norm[\frac{1}{H_{f}^{+}}]{\infi} & <\infi[\rho]
\end{align*}
where $\zero[\rho]$ and $\infi[\rho]$ are the radius of convergence
of $\zero[\varphi]$ and $\infi[\varphi]$, the former as a power
series in $h\in\neigh$ and the latter in $\frac{1}{h}\in\neigh$.

The natural idea that comes up is to use the flatness of sectorial
first-integrals at $0$: on a small pair of sectors near $0$ the
above condition can be easily enforced. One then synthesizes $\Delta$
on $\neigh$ in her preferred fashion (invoking Ahlfors-Bers theorem
or directly by a holomorphic fixed-point). Treading on that path involves
technical complications and other shortcomings, least of all the complete
lack of a control on the form of the resulting germ $\Delta$. This
in turn prevents any foreseeable statement about the uniqueness of
$\Delta$.

To address the uniqueness question one is therefore led to synthesize
a \emph{global} object $\Delta$. But one has to pay a price for it:
any given model first-integral has a given range $\zero[\Omega]$
over $\zero$ which cannot fit within the domain of \emph{every} germ
$\zero[\varphi]$. Admittedly one could contract the coordinate $h$
by a linear map, but this would automatically increase the observed
size of the orbital domain $\infi[\Omega]$ near $\infty$. Hence
only rare functions with $\zero[\rho]\infi[\rho]$ bounded from below
by the <<size>> of $\zero[\Omega]$ and $\infi[\Omega]$ can be
realized as a perturbation of this first-integral.
\begin{rem}
In particular when the data is unilateral (in Écalle terminology,
meaning $\zero[\varphi]=0$ or $\infi[\varphi]=0$) it is possible
to play the rescaling game on the orbits sphere and obtain a realization
for any $\lambda>0$. The same holds if one of the component of $\varphi$
is entire.
\end{rem}

To illustrate the point we can obtain bounds for the primitive function
\begin{align*}
\widehat{H}\left(x\right) & :=\exp\left(-\frac{2\ii\pi}{x}\right)x^{2\ii\pi\mu}
\end{align*}
 of the usual formal model $\frac{x^{2}}{1+\mu x}\pp x$.
\begin{lem}
\label{lem:usual_model_orbits_size}Define for $\tau\in\cc^{\times}$
\begin{align*}
\mathfrak{h}\left(\tau\right) & :=\exp\left(-2\ii\pi\left(\tau+\mu\log\tau\right)\right).
\end{align*}
For every $0<\delta<\frac{\pi}{2}$ the following estimates hold:
\begin{align*}
\left(\forall\tau~:~\left|\arg\tau+\frac{\pi}{2}\right|=\delta,~\mathfrak{t}<\left|\tau\right|\right)~~~~~~\left|\mathfrak{H}\left(\tau\right)\right| & <\mathfrak{m}\\
\left(\forall\tau~:~\left|\arg\tau-\frac{\pi}{2}\right|=\delta,~\mathfrak{t}<\left|\tau\right|\right)~~~~~~\left|\mathfrak{H}\left(\tau\right)\right| & >\frac{1}{\mathfrak{m}}
\end{align*}
where
\begin{align*}
\mathfrak{m} & =\mathfrak{m}_{\mu,\delta}:=\exp\left(2\pi^{2}\left|\re{\mu}\right|+2\pi\im{\mu}\ln\frac{\left|\im{\mu}\right|}{\ee\cos\delta}\right)>0\\
\mathfrak{t} & =\mathfrak{t}_{\mu,\delta}:=\max\left\{ 1,\frac{\ln\mathfrak{m}}{2\pi\cos\delta}\right\} .
\end{align*}
\end{lem}

We may think of $\tau$ as $\frac{1}{x}$ because $\widehat{H}\left(x\right)=\mathfrak{h}\left(\frac{1}{x}\right)$.
Hence the above estimate tells us that the size (in the coordinate
$\widehat{H}$) of the sectorial orbit spaces $\Omega_{0}$ and $\Omega_{\infty}$
shrinks to $0$ over smaller and smaller sectors of radius $\frac{1}{\mathfrak{t}}\to0$.
But if we want to work on unbounded sectors covering $\cc^{\times}$
we cannot act on the <<smallness>> of the sectors anymore, we must
thereby find another way to control the size of the orbits domains
to accommodate a given $\left(\zero[\psi],\infi[\psi]\right)$. This
is done by performing the pullback of $\frac{x^{2}}{1+\mu x}\pp x$
by a $1$-parameter family of rational maps $x:=\Pi\left(z\right)$
sending $V^{\cap}$ into sectors of size $\frac{1}{\mathfrak{t}}=\OO{\lambda}$
for which the previous lemma gives a bound $\mathfrak{m}\sim\exp\left(-\nf{\cst}{\lambda}\right)$.

\bigskip{}

Changing the variable $z\mapsto x=\Pi\left(z\right)$ in the usual
formal model $\frac{x^{2}}{1+\mu x}\pp x$ results in a vector field
whose time-$1$ flow has a first-integral of the form $\mathfrak{H}\left(\tau\right)$
for $\tau=\tau\left(z\right):=\frac{1}{\Pi\left(z\right)}$. Namely
if 
\begin{align*}
\Pi\left(z\right) & =\frac{\lambda z}{1-z^{2}}
\end{align*}
then $\frac{x^{2}}{1+\mu x}\pp x$ is transformed into the $\sigma$-invariant
vector field
\begin{align*}
X_{0}\left(z\right) & :=\frac{1-z^{2}}{1+z^{2}}\times\frac{\lambda z^{2}}{1+\mu\lambda z-z^{2}}\pp z
\end{align*}
and $\tau$ into 
\begin{align*}
\tau=\tau_{\lambda}\left(z\right):= & \frac{1-z^{2}}{\lambda z}
\end{align*}
yielding the first-integral 
\begin{align}
H_{0}\left(z\right) & =\widehat{H}\left(\frac{\lambda z}{1-z^{2}}\right)=\mathfrak{H}\left(\tau_{\lambda}\left(z\right)\right)\nonumber \\
 & =\exp\left(-2\ii\pi\frac{1-z^{2}}{\lambda z}-2\ii\pi\mu\log\frac{1-z^{2}}{\lambda z}\right).\label{eq:model_first-int}
\end{align}
Because $\tau\circ\sigma=\tau$, all the above objects are $\sigma$-invariant.
We establish now that this first-integral fulfills the expected properties
on the pair of unbounded sectors $V^{\pm}$. Firstly, we ensure that
$\tau\left(V^{\cap}\right)$ is included in an open sector of opening
$\frac{\pi}{2}$ with vertex at $\infty$.
\begin{lem}
\label{lem:model_Fatou_size_shrink}Fix $\lambda\leq1$ and set $\tau=\tau_{\lambda}\left(z\right):=\frac{1-z^{2}}{\lambda z}$.
Then
\begin{align*}
\begin{cases}
\inf_{z\in V^{\cap}}\left|\tau\right| & \geq\frac{1}{\lambda}\\
\sup_{z\in\zero}\left|\arg\tau+\frac{\pi}{2}\right| & =\frac{3\pi}{8}\\
\sup_{z\in\infi}\left|\arg\tau-\frac{\pi}{2}\right| & =\frac{3\pi}{8}
\end{cases} & .
\end{align*}
\end{lem}

Secondly, we adapt the bounds of Lemma~\ref{lem:usual_model_orbits_size}
for $\delta:=\frac{3\pi}{8}$ to obtain the sought values of $\mathfrak{m}_{\mu}$
and $\mathfrak{t}_{\mu}$ (we recall that $2\cos\frac{3\pi}{8}=\sqrt{2-\sqrt{2}}$):
\begin{align}
\mathfrak{m}_{\mu} & :=\exp\left(2\pi^{2}\left|\re{\mu}\right|+2\pi\im{\mu}\ln\frac{4\left|\im{\mu}\right|}{\ee\sqrt{2-\sqrt{2}}}\right)\label{eq:bounds_on_h}\\
\mathfrak{t}_{\mu} & :=\max\left\{ 1,\frac{\ln\mathfrak{m}_{\mu}}{\pi\sqrt{2-\sqrt{2}}}\right\} \nonumber 
\end{align}
as testified by the concluding result of this section.
\begin{cor}
\label{cor:model_orbits_size}Pick $0<\lambda<\frac{1}{\mathfrak{t}_{\mu}}$
and define
\begin{align*}
R_{\lambda} & :=\inf_{z\in\infi}\left|H_{0}\left(z\right)\right|,\\
r_{\lambda} & :=\sup_{z\in\zero}\left|H_{0}\left(z\right)\right|.
\end{align*}
The following estimate holds:
\begin{align*}
r_{\lambda} & \leq\frac{\mathfrak{m}_{\mu}}{\exp\nf 1{\lambda}}\leq\frac{\exp\nf 1{\lambda}}{\mathfrak{m}_{\mu}}\leq R_{\lambda}.
\end{align*}
In fact, one has more precisely
\begin{align*}
\begin{cases}
\left|H_{0}\left(z\right)\right|<\mathfrak{m}_{\mu}\ee^{-\left|\tau_{\lambda}\left(z\right)\right|} & \text{for }z\in\zero\\
\left|H_{0}\left(z\right)\right|>\frac{1}{\mathfrak{m}_{\mu}}\ee^{\left|\tau_{\lambda}\left(z\right)\right|} & \text{for }z\in\infi
\end{cases} & .
\end{align*}
\end{cor}

\begin{proof}
The constants $\mathfrak{m}_{\mu}^{2}$ and $\mathfrak{t}_{\mu}$
are related to that of Lemma~\ref{lem:usual_model_orbits_size} but
for $2\mu$. Indeed, we have for $\left|\arg\frac{\tau}{\ii}\right|<\frac{3\pi}{8}$:
\begin{align*}
\left|\mathfrak{H}\left(\tau\right)\right|^{2} & =\left|\exp\left(-2\ii\pi\tau\right)\right|\times\left|\exp\left(-2\ii\pi\left(\tau+2\mu\log\tau\right)\right)\right|\\
 & \geq\exp\left(\left|\tau\right|\pi\sqrt{2-\sqrt{2}}\right)\times\left|\exp\left(-2\ii\pi\left(\tau+2\mu\log\tau\right)\right)\right|
\end{align*}
which becomes (since $\pi\sqrt{2-\sqrt{2}}>2$):
\begin{align*}
\left|\mathfrak{H}\left(\tau\right)\right|^{2} & \geq\exp\left(2\left|\tau\right|\right)\frac{1}{\mathfrak{m}_{\mu}^{2}}
\end{align*}
whenever $\left|\tau\right|>\mathfrak{t}_{\mu}$. The latter condition
is ensured as soon as $\frac{1}{\lambda}>t_{\mu}$ thanks to Lemma~\ref{lem:model_Fatou_size_shrink}.
Under this hypothesis we finally obtain:
\begin{align*}
\left|H_{0}\left(z\right)\right|=\left|\mathfrak{h}\left(\tau\right)\right| & \geq\frac{1}{\mathfrak{m}_{\mu}}\exp\left|\tau\right|\geq\frac{\exp\nf 1{\lambda}}{\mathfrak{m}_{\mu}}.
\end{align*}
The case $\left|\arg\frac{\tau}{-\ii}\right|<\frac{3\pi}{8}$ is similar.
\end{proof}

\subsection{\label{subsec:Adapted_functions}Functions adapted to a data}

Let us transport in functions space the previous discussion.
\begin{defn}
\label{def:function_space}As in the beginning of the section we define
the intersection
\begin{align*}
V^{\cap} & :=V^{+}\cap V^{-}=\zero\sqcup\infi.
\end{align*}
\begin{enumerate}
\item Let $U\subset\cc$ be a domain. We introduce the Banach space
\begin{align*}
\holb[U]
\end{align*}
of holomorphic, bounded functions $f$ on $U$ with continuous extension
to the closure $\adh U$, equipped with the sup norm
\begin{align*}
\norm[f]U & :=\sup\left|f\left(U\right)\right|.
\end{align*}
\item As a particular case we will be interested in 
\begin{align*}
\mathcal{S}\left(V^{\pm}\right) & :=\left\{ f^{\pm}\in\holb[V^{\pm}]~:~\begin{array}{l}
f^{\pm}\left(0\right)=0\\
f^{\pm}\left(\infty\right)=0
\end{array}~,~\norm[f^{\pm}]{}:=\sup_{z\in V^{\pm}}\left|f^{\pm}\left(z\right)\right|<\infty\right\} .
\end{align*}
\item We define $\mathcal{S}$ as the space of pairs $f=\left(f^{+},f^{-}\right)$
with 1-flat difference in $V^{\cap}$, both at $0$ and $\infty$:
\begin{align*}
\mathcal{S} & :=\left\{ f\in\mathcal{S}\left(V^{+}\right)\times\mathcal{S}\left(V^{-}\right)~:~\lim_{z\to0,\infty}\left|z\right|\ln\left|f^{-}\left(z\right)-f^{+}\left(z\right)\right|<0\text{ for }z\in V^{\cap}\right\} 
\end{align*}
equipped with the canonical product Banach norm. We denote by $\mathcal{B}$
its unit ball.
\item For $f\in\mathcal{S}$ we define the associated \textbf{sectorial
first-integral} $H_{f}=\left(H_{f}^{+},H_{f}^{-}\right)$ given by
\begin{align*}
H_{}^{\pm} & :=H_{0}\exp\left(2\ii\pi f^{\pm}\right).
\end{align*}
\item Let $\varphi=\left(\zero[\varphi],\infi[\varphi]\right)$ be given.
We say that $\left(\lambda,f\right)$ is \textbf{adapted} to $\varphi$
whenever $\adh{H_{f}^{+}\left({\isect}\right)}$ is included in the
(open) disc of convergence of $\isect[\varphi]$ for $\sharp\in\left\{ 0,\infty\right\} $.
We define for all $\lambda>0$
\begin{align*}
\adapt & :=\left\{ f\in\mathcal{S}~:~\left(\lambda,f\right)\text{ is adapted to }\varphi\right\} .
\end{align*}
\end{enumerate}
\end{defn}

We have, with a corresponding estimate for $\infi[\varphi]$:
\begin{align}
\left(\forall h\in\neigh\right)~~~~~\left|\zero[\varphi]\left(h\right)\right| & \leq\left|h\right|\norm[\ddd{\zero[\varphi]}h]{\neigh}<+\infty.\label{eq:estim_phi}
\end{align}

\begin{prop}
\label{lem:adapted_open_condition}Choose a data $\varphi=\left(\zero[\varphi],\infi[\varphi]\right)$.
The subspace $\adapt$ is an open set of $\mathcal{S}$ and, if $\lambda$
is small enough, it contains the unit ball:
\begin{align*}
\mathcal{B} & \subset\adapt.
\end{align*}
More generally, being given $f\in\mathcal{S}$ it is always possible
to take $\lambda$ small enough to ensure that $f\in\adapt$. (The
following remark gives quantitative bounds.)
\end{prop}

\begin{rem}
Denote by $\rho^{\sharp}\in]0,+\infty]$ for $\sharp\in\left\{ 0,\infty\right\} $
the radius of convergence of $\varphi^{\sharp}$. Quantitatively the
following conditions ensure that $f\in\adapt$
\begin{align*}
\lambda & <\frac{1}{\mathfrak{t}_{\mu}}\\
2\pi\norm[f]{} & <\frac{1}{\lambda}+\ln\frac{\min\left\{ \zero[\rho],\infi[\rho]\right\} }{\mathfrak{m}_{\mu}},
\end{align*}
since according to Corollary~\ref{cor:model_orbits_size} this implies
\begin{align*}
\inf\left|H_{f}^{+}\left(\infi\right)\right| & >\frac{1}{\infi[\rho]}\\
\sup\left|H_{f}^{+}\left(\zero[V]\right)\right| & <\zero[\rho]~.
\end{align*}
In particular $\isect[\varphi]\circ H_{f}^{+}\in\holb[{\isect}]$.
\end{rem}

\begin{proof}
The mapping $f\in\mathcal{S}\mapsto H_{f}^{+}\in\holb[{\isect}]$
is continuous, and for each $f\in\adapt$ the set $\adh{H_{f}^{+}\left({\isect}\right)}$
is compact in $\mathcal{S}{}_{\isect[\varphi]}$. Thus $\adapt$ is
open in $\mathcal{S}$. The remark just above precisely states that
for any $r>0$ the ball $r\mathcal{B}$ is included in $\adapt$ whenever
\begin{align*}
\frac{1}{\lambda}+\ln\frac{\min\left\{ \zero[\rho],\infi[\rho]\right\} }{\mathfrak{m}_{\mu}} & >2\pi r.
\end{align*}
\end{proof}

\subsection{\label{subsec:Cauchy-Heine}Cauchy-Heine transform}

\begin{wrapfigure}{l}{2.5cm}%
\hfill{}\includegraphics[width=2cm]{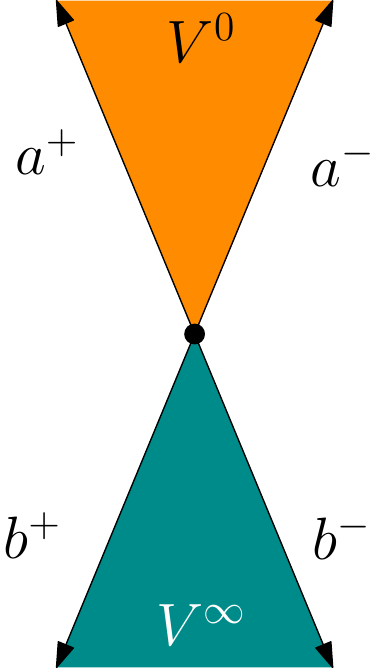}\hfill{}

\end{wrapfigure}%
Now that $H_{0}$ and the sectors have been defined we tackle the
Cousin problem itself with data $\varphi:=\left(\zero[\varphi],\infi[\varphi]\right)$
as in~$\left(\star\right)$. Let us define the integral transform
which is the key to the construction.
\begin{defn}
Let $\varphi:=\left(\zero[\varphi],\infi[\varphi]\right)$ be given.
Assume that $\left(\lambda,f\right)$ is adapted to it. Let $a^{\pm},~b^{\pm}$
be the outward-going half-lines making up the boundary $\partial V^{\pm}$
as in the side figure. We define $\Lambda_{f}:=\left(\Lambda_{f}^{+},\Lambda_{f}^{-}\right)$
where
\begin{align*}
\Lambda_{f}^{\pm}\left(z\right) & :=\frac{\sqrt{z}}{2\ii\pi}\int_{a^{\pm}}\frac{\zero[\varphi]\left(H_{f}^{+}\left(\xi\right)\right)}{\sqrt{\xi}\left(\xi-z\right)}\dd{\xi}-\frac{\sqrt{z}}{2\ii\pi}\int_{b^{\pm}}\frac{\infi[\varphi]\left(H_{f}^{+}\left(\xi\right)\right)}{\sqrt{\xi}\left(\xi-z\right)}\dd{\xi}.
\end{align*}
\end{defn}

Before stating anything about $\Lambda_{f}$ we need to get convinced
that it is well-defined. For the sake of example, let us deal with
the $\int_{a}$-part. According to the identity~(\ref{eq:estim_phi})
we may find some $C=\norm[\ddd{\varphi}h]{H_{f}^{+}\left(V^{\cap}\right)}>0$
such that for all $h\in H_{f}^{+}\left(a^{\pm}\right)$ we have
\begin{align*}
\left|\zero[\varphi]\left(h\right)\right| & \leq C\left|h\right|.
\end{align*}
Since we have for all $\xi\in a$
\begin{align}
\left|H_{f}^{+}\left(\xi\right)\right| & \leq\left|H_{0}\left(\xi\right)\right|\exp\left(2\pi\norm[f^{+}]{\zero[V]}\right),\label{eq:estim_H}
\end{align}
and because $\left|H_{0}\right|$ is flat at $0$ and $\infty$ along
$a^{\pm}$ and $b^{\pm}$, the integrals defining $\Lambda_{f}^{\pm}$
are absolutely convergent. That being said, as we want to control
the magnitude of $\left|\Lambda_{f}^{\pm}\right|$ with respect to
both the point $z\in V^{\pm}$ and the parameter $\lambda$, we need
to work a little bit more.
\begin{lem}
\label{lem:model_integrable}The model first-integral $H_{0}$ as
given in~(\ref{eq:model_first-int}) satisfies
\begin{align*}
\left(\forall z~:~\re z\geq0\right)~~~~~~\int_{a^{+}}\left|\frac{\sqrt{z}H_{0}\left(\xi\right)\dd{\xi}}{\sqrt{\xi}\left(\xi-z\right)}\right| & <\frac{3\mathfrak{m}_{\mu}}{2}\lambda^{2}.
\end{align*}
Identical bounds for $\int_{a^{-}}$ and $\int_{b^{\pm}}$ also hold.
\end{lem}

We are now ready to prove the main properties of the operator $\Lambda_{f}$.
\begin{prop}
\label{prop:cauchy-heine}Let $\varphi:=\left(\zero[\varphi],\infi[\varphi]\right)$
be given and assume that $\left(\lambda,f\right)\in\rr_{>0}\times\mathcal{S}$
is adapted to it. Let $\Lambda_{f}^{\pm}$ be as in the definition
and recall that $\sigma$ is the involution $z\mapsto\frac{-1}{z}$.
\begin{enumerate}
\item For all $z\in V^{\pm}$ we have
\begin{align*}
\Lambda_{f}^{\mp}\circ\sigma\left(z\right) & =-\Lambda_{f\circ\sigma}^{\pm}\left(z\right).
\end{align*}
\item For all $z\in V^{\cap}$ we have
\begin{align*}
\Lambda_{f}^{-}-\Lambda_{f}^{+} & =\begin{cases}
\zero[\varphi]\circ H_{f}^{+} & \text{on }\zero\\
\infi[\varphi]\circ H_{f}^{+} & \text{on }\infi
\end{cases}.
\end{align*}
\item $\Lambda_{f}^{\pm}\in\holb[V^{\pm}]$ and 
\begin{align*}
\norm[\Lambda_{f}^{\pm}]{V^{\pm}} & \leq4\mathfrak{m}_{\mu}\lambda^{2}\times\ee^{2\pi\norm[f^{+}]{V^{\cap}}}\times\max\left\{ \norm[\ddd{\zero[\varphi]}h]{H_{f}^{+}\left({\zero[V]}\right)},\norm[\ddd{\infi[\varphi]}h]{H_{f}^{+}\left({\infi}\right)}\right\} .
\end{align*}
\end{enumerate}
\end{prop}

\begin{proof}
~
\begin{enumerate}
\item Since $H_{0}\circ\sigma=H_{0}$ we have $H_{f}^{\pm}\circ\sigma=H_{f\circ\sigma}^{\mp}$.
The rest follows from applying the straightforward change of variable
$u:=\sigma\left(\xi\right)$, which permutes $a^{\pm}\leftrightarrow-a^{\mp}$
and $b^{\pm}\leftrightarrow-b^{\mp}$. If $z\in V^{-}$ we have $\sigma\left(z\right)=\frac{-1}{z}\in V^{+}$
and
\begin{align*}
\Lambda_{f}^{+}\left(\sigma\left(z\right)\right)= & \frac{1}{2\ii\pi}\int_{a^{+}}\frac{\sqrt{-z\xi}\zero[\varphi]\left(H_{f}^{+}\left(\xi\right)\right)}{-z\xi-1}\frac{\dd{\xi}}{\xi}\\
 & +\left(\text{same integral with \ensuremath{\infi[\varphi]} over }b^{+}\right)\\
= & \frac{1}{2\ii\pi}\int_{-a^{-}}\frac{\sqrt{\frac{z}{u}}\zero[\varphi]\left(H_{f}^{+}\left(\sigma\left(u\right)\right)\right)}{1-\frac{z}{u}}\frac{\dd u}{u}\\
 & +\left(\text{same integral with \ensuremath{\infi[\varphi]} over }-b^{-}\right)\\
= & -\Lambda_{f\circ\sigma}^{-}\left(z\right).
\end{align*}
\item is a consequence of Cauchy's formula. Assume for the sake of example
that $z\in\zero$ and pick $\varepsilon>0$ so small that
\begin{align*}
z & \in\zero_{\varepsilon}:=\left\{ x\in\zero~:~\varepsilon<\left|x\right|<\frac{1}{\varepsilon}\right\} .
\end{align*}
Hence (with the direct orientation on the boundary)
\begin{align*}
\int_{\partial\zero_{\varepsilon}}\frac{\zero[\varphi]\left(H_{f}^{+}\left(\xi\right)\right)}{\sqrt{\xi}\left(\xi-z\right)}\dd{\xi} & =\frac{2\ii\pi}{\sqrt{z}}\zero[\varphi]\left(H_{f}^{+}\left(z\right)\right)
\end{align*}
and with corresponding notations
\begin{align*}
\int_{\partial\infi_{\varepsilon}}\frac{\infi[\varphi]\left(H_{f}^{+}\left(\xi\right)\right)}{\sqrt{\xi}\left(\xi-z\right)}\dd{\xi} & =0.
\end{align*}
Because $\zero[\varphi]\left(H_{f}^{+}\left(\xi\right)\right)$ is
exponentially flat near $0$ and $\infty$ in the sector $\zero$
these identities hold at the limit $\varepsilon\to0$, so that
\begin{align*}
\Lambda_{f}^{-}\left(z\right)-\Lambda_{f}^{+}\left(z\right) & =\frac{\sqrt{z}}{2\ii\pi}\int_{\partial\zero}\frac{\zero[\varphi]\left(H_{f}^{+}\left(\xi\right)\right)}{\sqrt{\xi}\left(\xi-z\right)}\dd{\xi}+\frac{\sqrt{z}}{2\ii\pi}\int_{\partial\infi}\frac{\infi[\varphi]\left(H_{f}^{+}\left(\xi\right)\right)}{\sqrt{\xi}\left(\xi-z\right)}\dd{\xi}
\end{align*}
yields the expected result.
\item When $z\in V^{+}$ with $\re z\geq0$ the bound 
\begin{align*}
\left|\Lambda_{f}^{\pm}\left(z\right)\right| & \leq\max\left\{ \norm[\ddd{\zero[\varphi]}h]{H_{f}^{+}\left({\zero[V]}\right)},\norm[\ddd{\infi[\varphi]}h]{H_{f}^{+}\left({\infi}\right)}\right\} \exp\left(2\pi\norm[f^{+}]{V^{\cap}}\right)\times3\mathfrak{m}_{\mu}\lambda^{2}
\end{align*}
 is obtained by putting together the estimates of~(\ref{eq:estim_phi}),~(\ref{eq:estim_H})
and twice Lemma~\ref{lem:model_integrable} (once for $a$ and once
for $b$). According to~1., the same argument proves the estimate
for $\left|\Lambda_{f}^{-}\left(z\right)\right|$ when $\re z\leq0$
and $z\in V^{-}$.\\
We must now deal with the case $z\in V^{+}$ and $\re z<0$. We use
the following trick: since we just proved that $\left|\Lambda_{f}^{-}\left(z\right)\right|$
satisfies the expected estimate on $V^{-}$, and because $\Lambda_{f}^{+}=\Lambda_{f}^{-}-\isect[\varphi]\circ H_{f}^{+}$
we have
\begin{align*}
\left|\Lambda_{f}^{+}\left(z\right)\right| & \leq\left|\Lambda_{f}^{-}\left(z\right)\right|+\norm[\ddd{\isect[\varphi]}h]{H_{f}^{+}\left({\isect}\right)}\left|H_{f}^{+}\left(z\right)\right|\\
 & \leq\left|\Lambda_{f}^{-}\left(z\right)\right|+\norm[\ddd{\isect[\varphi]}h]{H_{f}^{+}\left({\isect}\right)}\exp\left(2\pi\norm[f^{+}]{}\right)\times\frac{\mathfrak{m}_{\mu}}{\exp\nf 1{\lambda}}
\end{align*}
according to Corollary~\ref{cor:model_orbits_size}. The conclusion
follows from $\frac{1}{\exp\nf 1{\lambda}}\leq4\lambda^{2}\ee^{-2}$
and $4\ee^{-2}+3<4$.
\end{enumerate}
\end{proof}
\begin{defn}
\label{def:Cauchy-Heine}We call 
\begin{align*}
\text{\ensuremath{\cauchein[\varphi]}}~:~\adapt & \longto\mathcal{S}\\
f & \longmapsto\Lambda_{f}-\Lambda_{f}\left(0\right)
\end{align*}
the \textbf{Cauchy-Heine transform}.
\end{defn}

We deduce the following facts from the items of Proposition~\ref{prop:cauchy-heine}.
\begin{itemize}
\item From~1. and thanks to the $\sigma$-action we obtain
\begin{align*}
0=\cauchein\left(f\right)^{\pm}\left(0\right) & =\cauchein\left(f\right)^{\pm}\left(\infty\right).
\end{align*}
Moreover if $f$ is a fixed-point of $\cauchein$ then
\begin{align}
f^{\pm}\circ\sigma & =-f^{\mp}.\label{eq:sigma_action_on_f}
\end{align}
\item From~2. we find that $\Lambda_{f}^{+}\left(0\right)=\Lambda_{f}^{-}\left(0\right)$,
hence $\text{\ensuremath{\cauchein[\varphi]}}\left(f\right)$ also
solves the Cousin problem 
\begin{align*}
\cauchein\left(f\right)^{-}-\cauchein\left(f\right)^{+} & =\varphi\circ H_{f}^{+}.
\end{align*}
\item From~3. we derive the estimate
\begin{align}
\norm[{\cauchein[\varphi]}\left(f\right)]{} & \leq8\mathfrak{m}_{\mu}\lambda^{2}\times\ee^{2\pi\norm[f^{+}]{V^{\cap}}}\times\max\left\{ \norm[\ddd{\zero[\varphi]}h]{H_{f}^{+}\left({\zero[V]}\right)},\norm[\ddd{\infi[\varphi]}h]{H_{f}^{+}\left({\infi}\right)}\right\} .\label{eq:estim_CH}
\end{align}
\end{itemize}

\subsection{\label{subsec:Convergence}Convergence of the fixed-point method}

From what we established in Proposition~\ref{prop:cauchy-heine}
we can derive more interesting properties of $\cauchein$, which will
allow us to iterate it provided $\lambda$ be small enough. Define
for $\varphi\in h\germ h$ and when it makes sense:
\begin{align*}
\norm[\varphi]{\lambda} & :=\sup_{\left|z\right|\leq\mathfrak{m}_{\mu}\exp\left(2\pi-\nf 1{\lambda}\right)}\left|\varphi'\left(z\right)\right|.
\end{align*}
For given $\varphi$ it decreases to $\left|\varphi'\left(0\right)\right|$
as $\lambda\to0$. For given $\lambda>0$ it is a norm on the Banach
space of bounded and holomorphic functions on the disc $\mathfrak{m}_{\mu}\exp\left(2\pi-\nf 1{\lambda}\right)\ww D$
vanishing at $0$. Moreover $\left|\varphi\left(h\right)\right|\leq\left|h\right|\norm[\varphi]{\lambda}$
for all $h$ in the disc.
\begin{prop}
\label{prop:CH_contractant}Let $\mathcal{B}$ stand for the closed
unit ball in the Banach space $\mathcal{S}$ (as in Section~\ref{subsec:Adapted_functions}).
Being given $\varphi:=\left(\zero[\varphi],\infi[\varphi]\right)$
define
\begin{align*}
\ell=\ell\left(\varphi\right) & :=\max\left\{ 1,\frac{1}{\mathfrak{t}_{\mu}},\frac{1}{2\pi+\ln\frac{\mathfrak{m}_{\mu}}{\min\left\{ \zero[\rho],\infi[\rho]\right\} }}\right\} .
\end{align*}
For $0<\lambda\leq\ell$ let
\begin{align*}
\kappa_{\lambda} & :=8\mathfrak{m}_{\mu}\lambda^{2}\max\left\{ \norm[{\zero[\varphi]}]{\ell},\norm[{\infi[\varphi]}]{\ell}\right\} \\
r_{\lambda} & :=536\kappa_{\lambda}
\end{align*}
(which is well-defined since $\zero[\varphi]$ and $\infi[\varphi]$
are holomorphic and bounded on a disc of radius at least $\mathfrak{m}_{\mu}\exp\left(2\pi-\nf 1{\ell}\right)$).
The transform 
\begin{align*}
\cauchein~:~\mathcal{B} & \longto r_{\lambda}\mathcal{B}
\end{align*}
is a well-defined $2\kappa_{\lambda}$-Lipschitz map. In particular
if $\kappa_{\lambda}<\frac{1}{2}$ then it is a contracting self-map
of $r_{\lambda}\mathcal{B}$.
\end{prop}

\begin{proof}
Clearly $\lim_{\lambda\to0}\kappa_{\lambda}=0$. The fact that $\cauchein[\varphi]$
ranges in $r_{\lambda}\mathcal{B}$ comes from~(\ref{eq:estim_CH}),
since $\exp\left(2\pi\norm[f^{+}]{V^{\cap}}\right)\leq\ee^{2\pi}<536$.

Observe next that the assumption made on $\lambda$ guarantees that
the condition $\mathfrak{m}_{\mu}\exp\left(2\pi-\nf 1{\lambda}\right)<\zero[\rho]$
is met, hence
\begin{align*}
\left|\zero[\varphi]\left(H_{f}^{+}\left(\xi\right)\right)\right| & \leq\norm[{\zero[\varphi]}]{\lambda}\left|H_{f}^{+}\left(\xi\right)\right|\leq\norm[{\zero[\varphi]}]{\ell}\left|H_{f}^{+}\left(\xi\right)\right|
\end{align*}
because 
\begin{align*}
\left|H_{f}^{+}\left(\xi\right)\right| & <\mathfrak{m}_{\mu}\exp\left(-\nf 1{\lambda}\right)\exp\left(2\pi\norm[f^{+}]{V^{\cap}}\right)\leq\mathfrak{m}_{\mu}\exp\left(2\pi-\nf 1{\lambda}\right)\leq\mathfrak{m}_{\mu}\exp\left(2\pi-\nf 1{\ell}\right).
\end{align*}
Being given $f_{1},~f_{2}\in\mathcal{B}$ and $z\in\zero$ we thereby
derive the bound:
\begin{align*}
\left|\Lambda_{f_{1}}^{\pm}\left(z\right)-\Lambda_{f_{2}}^{\pm}\left(z\right)\right| & \leq\frac{\left|\sqrt{z}\right|}{2\pi}\norm[{\zero[\varphi]}]{\ell}\int_{\partial V^{+}}\left|\exp\left(2\ii\pi f_{1}^{+}\left(\xi\right)\right)-\exp\left(2\ii\pi f_{2}^{+}\left(\xi\right)\right)\right|\left|\frac{H_{0}\left(\xi\right)}{\sqrt{\xi}\left(\xi-z\right)}\right|\left|\dd{\xi}\right|\\
 & \leq\ee\left|\sqrt{z}\right|\norm[{\zero[\varphi]}]{\ell}\int_{\partial V^{+}}\left|f_{1}^{+}\left(\xi\right)-f_{2}^{+}\left(\xi\right)\right|\left|\frac{H_{0}\left(\xi\right)}{\sqrt{\xi}\left(\xi-z\right)}\right|\left|\dd{\xi}\right|\\
 & \leq\frac{\ee}{2}\norm[{\zero[\varphi]}]{\ell}\norm[f_{1}-f_{2}]{}\times4\mathfrak{m}_{\mu}\lambda^{2},
\end{align*}
the last step coming from Lemma~(\ref{lem:model_integrable}). Hence
\begin{align*}
\left|\cauchein\left(f_{1}\right)\left(z\right)-\cauchein\left(f_{2}\right)\left(z\right)\right| & \leq\left|\Lambda_{f_{1}}^{\pm}\left(z\right)-\Lambda_{f_{2}}^{\pm}\left(z\right)\right|+\left|\Lambda_{f_{1}}^{\pm}\left(0\right)-\Lambda_{f_{2}}^{\pm}\left(0\right)\right|\\
 & \leq\frac{\ee}{2}\norm[{\zero[\varphi]}]{\ell}\norm[f_{1}-f_{2}]{}\times8\mathfrak{m}_{\mu}\lambda^{2}.
\end{align*}
The case of $\infi[\varphi]$ is completely similar.
\end{proof}
\begin{cor}
\label{cor:fixed-point_existence_local_uniqueness}Let $\ell:=\max\left\{ 1,\frac{1}{\mathfrak{t}_{\mu}},\frac{1}{2\pi+\ln\frac{\mathfrak{m}_{\mu}}{\min\left\{ \zero[\rho],\infi[\rho]\right\} }}\right\} $
and take 
\begin{align*}
\lambda & <\min\left\{ \ell,\frac{1}{4\sqrt{\mathfrak{m}_{\mu}\max\left\{ \norm[{\zero[\varphi]}]{\ell},\norm[{\infi[\varphi]}]{\ell}\right\} }}\right\} .
\end{align*}
 The map $\cauchein|_{\mathcal{B}}$ admits a unique fixed-point $f\in\mathcal{B}$,
obtained for instance by considering the $\cauchein$-orbit of $0$.
Moreover
\begin{align*}
\norm[f]{} & \leq\kappa_{\lambda}\exp\left(3365\kappa_{\lambda}\right).
\end{align*}
\end{cor}

\begin{rem}
The bound $\kappa_{\lambda}\exp\left(3365\kappa_{\lambda}\right)\to0$
is marginally sharper as $\lambda\to0$ than $\norm[f]{}\leq r_{\lambda}=536\kappa_{\lambda}$.
\end{rem}

\begin{proof}
Well, this is just Banach's theorem. The bound on the norm of the
fixed-point $f$ comes from the fact that
\begin{align*}
\norm[f]{}=\norm[{\cauchein}\left(f\right)]{} & \leq\kappa_{\lambda}\exp\left(2\pi\norm[f]{}\right)
\end{align*}
(Proposition~\ref{prop:cauchy-heine}) and that $\norm[f]{}\leq r_{\lambda}=536\kappa_{\lambda}$.
\end{proof}

\subsection{\label{subsec:proofs}Proof of the lemmas}

\subsubsection{Proof of Lemma~\ref{lem:usual_model_orbits_size}}

We need to bound $\mathfrak{H}$ on the half-lines $\tau=\ii t\theta$
for $t>0$ and for fixed $\theta\in\sone$ with $\arg\theta=\pm\delta$.
Define
\begin{align*}
M\left(t\right):=\left|\mathfrak{H}\left(\tau\right)\right| & =\exp\left(2\pi\im{\tau+\mu\log\tau}\right)\\
 & =\exp\left(2\pi\left(t\cos\delta+\im{\mu}\ln t+\left(\frac{\pi}{2}\pm\delta\right)\re{\mu}\right)\right).
\end{align*}
An extremum is reached only if 
\begin{align*}
t & =-\frac{\im{\mu}}{\cos\delta}>0.
\end{align*}

\begin{itemize}
\item If $\im{\mu}=0$ the function $M$ increases from $M\left(0\right)$
to $+\infty$ as $t$ goes from $0$ to $+\infty$: 
\begin{align*}
M\left(0\right) & =\exp\left(2\pi\re{\mu}\left(\frac{\pi}{2}\pm\delta\right)\right)>\exp\left(\pi^{2}\re{\mu}-\pi^{2}\left|\re{\mu}\right|\right),
\end{align*}
therefore $M\left(0\right)\geq\exp\left(-2\pi^{2}\left|\re{\mu}\right|\right)=\frac{1}{\mathfrak{m}}$.
\item If $\im{\mu}<0$ the minimum of $M$ is bounded from below by
\begin{align*}
M\left(-\frac{\im{\mu}}{\cos\delta}\right) & \geq\exp\left(-2\pi\im{\mu}\left(\ln\left|\frac{\im{\mu}}{\cos\delta}\right|-1\right)+\pi^{2}\re{\mu}-\pi^{2}\left|\re{\mu}\right|\right)\\
 & \geq\frac{1}{\mathfrak{m}}.
\end{align*}
\item If $\im{\mu}>0$ the function $M$ increases from $0$ to $+\infty$
as $t$ runs along $\rr_{>0}$. The unique $t>0$ such that $M\left(t\right)=\frac{1}{\mathfrak{m}}$
satisfies
\begin{align*}
t\cos\delta+\im{\mu}\ln t & =\im{\mu}\left(\ln\frac{\im{\mu}}{\cos\delta}-1\right)+\pi\left|\re{\mu}\right|-\left(\frac{\pi}{2}\pm\delta\right)\re{\mu}\leq\mathfrak{t}\cos\delta.
\end{align*}
If $t\geq1$ then $t\cos\delta+\im{\mu}\ln t\geq t\cos\delta$, therefore
for all $t>\mathfrak{t}$ we have $M\left(t\right)>M\left(\mathfrak{t}\right)\geq\frac{1}{\mathfrak{m}}$.
\end{itemize}
The case $t<0$ is taken care of similarly.

\subsubsection{Proof of Lemma~\ref{lem:model_Fatou_size_shrink}}

Assume that $z=t\theta$ with $t>0$ and $\theta:=\ii\ee^{\ii\eta}$
for $\eta\in\left[-\frac{\pi}{8},\frac{\pi}{8}\right]$, \emph{i.e}.
$z\in\zero$. Then 
\begin{align*}
\arg\frac{\tau}{-\ii} & =\arg\left(1+t^{2}\ee^{2\ii\eta}\right)-\eta.
\end{align*}
Because $-\frac{\pi}{2}<2\eta<\frac{\pi}{2}$ we have
\begin{align*}
\left|\arg\left(1+t^{2}\ee^{2\ii\eta}\right)\right| & \leq2\eta
\end{align*}
and $\arg\left(1+t^{2}\ee^{2\ii\eta}\right)\to_{+\infty}2\eta$ monotonically.
On the one hand we obtain, with sharp bounds,
\begin{align*}
-\frac{3\pi}{8}\leq-3\eta\leq & \arg\frac{\tau}{-\ii}\leq\eta\leq3\eta\leq\frac{3\pi}{8}.
\end{align*}
On the other hand,
\begin{align}
\left|\tau\right| & =\frac{\left|1+t^{2}\ee^{2\ii\eta}\right|}{\lambda t}\geq\frac{1+\cos\left(2\eta\right)t^{2}}{\lambda t}\geq\frac{1+\frac{1}{\sqrt{2}}t^{2}}{\lambda t}\label{eq:estim_tau_on_ray}
\end{align}
 and basic calculus yields
\begin{align*}
\left(\forall t>0\right) & ~~~~~~\frac{1+\frac{1}{\sqrt{2}}t^{2}}{\lambda t}\geq\frac{3}{\lambda\sqrt{2\sqrt{2}}}>\frac{1}{\lambda}.
\end{align*}

\subsubsection{Proof of Lemma~(\ref{lem:model_integrable})}

For the sake of concision we only deal with the case $\int_{a}\frac{H_{0}\left(\xi\right)}{\sqrt{\xi}\left(\xi-z\right)}\dd{\xi}$
where $z\in U:=\left\{ z\neq0~:~\re z\geq0\right\} $ and $a=\ee^{\ii\frac{5\pi}{8}}\rr_{>0}$.
For $\re z\geq0$ let us define
\begin{align*}
d\left(z\right) & :=\sup\left\{ \frac{\sqrt{\left|\xi\right|}}{\left|\xi-z\right|}~:~\xi\in a\right\} \in\left]0,+\infty\right[.
\end{align*}
We have $\frac{1}{\left|\sqrt{\xi}\left(\xi-z\right)\right|}\leq\frac{d\left(z\right)}{\left|\xi\right|}$
so that writing $\xi=\ee^{\ii\nf{5\pi}8}t$ for $t>0$ yields:
\begin{align*}
\left|\frac{H_{0}\left(\xi\right)\dd{\xi}}{\sqrt{\xi}\left(\xi-z\right)}\right| & \leq\mathfrak{m}_{\mu}d\left(z\right)\times\frac{1}{t}\exp\frac{1+\frac{1}{\sqrt{2}}t^{2}}{-\lambda t}\dd t
\end{align*}
(Corollary~\ref{cor:model_orbits_size} and bound~(\ref{eq:estim_tau_on_ray}))
while
\begin{align*}
\int_{0}^{+\infty}\frac{1}{t}\exp\left(-\frac{1}{\lambda t}-\frac{1}{\lambda\sqrt{2}}t\right)\dd t & \leq\max_{t>0}\left(\frac{1}{t\exp\frac{1}{\lambda t}}\right)\int_{0}^{+\infty}\exp\left(-\frac{1}{\lambda\sqrt{2}}t\right)\dd t=\frac{\lambda}{\ee}\times\lambda\sqrt{2}.
\end{align*}
In order to conclude we need to bound $d\left(z\right)$. First, we
claim that $\left|\xi-z\right|\geq\left|\xi-\ii\left|z\right|\right|$
for all $\xi\in a$ and $\re z\geq0$, so that we may as well assume
that $z\in\ii\rr_{>0}$. Next, elementary calculus provides the bound
$d\left(z\right)=\frac{1}{\sqrt{2\left|z\right|}\sqrt{1-\cos\frac{\pi}{8}}}$,
completing the proof since $\frac{1}{\ee\sqrt{1-\cos\frac{\pi}{8}}}<\frac{3}{2}$.

\section{\label{sec:Globalization-Theorem}Globalization Theorem}

This section is devoted to the proof of the Globalization Theorem.
We prove Items~1.--3. in Section~\ref{subsec:Sectorial_dynamics}
below but begin with establishing Item~4. in Section~\ref{subsec:inversion_module},
which introduces the need material about the action of the involution
$\sigma=\frac{-1}{\id}$.

Every object $O$ involved in the paper comes as a pair of sectorial
objects $O=\left(O^{+},O^{-}\right)$, each component being meromorphic
on the corresponding $V^{\pm~}$. The action of 
\begin{align*}
\sigma~:~\cbar & \longto\cbar\\
z & \longmaps-\frac{1}{z}
\end{align*}
 on such objects is defined as
\begin{align*}
\act O & :=\left(\sigma^{*}O^{-},\sigma^{*}O^{+}\right)
\end{align*}
 where $\sigma^{*}$ is just the usual action on holomorphic objects
by change of coordinate:
\begin{lyxlist}{00.00.0000}
\item [{on~functions}] $\sigma^{*}g:=g\circ\sigma$;
\item [{on~biholomorphisms}] $\sigma^{*}\phi:=\sigma^{\circ-1}\circ\phi\circ\sigma$;
\item [{on~vector~fields}] $\sigma^{*}X:=\DD{\sigma}^{-1}\left(X\circ\sigma\right)$.
\end{lyxlist}
In what follows every object will be \textbf{$\act$-invariant}:
\begin{align*}
\act O_{f} & =O_{\act f},
\end{align*}
 as is the case for example of the sectorial first-integrals (because
$H_{0}\circ\sigma=H_{0}$ we have $\act H_{f}=H_{\act f}$), but for
the Cauchy-Heine transform (Proposition~\ref{prop:cauchy-heine}
and identity~(\ref{eq:sigma_action_on_f})) which is $\act$\textbf{-antivariant}
since $\act\cauchein[\varphi]\left(f\right)=-\cauchein[\varphi]\left(\act f\right)$.

\subsection{\label{subsec:inversion_module}Modulus at $\infty$}

Let $f\in\mathcal{B}$ in the unit ball of $\mathcal{S}$ give the
parabolic realization $\Delta_{f}$ of some $\psi_{f}=\left(\zero[\psi]_{f},\infi[\psi]_{f}\right)$.
According to Proposition~\ref{prop:reduction} we know that $f^{-}-f^{+}$
is a first-integral of $\Delta_{f}$ on the whole $V^{\cap}$, in
particular near $\infty$. Therefore we can apply the converse of
Proposition~\ref{prop:reduction} to $\act f$: there exists some
parabolic germ $\Delta_{\act f}\in\parab$ given by the sectorial
time-1 map of $X_{\act f}=\act X_{f}$ (by uniqueness of the normal
form).

By definition of the Birkhof--Écalle--Voronin modulus it means that
the identities
\begin{align*}
\begin{cases}
H_{f}^{-} & =\psi_{f}\circ H_{f}^{+}\\
H_{\act f}^{-} & =\psi_{\act f}\circ H_{\act f}^{+}
\end{cases}
\end{align*}
hold on $V^{\cap}$. By letting $\sigma$ act on the first equality
we obtain
\begin{align*}
H_{\act f}^{+}=H_{f}^{-}\circ\sigma=\psi_{f}\circ H_{f}^{+}\circ\sigma & =\psi_{f}\circ H_{\act f}^{-}\\
 & =\psi_{f}\circ\psi_{\act f}\circ H_{\act f}^{+}
\end{align*}
 as expected.

\subsection{\label{subsec:Quantitative-bounds}Quantitative bounds for the dynamics
of the model $\Delta_{0}$}

The first step to control precisely the dynamics of $X^{\pm}$ is
to control that of $X_{0}$ and more precisely of $\flow{X_{0}}{\tau}{}$
for $\left|\tau\right|\leq1$. We rely on the following technical
lemma. 
\begin{lem}
\label{lem:flow_control}Let $X$ be a holomorphic vector field on
some pointed disc $\mathcal{D}:=\rho\ww D\backslash\left\{ 0\right\} $
and assume there exists some convenient constant $C>0$ such that
for any $0<r<\rho$ we have:
\begin{align*}
\sup\left|X\left(r\sone\right)\right| & \leq\frac{C}{r}.
\end{align*}
Pick $0<r<\rho$ and suppose moreover that $0<\sqrt{r^{2}-2C}<\sqrt{r^{2}+2C}<\rho$.
Then:
\begin{enumerate}
\item for all $\left|\tau\right|\leq1$ the mapping $\Delta:=\flow X{\tau}{}$
is holomorphic on a neighborhood of $r\sone$;
\item $\Delta\left(r\sone\right)\subset\mathcal{D}$.
\end{enumerate}
\end{lem}

\begin{proof}
This is a simple variational argument applied to $\phi\left(t\right):=\left|z\left(t\right)\right|^{2}$
for $t\in\left[0,\tau\right]$, where $z$ solves $\dot{z}=X\left(z\right)$.
Indeed $\phi=z\overline{z}$ so that:
\begin{align*}
\dot{\phi} & =2\re{X\left(z\right)\overline{z}}
\end{align*}
and
\begin{align*}
\left|\dot{\phi}\right| & \leq2C.
\end{align*}
Integrating the left- and right-hand side with respect to $t\in\left[0,1\right]$
one obtains
\begin{align*}
r^{2}-2C\leq\left|z\left(\tau\right)\right|^{2} & \leq r^{2}+2C.
\end{align*}
\end{proof}
We apply now this result to the model $X_{0}$.
\begin{lem}
\label{lem:estim_X0_i}Consider $\rho\leq\frac{1}{20}$ and assume
that $\lambda<\min\left\{ \frac{1}{2\left|\mu\right|},\frac{\rho^{2}}{6}\right\} $.
\begin{enumerate}
\item For each $0<r<\rho$ and $\left|z\pm\ii\right|=r$ we have
\begin{align*}
\left|X_{0}\left(z\right)\right| & <\frac{\lambda}{r}.
\end{align*}
\item If $r:=2\sqrt{\lambda}$ and $\left|\tau\right|\leq1$ then $\flow{X_{0}}{\tau}{}$
is holomorphic on the annulus $\mathcal{A}_{\pm\ii}:=\pm\ii+\rho\ww D\backslash\overline{r\ww D}$
and $\flow{X_{0}}{\tau}{\left(\mathcal{A}_{\pm\ii}\right)\subset\pm\ii+\rho\ww D}\backslash\left\{ 0\right\} $.
\end{enumerate}
\end{lem}

\begin{example}
For instance for $\rho:=\frac{1}{20}$ we can choose the disk $D_{\pm\ii}$
in Proposition~\ref{prop:X0_dynamics}~2. to be of radius $2\sqrt{\lambda}$
provided $\lambda<\frac{1}{2\thinspace400}$.
\end{example}

\begin{proof}
The proof for an annulus centered at $-\ii$ is identical to that
of $\ii$, so we only deal with that case.
\begin{enumerate}
\item The dominant part of $R$ near $\ii$ is given by $\frac{\lambda}{1+z^{2}}$.
We use in the roughest possible fashion the triangular inequalities: 
\begin{itemize}
\item $\left|z\right|^{2}\leq\left(1+\rho\right)^{2}$;
\item $\sqrt{2}-\rho<\left|1-z^{2}\right|\leq1+\left(1+\rho^{2}\right)$;
\item $\left|1-z^{2}+\lambda\mu z\right|>\left(\sqrt{2}-\rho\right)^{2}-\frac{1+\rho}{2}$,
since $\lambda\left|\mu\right|\left|z\right|<\frac{1+\rho}{2}$;
\item $\left|1+z^{2}\right|\geq r\left(2-\rho\right)$;
\end{itemize}
finally yielding
\begin{align*}
\left|X_{0}\left(z\right)\right| & <\frac{\lambda}{r}.
\end{align*}

\item To apply Lemma~\ref{lem:flow_control} we must find $r$ such that
$0<r^{2}-2\lambda$ and $r^{2}+2\lambda<\rho^{2}$. The choice $r:=2\sqrt{\lambda}$
fulfills both conditions provided $\lambda<\frac{\rho^{2}}{6}$.\\
To conclude this lemma we recall Lemma~\ref{lem:time-1_rational}.
We just proved that no point form $\mathcal{C}$ can ever reach the
pole $\ii$ in time $\tau$, hence $\flow{X_{0}}{\tau}{}$ must be
holomorphic on $\mathcal{C}$.
\end{enumerate}
\end{proof}
We play the same game near the poles 
\begin{align*}
z_{\pm} & =\pm\sqrt{1+\frac{\lambda^{2}\mu^{2}}{4}}+\frac{\lambda\mu}{2}
\end{align*}
Observe that if $\lambda<\frac{1}{2\left|\mu\right|}$ then 
\begin{align}
\left|z_{\pm}\mp1\right| & <\frac{\lambda\left|\mu\right|}{2}+\frac{\lambda^{2}\left|\mu\right|^{2}}{8}<\lambda\frac{9\left|\mu\right|}{16}.\label{eq:estim_z+}
\end{align}
In particular $\frac{1}{2}<\left|z_{\pm}\right|<\frac{3}{2}$.
\begin{lem}
\label{lem:estim_X0_z+}Consider $\rho\leq\frac{1}{20}$ and assume
that $\lambda<\min\left\{ \frac{1}{2\left|\mu\right|},\frac{\rho^{2}}{32}\right\} $.
\begin{enumerate}
\item For each $0<r<\rho$ and $\left|z-z_{\pm}\right|=r$ we have
\begin{align*}
\left|X_{0}\left(z\right)\right| & <\frac{8\lambda}{r}.
\end{align*}
\item If $r:=4\sqrt{\lambda}$ and $\left|\tau\right|\leq1$ then $\flow{X_{0}}{\tau}{}$
is holomorphic on the annulus $\mathcal{A}_{\pm}:=z_{\pm}+\rho\ww D\backslash\overline{r\ww D}$
and $\flow{X_{0}}{\tau}{\left(\mathcal{A}_{\pm}\right)\subset z_{\pm}+\rho\ww D\backslash\left\{ 0\right\} }$.
\end{enumerate}
\end{lem}

\begin{example}
For instance for $\rho:=\frac{1}{20}$ we can choose the disk $D_{z_{\pm}}$
in Proposition~\ref{prop:X0_dynamics}~2. to be of radius $4\sqrt{\lambda}$
provided $\lambda<\frac{1}{12\thinspace800}$.
\end{example}

\begin{proof}
Again we only deal with the neighborhood of $z_{+}$, and obtain the
expected bound using $\left|z-z_{+}\right|=r$:
\begin{itemize}
\item $\left|z^{2}\right|\leq\left(\frac{3}{2}+\rho\right)^{2}$;
\item $\left|1-z^{2}\right|\leq1+\left(\frac{3}{2}+\rho\right)^{2}$;
\item (from~(\ref{eq:estim_z+})) $\left|z\pm\ii\right|\geq\sqrt{2}-\rho-\frac{9}{32}$;
\item $\left|1+\lambda\mu z-z^{2}\right|=r\left|2z_{+}+r\right|\geq r\left(1-\rho\right)$.
\end{itemize}
\end{proof}

\subsection{\label{subsec:Sectorial_dynamics}Dynamics of $X^{\pm}$}

Here and all the following we take $f\in\mathcal{B}$ in the unit
ball of $\mathcal{S}$ and suppose
\begin{align*}
0<\lambda & <\min\left\{ \frac{1}{2\left|\mu\right|},\frac{1}{12\thinspace800}\right\} .
\end{align*}

\begin{lem}
Each vector field $X^{-}$ and $X^{+}$ admits a unique pole in each
disc $\pm\ii+3\sqrt{\lambda}\ww D$, which is simple.
\end{lem}

\begin{proof}
For the sake of clarity we chose a sector $V^{+}$ and let $f$ stand
for $f^{+}$, the case of $f^{-}$ on $V^{-}$ being completely similar.
Let $\rho:=9\sqrt{\lambda}\leq\frac{1}{20}$ so that $0<\lambda<\frac{\rho^{2}}{6}$
and we can apply Lemma~\ref{lem:estim_X0_i}: for any $0<r<\rho$
and $\left|z\pm\ii\right|=r$ we have 
\begin{align*}
\frac{1}{\left|X_{0}\left(z\right)\right|} & >\frac{r}{\lambda}.
\end{align*}
Besides Cauchy's formula applied to $f$ on the circle $\pm\ii+\rho\sone$
(which is included in $V^{\cap}$) yields for $z\in\pm\ii+r\sone$
\begin{align*}
\left|f'\left(z\right)\right| & \leq\frac{1}{2\pi}\oint_{\ii+r\ww D}\frac{\left|f\left(z\right)\right|}{\left|z-\ii\right|^{2}}\left|\dd z\right|\leq\frac{\rho}{\left(\rho-r\right)^{2}}.
\end{align*}
 Whenever $\frac{\rho}{\left(\rho-r\right)^{2}}<\frac{r}{\lambda}$
we can apply Rouché's theorem to the holomorphic functions $\frac{1}{X_{0}\cdot\id}$
and $\frac{1}{X_{0}\cdot\id}+f'$ on the disc $\pm\ii+r\ww D$, so
that $\frac{1}{X_{0}\cdot\id}+f'$ has a single simple zero in the
disc. In other words $\frac{1}{1+X_{0}\cdot f}X_{0}$ has a single
(simple) pole in the disc.

The condition $\frac{\rho}{\left(\rho-r\right)^{2}}<\frac{r}{\lambda}$
is equivalent to $\phi\left(r\right)>\lambda$ where $\phi~:~r\mapsto\frac{r}{\rho}\left(r-\rho\right)^{2}$
vanishes at $0$ and $\rho$. The function reaches its maximum $\frac{4\rho^{2}}{27}>\lambda$
at $\frac{\rho}{3}=:r$.
\end{proof}
A similar phenomenon happens near $z_{\pm}$.
\begin{lem}
The vector field $X^{\pm}$ admits a unique pole $p_{\pm}$ in the
disc $\pm1+5\sqrt{\lambda}\ww D$, which is simple.
\end{lem}

\begin{proof}
This time we take $\rho:=15\sqrt{\lambda}$ in Lemma~\ref{lem:estim_X0_z+}
and apply Cauchy's formula to the circle $z_{\pm}+r\sone$ where $0<r<\rho$
so that, for every $\left|z-z_{\pm}\right|=r$:
\begin{align*}
\left|f'\left(z\right)\right| & \leq\frac{\rho}{\left(\rho-r\right)^{2}}\\
\frac{r}{8\lambda} & \leq\frac{1}{\left|X_{0}\left(z\right)\right|}.
\end{align*}
The condition $8\lambda<\phi\left(r\right)$ is met if $r:=5\sqrt{\lambda}$.
\end{proof}
This proposition ends the proof of Items~1.--3. of the Globalization
Theorem.
\begin{prop}
\label{prop:sectorial_dynamics}Let us denote by $p_{\ii}^{\pm}$,
$p_{-\ii}^{\pm}$ and $z_{\pm}$ the simple poles of $X^{\pm}$ provided
by the previous lemmas. The following properties hold.
\begin{enumerate}
\item $X^{\pm}$ does not have any other pole in $V^{\pm}$.
\item $p_{\ii}^{+}=p_{\ii}^{-}=:p_{\ii}$ and $p_{-\ii}^{+}=p_{-\ii}^{-}:=p_{-\ii}$.
Moreover $\sigma\left(p_{\ii}\right)=p_{-\ii}$ and $\sigma\left(p_{+}\right)=p_{-}$.
\item The vector field $X^{\pm}$ is locally conjugate to $\frac{1}{w}\pp w$
as in Lemma~\ref{lem:polar_foliation} on a disc centered at $p$
and containing the two attached ramification points $\left\{ z_{p},w_{p}\right\} $,
hence conjugating the monodromies of the time-1 maps.
\end{enumerate}
\end{prop}

\begin{proof}
~
\begin{enumerate}
\item Recalling Proposition~\ref{prop:flow_change_variable}~3. the sectorial
normalization $\Psi^{\pm}$ between $X_{0}$ and $X^{\pm}$ is holomorphic
on $\mathcal{D}_{\lambda}$. If $X^{\pm}$ were to admit a pole outside
$\Psi^{\pm}\left(\mathcal{D}_{\lambda}\right)$ it would show in $X_{0}$.
Since $\Psi^{\pm}\left(\partial\mathcal{D}_{\lambda}\right)$ lies
within the union of the discs given by the two previous lemmas, this
cannot be the case. Indeed, according to Lemma~\ref{lem:estim_X0_i}
the circle $\pm\ii+2\sqrt{\lambda}\sone$ is mapped by $\Psi^{\pm}$
into a disc centered at $\pm\ii$ of radius $\rho:=3\sqrt{\lambda}>r$.
The same goes accordingly near $z_{\pm}$. 
\item Recalling Lemma~\ref{lem:time-1_rational}, which also holds locally
for meromorphic vector fields, a pole $p$ of $X^{+}$ induces two
ramification points $\left\{ z_{p},w_{p}\right\} $ in $\Delta$,
which are mapped by $\Delta$ to $p$. Since $\Delta$ is also the
time-1 map of $X^{-}$, the points $z_{p}$ and $w_{p}$ also belong
to a stable manifold of $X^{-}$ and are sent to a pole of $X^{-}$
by its time-1 flow. Hence $p$ is also a pole of $X^{-}$. \\
For the very same reason that a pole of $X_{f}=\left(X^{+},X^{-}\right)$
is a dynamical feature the poles of $X_{\act f}$ and that of $\act X_{f}$
coincide. But these vector fields have only poles near fixed-points
$\left\{ \pm\ii,\pm1\right\} $ of $\sigma$, hence must correspond
by $\sigma$.
\item Let $p$ be a pole of $X^{\pm}$. For $\psi$ to conjugate $\frac{1}{z-p}\pp z$
to $X^{\pm}\left(z\right)=\frac{1}{z-p}R\left(z\right)$ it needs
to solve
\begin{align*}
X^{\pm}\cdot\psi & =\frac{1}{\psi-p}
\end{align*}
or, in other words
\begin{align*}
\psi & \left(z\right)=p+2\sqrt{\int_{p}^{z}\frac{\left(x-p\right)\dd x}{R\left(x\right)}}.
\end{align*}
For $R\left(p\right)\notin\left\{ 0,\infty\right\} $, this formula
defines a holomorphic function on a simply-connected domain $U$ as
long as $R$ is holomorphic on $U$. Of course $U$ can be chosen
to avoid the zeroes of $R$, therefore $X^{\pm}$ is conjugate to
the polar model $\frac{1}{z-p}\pp z$ on a domain encompassing the
ramification points of the model time-1 map. 
\end{enumerate}
\end{proof}

\section{Corollaries}

\subsection{\label{subsec:Parabolic}Parabolic Renormalization }

Let $\text{Synth}_{\lambda}~:~\varphi\mapsto\delta$ be the synthesis
map built in Section~\ref{sec:Main_th}, that is
\begin{align*}
\ev\left(\id\exp\delta\right) & =\left(\id\exp\left(4\mu\pi^{2}+\zero[\varphi]\right),\id\exp\infi[\varphi]\right)
\end{align*}
 and consider for given $\zero[\varphi]\in\holb[\cc,0]$ the iteration
$\delta_{0}:=0$ and $\delta_{n+1}:=\text{Synth}_{\lambda}\left(\zero[\varphi],\delta_{n}\right)$.
We wish to prove that the iteration is well-defined (that is, one
can chose $\lambda$ independently on $n$) and that it converges
towards a unique fixed-point. For this we exploit the nice dynamical
feature of the spherical normal forms, providing us with a uniform
lower bound on the radius of convergence of $\Delta$ and allowing
us to control $\Delta'$. 

Let us be quantitative, as in Section~\ref{subsec:Quantitative-bounds}
but around $0$ this time.
\begin{prop}
\label{prop:synthesis_bounds}Fix some $0<\lambda<\frac{1}{2\left|\mu\right|}$.
\begin{enumerate}
\item For all $0<\left|z\right|\leq\frac{1}{2}$ we have 
\begin{align*}
 & \frac{2\lambda\left|z\right|^{2}}{5}\leq\left|X_{0}\left(z\right)\right|\leq\frac{10\lambda\left|z\right|^{2}}{3}.
\end{align*}
\item Let $f=\left(f^{+},f^{-}\right)\in\mathcal{S}$. 
\begin{align*}
\left(\forall\left|z\right|\leq\frac{1}{4}\right)~~~~\left|X_{0}\cdot f\right|\left(z\right)\leq4\lambda\norm[f]{}.
\end{align*}
\item Let $\Delta:=\id\exp\delta$ for $\delta:=\text{Synth}\left(\zero[\varphi],\infi[\varphi]\right)$
be some spherical normal form and assume $\lambda<\frac{1}{8}.$ 
\begin{enumerate}
\item ~The sectorial vector fields $X^{\pm}=\frac{1}{1+X_{0}\cdot f}X_{0}$
have magnitude
\begin{align*}
\left(\forall\left|z\right|\leq\frac{1}{4}~,~\pm\re z\geq0\right) & ~~\left|X^{\pm}\left(z\right)\right|\leq\frac{20\lambda\left|z\right|^{2}}{3}.
\end{align*}
\item $\Delta$ is biholomorphic on the disc $\frac{3}{16}\ww D$ and 
\begin{align*}
\Delta\left(\frac{3}{16}\ww D\right) & \subset\frac{1}{4}\ww D.
\end{align*}
\item ~
\begin{align*}
\norm[\Delta']{\frac{3}{16}\ww D} & \leq45\\
\norm[\delta]{\frac{3}{16}\ww D} & <11
\end{align*}
\end{enumerate}
\end{enumerate}
\end{prop}

\begin{proof}
~
\begin{enumerate}
\item This is nothing but the plain triangular inequalities, taking into
account that $\left|1+\lambda\mu z-z^{2}\right|\geq1-r^{2}-\frac{r}{2}\geq1-r$
since $\frac{1}{2}\geq r$.
\item For $\re z\geq0$ we apply Cauchy formula to the pacman $P:=\partial\left(V^{+}\cap\frac{1}{2}\overline{\ww D}\right)$
and drop the index $\pm$ altogether. The length of $P$ is bounded
by $\pi+1$ and $\left|x-z\right|\geq\left|z\right|\sin\frac{\pi}{8}$
for $x\in P$ and $\left|z\right|\leq\frac{1}{4}$ so that:
\begin{align*}
\left(\forall\left|z\right|\leq\frac{1}{4}~,~\re z\geq0\right) & ~~\left|f'\left(z\right)\right|\leq\frac{\left(\pi+1\right)\norm[f]{}}{\pi\sin^{2}\frac{\pi}{8}\left|z\right|^{2}}\leq\frac{2\norm[f]{}}{\left|z\right|^{2}}.
\end{align*}
As a consequence
\begin{align*}
\left|X_{0}\cdot f\right|\left(z\right) & =\lambda\left|\frac{1-z^{2}}{\left(1+z^{2}\right)\left(1+\lambda\mu z-z^{2}\right)}\right|\left|z^{2}f'\left(z\right)\right|\leq4\lambda\norm[f]{}.
\end{align*}
The argument is identical for $\re z\leq0$ on $V^{-}$.
\item Here $f\in\mathcal{B}$. 
\begin{enumerate}
\item On the disc $\frac{1}{4}\ww D$ we have $\left|X_{0}\cdot f^{\pm}\right|\leq\frac{1}{2}$
so that $\frac{1}{2}\leq\left|1+X_{0}\cdot f^{\pm}\right|\leq\frac{3}{2}$:
we can bound the magnitude of the sectorial vector field $X^{\pm}=\frac{1}{1+X_{0}\cdot f}X_{0}$
by
\begin{align*}
\left(\forall\left|z\right|\leq\frac{1}{4}~:~\pm\re z\geq0\right)~~\frac{4\lambda\left|z\right|^{2}}{15}\leq & \text{\ensuremath{\left|X^{\pm}\left(z\right)\right|\leq}}\frac{20\lambda\left|z\right|^{2}}{3}.
\end{align*}
\item Invoking the same variational argument as in the proof of Lemma~\ref{lem:flow_control},
if the \emph{a priori} bound
\begin{align*}
\frac{r}{1-\nf{20\lambda r}3} & \leq\frac{1}{4}
\end{align*}
holds then the time-1 flow of $z\in r\sone$ along $X^{\pm}$ is well-defined
and lands in $\frac{1}{4}\ww D$. But the condition is fulfilled for
$r:=\frac{3}{16}$ since $\frac{20\lambda r}{3}=\frac{5\lambda}{4}<\frac{5}{32}$. 
\item Because $\Delta$ is a symmetry of $X^{\pm}$ we have $X^{\pm}\cdot\Delta=X^{\pm}\circ\Delta$,
that is:
\begin{align*}
\left|\Delta'\right| & =\frac{\left|X\circ\Delta\right|}{\left|X\right|}.
\end{align*}
For $\left|z\right|=\frac{3}{16}$ we obtain
\begin{align*}
\left|\Delta'\left(z\right)\right| & \leq\frac{\norm[X^{\pm}]{\frac{1}{4}\ww D}}{\left|X^{\pm}\left(z\right)\right|}\leq\frac{\frac{20\lambda}{3}\left(\frac{1}{4}\right)^{2}}{\frac{4\lambda}{15}\left(\frac{3}{16}\right)^{2}}\leq45.
\end{align*}
Continuing we find 
\begin{align*}
\left|\delta\left(z\right)\right| & =\left|\log\frac{\Delta\left(z\right)}{z}\right|\leq2\pi+\ln\left|\frac{\Delta\left(z\right)}{z}\right|\leq2\pi+\ln\norm[\Delta']{\frac{3}{16}\ww D}<11.
\end{align*}
\end{enumerate}
\end{enumerate}
\end{proof}
Consider now $\Delta_{1}$ and $\Delta_{2}$ in normal form, being
the respective time-1 map of $X_{j}^{\pm}:=\frac{1}{1+X_{0}\cdot f_{j}^{\pm}}X_{0}$
with $f_{j}\in\mathcal{B}$ for given $0<\lambda<\min\left\{ \frac{1}{2\left|\mu\right|},\frac{1}{8}\right\} $
as in the Proposition. Being given $O=\left(O^{-},O^{+}\right)$ we
introduce the slight abuse of notation:
\begin{align*}
\norm[O]{r\ww D} & :=\max\left\{ \sup_{\left|z\right|\leq r~,~\pm\re z\geq0}\left|O^{\pm}\left(z\right)\right|\right\} 
\end{align*}
and choose
\begin{align*}
0<\rho\leq\frac{3}{16}~,~ & r:=\frac{1}{4}.
\end{align*}
For the sake of readibility we drop the superscripts $\pm$ altogether.

If $z_{1}$, $z_{2}$ is the respective trajectory of $X_{1}$, $X_{2}$
emanating from some common $z\in\rho\overline{\ww D}$, then $\phi:=\left|z_{1}-z_{2}\right|^{2}$
solves
\begin{align*}
\dot{\phi} & =2\re{\left(X_{1}\left(z_{1}\right)-X_{2}\left(z_{2}\right)\right)\overline{z_{1}-z_{2}}}~,~\phi\left(0\right)=0.
\end{align*}
Since $z_{1}\left(t\right),~z_{2}\left(t\right)\in r\ww D$ for all
$t\in\left[0,1\right]$ we have
\begin{align*}
\left|X_{1}\left(z_{1}\right)-X_{2}\left(z_{2}\right)\right| & \leq\norm[X_{1}-X_{2}]{r\ww D}+\norm[X_{2}']{r\ww D}\sqrt{\phi}.
\end{align*}
The Proposition yields 
\begin{align*}
\norm[X_{1}-X_{2}]{r\ww D} & =\sup_{r\ww D}\frac{\left|X_{0}\cdot\left(f_{1}-f_{2}\right)\right|\left|X_{0}\right|}{\left|1+X_{0}\cdot f_{1}\right|\left|1+X_{0}\cdot f_{2}\right|}\leq\frac{10\lambda^{2}}{3}\norm[f_{1}-f_{2}]{}=:a
\end{align*}
and 
\begin{align*}
\norm[X_{2}']{r\ww D} & =\norm[\frac{\left(X_{0}\cdot\id\right)'}{1+X_{0}\cdot f_{2}}-\frac{X_{0}\cdot^{2}f_{2}}{\left(1+X_{0}\cdot f_{2}\right)^{2}}]{r\ww D}\\
 & \leq2\norm[X_{0}']{r\ww D}+4\norm[X_{0}\cdot^{2}f_{2}]{\frac{1}{4}\ww D}\\
 & \leq2\norm[X_{0}']{r\ww D}+16\lambda\norm[X_{0}\cdot f_{2}]{\frac{1}{2}\ww D}\\
 & \leq\frac{5}{2}\lambda+64\lambda^{2}\norm[f_{2}]{}\leq2.
\end{align*}

From a direct integration of 
\begin{align*}
\left(\forall\left|t\right|\leq1\right)~~\left|\frac{\dot{\phi}}{2\sqrt{\phi}\left(a+2\sqrt{\phi}\right)}\right| & \leq1
\end{align*}
we finally deduce the following estimate.
\begin{lem}
We have
\begin{align*}
\norm[\Delta_{1}-\Delta_{2}]{\rho\ww D} & \leq a\frac{\ee^{2}-1}{2}\leq11\lambda^{2}\norm[f_{1}-f_{2}]{}.
\end{align*}
\end{lem}

Assume now that $\ev\left(\Delta_{j}\right)=\left(\id\exp\left(4\pi^{2}\mu+\zero[\varphi_{j}]\right),\id\exp\infi[\varphi_{j}]\right)$
and pick $\rho<\min\left\{ \infi[\rho_{1}],\infi[\rho_{2}]\right\} $
less than the least radius of convergence of $\infi[\varphi_{1}],\infi[\varphi_{2}]$,
so that $\infi[\varphi_{j}]\in\holb[\rho\ww D]$, and consider
\begin{align}
\ell & :=\min\left\{ 1,\mathfrak{t}_{\mu},\frac{1}{2\pi+\ln\frac{\mathfrak{m}_{\mu}}{\min\left\{ \zero[\rho],\nf{\rho}2\right\} }}\right\} \label{eq:parabo_bound}
\end{align}
 similarly as in Proposition~\ref{prop:CH_contractant}. Invoking~\ref{eq:estim_CH}
we derive
\begin{align*}
\norm[f_{1}-f_{2}]{} & \leq4\mathfrak{m}_{\mu}\lambda^{2}\times\left(\norm[{\infi[\varphi_{1}]-\infi[\varphi_{2}]}]{\ell}+\ee\left(\norm[{\zero[\varphi]}]{\ell}+\norm[{\infi[\varphi_{2}]}]{\ell}\right)\norm[f_{1}-f_{2}]{}\right),
\end{align*}
that is
\begin{align*}
\norm[f_{1}-f_{2}]{} & \leq8\mathfrak{m}_{\mu}\lambda^{2}\norm[{\infi[\varphi_{1}]-\infi[\varphi_{2}]}]{\ell}
\end{align*}
provided $\lambda>0$ be small enough to ensure $4\ee\mathfrak{m}_{\mu}\lambda^{2}\left(\norm[{\zero[\varphi]}]{\ell}+\norm[{\infi[\varphi_{2}]}]{\ell}\right)\leq\frac{1}{2}$,
say. By assumption we have $\mathfrak{m}_{\mu}\exp\left(2\pi-\frac{1}{\lambda}\right)\leq\frac{\rho}{2}$
so that Cauchy's formula applied on $\rho\sone$ yields
\begin{align*}
\norm[\varphi]{\ell} & \leq\frac{\rho}{\left(\rho-\mathfrak{m}_{\mu}\exp\left(2\pi-\frac{1}{\ell}\right)\right)^{2}}\norm[\varphi]{\rho\ww D}\leq\frac{4}{\rho}\norm[\varphi]{\rho\ww D}
\end{align*}
from which we deduce
\begin{align*}
\norm[\Delta_{1}-\Delta_{2}]{\rho\ww D}\leq & 171\mathfrak{m}_{\mu}\lambda^{4}\norm[{\infi[\varphi_{1}]-\infi[\varphi_{2}]}]{\rho\ww D}.
\end{align*}
We write $\Delta_{j}=\id\exp\delta_{j}$ with $\delta_{j}\left(0\right)=0$
which, thanks to Proposition~\ref{prop:synthesis_bounds}~(3)(c),
satisfies $\norm[\delta_{j}]{\rho\ww D}\leq11$. Because we iterate
$\text{Synth}_{\lambda}$ we lose no generality in assuming $\norm[{\infi[\varphi_{j}]}]{\rho\ww D}\leq11$
as well. Supposing therefore 
\begin{align*}
\frac{171}{\rho}\ee^{11}\mathfrak{m}_{\mu}\lambda^{4} & \leq\frac{1}{44}
\end{align*}
we deduce for $\left|z\right|=\rho$
\begin{align*}
\left|\frac{\Delta_{1}-\Delta_{2}}{\Delta_{2}}\right|\left(z\right)\leq\frac{\ee^{11}}{\rho}\norm[\Delta_{1}-\Delta_{2}]{\rho\ww D} & \leq\frac{171}{\rho}\ee^{11}\mathfrak{m}_{\mu}\lambda^{4}\norm[{\infi[\varphi_{1}]-\infi[\varphi_{2}]}]{\rho\ww D}\leq\frac{11+11}{44}=\frac{1}{2}.
\end{align*}
As a matter of consequence we have
\begin{align*}
\norm[\delta_{1}-\delta_{2}]{\rho\ww D}=\norm[\log\left(1+\frac{\Delta_{1}-\Delta_{2}}{\Delta_{2}}\right)]{\rho\sone} & \leq2\norm[\frac{\Delta_{1}-\Delta_{2}}{\Delta_{2}}]{\rho\sone}\\
 & \leq\frac{1}{22}\norm[{\infi[\varphi_{1}]-\infi[\varphi_{2}]}]{\rho\ww D}.
\end{align*}
Now we can chose the right constants in advance.
\begin{cor}
\label{cor:fixed-point_parab}Pick $\zero[\varphi]\in\holf[\cc,0]$
and $\rho:=\frac{3}{16}$; let $\mathcal{B}_{11}$ be the ball $\left\{ \varphi\in\holb[\rho\ww D]\,:~\norm[\varphi]{\rho\ww D}\leq11\right\} $.
Define 
\begin{align*}
\widehat{\ell}=\widehat{\ell}\left(\zero[\varphi]\right):=\min\left\{ 1,\mathfrak{t}_{\mu},\frac{10^{-2}}{\sqrt[4]{\mathfrak{m}_{\mu}}},\frac{1}{2\pi+\ln\frac{\mathfrak{m}_{\mu}}{\min\left\{ \zero[\rho],\nf 3{32}\right\} }}\right\}  & .
\end{align*}
For any 
\begin{align*}
0<\lambda & \leq\widehat{\lambda}:=\min\left\{ \ell,\frac{1}{8\ee\sqrt{\mathfrak{m}_{\mu}}\sqrt{\norm[{\zero[\varphi]}]{\widehat{\ell}}+9}}\right\} 
\end{align*}
the partial map $\text{Synth}_{\lambda}\left(\zero[\varphi],\bullet\right)$
\begin{align*}
\mathcal{B}_{11} & \longto\mathcal{B}_{11}\\
\infi[\varphi] & \longmapsto\delta=\text{Synth}_{\lambda}\left(\zero[\varphi],\infi[\varphi]\right)
\end{align*}
is a well-defined $\frac{1}{22}$-lipschitz map. Its unique fixed-point
$\delta_{*}$ provides the sought fixed-point $\Delta_{*}:=\id\exp\delta_{*}$
of the parabolic renormalization.
\end{cor}

\subsection{\label{subsec:Real_case}Real Synthesis}

Let $\mu\in\rr$ and $\psi=\left(\zero[\psi],\infi[\psi]\right)$
be given such that condition $\left(\square\right)$ holds:
\begin{align*}
\left(\forall h\in\neigh\right)~~~~\overline{\zero[\psi]\left(\overline{h}\right)} & =\frac{\ee^{4\pi^{2}\mu}}{\infi[\psi]\left(\frac{1}{h}\right)}.\tag{\ensuremath{\square}}
\end{align*}
Expressed for the data $\varphi=\left(\zero[\varphi],\infi[\varphi]\right)$
it becomes
\begin{align*}
\left(\forall h\in\neigh\right)~~~~\overline{\zero[\varphi]\left(\overline{h}\right)} & =-\infi[\varphi]\left(\frac{1}{h}\right).\tag{\ensuremath{\widetilde{\square}}}
\end{align*}
Recall that $\Delta=\flow X1{}$ where $X=\frac{1}{1+X_{0}\cdot f}X_{0}$
and $f$ is obtained as the limit of the iteration $\left(f_{n}\right)_{n}$
defined by:
\begin{align*}
f_{0} & :=\left(0,0\right)\\
f_{n+1} & :=\cauchein[\varphi]\left(f_{n}\right)=\Lambda_{f_{n}}-\Lambda_{f_{n}}\left(0\right)
\end{align*}
as in Definition~\ref{def:Cauchy-Heine}. We wish to prove that $f$
is a real function by induction on $n$. This is clearly true for
$n:=0$ and it is sufficient to prove that $\Lambda_{f}$ is real.

Assume that $\overline{f_{n}\left(z\right)}=f_{n}\left(\overline{z}\right)$
for every $z\in V$ (this identity is well-defined because $V$ is
invariant by the complex conjugation). Because $\overline{a}=b$ we
only need to check that
\begin{align*}
\left(\forall\xi\in b\right)~~~\overline{\zero[\varphi]\left(H_{n}\left(\overline{\xi}\right)\right)} & =-\infi[\varphi]\left(H_{n}\left(\xi\right)\right),
\end{align*}
where we have set 
\begin{align*}
H_{n}\left(\xi\right):=H_{f_{n}}^{+}\left(\xi\right) & =H_{0}\left(\xi\right)\exp\left(2\ii\pi f_{n}^{+}\left(\xi\right)\right),
\end{align*}
which in turn holds thanks to~$\left(\widetilde{\square}\right)$
whenever $\overline{H_{n}\left(\overline{\xi}\right)}=\frac{1}{H_{n}\left(\xi\right)}$.
By the induction hypothesis we can assert 
\begin{align*}
\overline{H_{n}\left(\overline{\xi}\right)} & =\overline{H_{0}\left(\overline{\xi}\right)}\exp\left(-2\ii\pi f_{n}^{+}\left(\xi\right)\right)=\frac{1}{H_{n}\left(\xi\right)}
\end{align*}
since, according to~(\ref{eq:model_first-int}):
\begin{align*}
\overline{H_{0}\left(\overline{\xi}\right)} & =\exp\left(2\ii\pi\frac{1-\xi^{2}}{\lambda\xi}+2\ii\pi\mu\log\frac{1-\xi^{2}}{\lambda\xi}\right)=\frac{1}{H_{0}\left(\xi\right)}.
\end{align*}
This completes the proof of the Real Synthesis Corollary.

\bibliographystyle{alpha}
\bibliography{biblio}

\end{document}